\documentclass[a4paper]{amsart} 
\usepackage[T1]{fontenc}
\usepackage[margin=1in]{geometry}
\usepackage[american]{babel}
\usepackage[utf8x]{inputenc}
\usepackage[sc]{mathpazo}
\usepackage{amsmath,amssymb,amsfonts,mathrsfs}
\usepackage{hyperref}
\usepackage{amsthm}
\usepackage{graphicx}
\usepackage{dutchcal}
\usepackage{varioref}
\usepackage{datetime}
\usepackage{cancel}
\usepackage{mathtools}

\usepackage{booktabs}

\usepackage{tikz-cd}
\usepackage{csquotes}
\usepackage{epigraph}

\usepackage{interval}
\intervalconfig{soft open fences}
\usepackage{dsfont}
\usepackage{microtype}

\usepackage{faktor}
\usepackage{xcolor}
\usepackage{prerex}

\usepackage{pgfplots}
\pgfplotsset{compat=newest}
\usepgfplotslibrary{groupplots}
\usepgfplotslibrary{polar}
\usepgfplotslibrary{smithchart}
\usepgfplotslibrary{statistics}
\usepgfplotslibrary{dateplot}
\usepgfplotslibrary{ternary}
\usepgfplotslibrary{fillbetween}

\DeclareMathAlphabet{\mathdutchcal}{U}{dutchcal}{m}{n}
\SetMathAlphabet{\mathdutchcal}{bold}{U}{dutchcal}{b}{n}
\DeclareMathAlphabet{\mathdutchbcal}{U}{dutchcal}{b}{n} 
\newtheorem{theorem}{Theorem}[section]

\newtheorem{lemma}[theorem]{Lemma}
\newtheorem{corollary}[theorem]{Corollary}

\theoremstyle{definition}
\newtheorem{definition}[theorem]{Definition}
\newtheorem{example}[theorem]{Example}

\theoremstyle{remark}
\newtheorem{remark}[theorem]{Remark}
\theoremstyle{conjecture}
\newtheorem{conjecture}[theorem]{Conjecture}
\theoremstyle{question}

\numberwithin{equation}{section}

\newcommand{\C}{\mathbb{C}}

\newcommand{\N}{\mathbb{N}}
\newcommand{\Q}{\mathbb{Q}}
\newcommand{\R}{\mathbb{R}}
\newcommand{\Z}{\mathbb{Z}}

\newcommand{\T}{\mathbb{T}}

\DeclarePairedDelimiter\abs{\lvert}{\rvert}
\DeclarePairedDelimiter\norm{\lVert}{\rVert}

\renewcommand{\epsilon}{\ensuremath\varepsilon}

\renewcommand{\phi}{\ensuremath{\varphi}}

\newcommand{\eps}{\epsilon}

\newcommand{\anyfield}{\mathbb{F}}
\NewDocumentCommand\field{e{_}}{\IfValueF{#1}{\anyfield}\IfNoValueF{#1}{\Z_{#1}}}

\newcommand{\Reals}{\R}
\newcommand{\Nats}{\N}
\newcommand{\into}{\hookrightarrow}

\newcommand{\Laplace}{\Delta}
\newcommand{\E}{\mathbb{E}}

\newcommand{\define}[1]{
    \begin{definition}
        #1 
    \end{definition}
}
\newcommand{\rmk}[1]{
    \begin{remark}
        #1 
    \end{remark}
}

\newcommand{\Prob}{\mathbb{P}}

\newcommand{\muGP}{\mu_{\Ham}}
\newcommand{\muAut}{\mu_{\Aut}}
\newcommand{\muGPH}{\mu_{\mathcal{H}}}
\newcommand{\Hil}{\mathbb{H}}
\newcommand{\ham}{\mathcal{P}_{\id}\Ham}
\newcommand{\reg}{\mathfrak{r}}
\newcommand{\ldd}{\text{law-defining datum}}
\newcommand{\ldds}{\text{law-defining data}}
\newcommand{\weakto}{\rightharpoonup}

\DeclareMathOperator{\id}{id}

\DeclareMathOperator{\vol}{vol}
\DeclareMathOperator{\supp}{supp}

\DeclareMathOperator{\Hof}{Hofer}
\DeclareMathOperator{\Ham}{Ham}
\DeclareMathOperator{\Hameo}{Hameo}
\DeclareMathOperator{\Homeo}{Homeo}
\DeclareMathOperator{\Symp}{Symp}

\DeclareMathOperator{\Aut}{Aut}

\DeclareMathOperator{\osc}{osc}

\DeclareMathOperator{\Cov}{Cov}
\DeclareMathOperator{\Var}{Var}
\DeclareMathOperator{\GL}{GL}

\definecolor{ceruleanblue}{rgb}{0.16, 0.32, 0.75}

\tikzset{    
    mypoint/.style={
        circle,
        draw,
        inner sep=.3mm
        },  
    whitepoint/.style={
        fill=white, 
        mypoint
        },  
    blackpoint/.style={
        fill=black, 
        mypoint
        },  
    textnode/.style={
        text height=2.5ex, 
        text depth=1ex
        },  
    }

\begin{document}

\title[Random Hamiltonians I]{Random Hamiltonians I: Probability measures and random walks on the Hamiltonian diffeomorphism group}
\author{Adrian Dawid}
\address{Department of Pure Mathematics and Mathematical Statistics, University of Cambridge}
\email{apd55@cam.ac.uk}

\maketitle

\begin{abstract}
    We construct a family of probability measures on the group of Hamiltonian diffeomorphisms of a closed symplectic manifold $(M,\omega)$.
    We show that these measures are Borel measures with respect to the topology induced by the Hofer metric.
    Further, we show that these measures turn any Hofer-Lipschitz function into a random variable with finite expectation.
    These measures have (for suitable choices of parameters) several desirable properties, 
    such as full support on $\Ham(M,\omega)$, explicit estimates of the measure of Hofer-balls,
    and certain controls under the action of the group.
    We also define a family of probability measures on the space of autonomous Hamiltonian diffeomorphisms.
    These measures have similar properties and give rise to a random walk on the group $\Ham(M,\omega)$.
    Finally, we show that under certain limits this construction gives rise to probability measures on the space of Hamiltonian
    homeomorphisms and on the metric completion of $\Ham(M,\omega)$ with respect to the Hofer metric and the spectral metric.
\end{abstract}

\section{Introduction}
We propose a notion of \textit{random Hamiltonian diffeomorphisms} on a closed symplectic manifold $(M,\omega)$
by constructing a family of probability measures on the group of Hamiltonian diffeomorphisms $\Ham(M,\omega)$.
These measures are constructed using Gaussian processes
and inherit many useful properties of Gaussian measures.

While in general there is no model of random diffeomorphisms on a manifold that is seen as canonical,
Gaussian processes are a canonical model for random functions on manifolds.
They appear prominently both in theory~\cite{adler-taylor-2007,castillo-et-al-2014,gine-nickl-2016} and in real-world applications~\cite{worsley-1996,rasmussen-2004}.
Based on this idea, we propose a construction specific to the symplectic setting.
In symplectic geometry, the special class of Hamiltonian diffeomorphisms is directly related to smooth functions on the underlying manifold.
This generalizes the familiar notion of Hamiltonian functions in classical mechanics. When $M$ models the phase space of a mechanical system, $\Ham(M)$ arises as the group describing all admissible motions.
In a symplectic setting, Hamiltonian diffeomorphisms are time-$1$ flows of (time-dependent) Hamiltonian vector fields. Hamiltonian vector fields are vector fields that are dual under the symplectic form to a (time-dependent) Hamiltonian function.
By considering Hamiltonian diffeomorphisms obtained from vector fields that are dual to a Gaussian process, we can construct a probability measure on the space of Hamiltonian diffeomorphisms $\Ham(M,\omega)$.
This construction can be adapted to multiple settings. We obtain a family of measures on $\Ham(M,\omega)$ 
as well another family of measures on the space of autonomous Hamiltonian diffeomorphisms $\Aut(M,\omega) \subset \Ham(M,\omega)$.

The Hofer norm, introduced by Hofer in~\cite{hofer-1990}, is a natural bi-invariant metric on $\Ham(M,\omega)$.
It plays a central role in symplectic geometry, and many relevant questions can be phrased in terms of functions on $\Ham(M,\omega)$ that are Lipschitz or at least continuous with respect to the Hofer metric.
Random dynamical systems are much-studied, though seldom in the context of symplectic geometry and Hamiltonian dynamics. What is the expected Hofer norm of a random Hamiltonian diffeomorphism?  One purpose of this paper is to provide a framework in which this is a meaningful and rigorous question — and the answer is some definite (finite, non-zero) value, which one can even aim to estimate through computer simulation.
Thus, the family of measures we construct is compatible with the Hofer norm.
The measures are Borel measures with respect to the Hofer topology. Thus, all Hofer-continuous functions on $\Ham(M,\omega)$ are well-defined random variables with respect to the measure.
Furthermore, we show that the expectation of any Hofer-Lipschitz function with respect to the measure is finite.

\subsection{Main results and strategy}\label{subsec:general-strategy}
In this text, we will work on a \textit{closed} symplectic manifold $(M,\omega)$ endowed with a compatible almost complex structure $J$.
In the following denote by $C^\infty_0$ the set
\[
C^\infty_0 \coloneqq \left\{H \in C^\infty([0,1] \times M, \Reals) \Big\vert \int_M H_t \omega^n = 0 \forall t \in [0,1]\right\}
\]
of mean-normalized time-dependent Hamiltonian functions on $M$.
Furthermore, denote the space of smooth Hamiltonian isotopies $\{\phi_t\}_{t\in[0,1]}$ starting at the identity, i.e., $\phi_0 = \id_M$, by $\ham(M,\omega)$.
Since a full Hamiltonian isotopy uniquely determines a corresponding Hamiltonian function in $C^\infty_0$,
these two spaces are homeomorphic.
Furthermore, we have a fibration
\begin{center}
    \begin{tikzcd}
        C^\infty_{0} \arrow[rr, "\cong"] &  & {\ham(M,\omega)} \arrow[rr, two heads] &  & {\Ham(M,\omega)}
    \end{tikzcd}
\end{center}
essentially given by the path-space fibration.
Our basic idea is simple: Smooth Gaussian processes (which we will introduce in Section~\ref{subsec:gaussian-processes}) 
lead to a canonical way to construct measures on $C^\infty_0$.
These measures have been intensively studied in the probability theory and statistics literature, and have many desirable properties.
We argue that the push-forward of such a measure under the fibration above gives rise to a canonical family of probability measures on $\Ham(M,\omega)$.  

We adapt a construction from~\cite{castillo-et-al-2014,nicolaescu-2014} to obtain 
a smooth Gaussian process on $M$ by using the Riemannian structure induced by a compatible almost complex structure $J$.
In order to obtain a specific Gaussian measure on $C^\infty_0$ we need to make several choices.
To simplify the notation, we incorporate all these choices into a \textit{\ldd} $\mathcal{D}$ (see Definition~\ref{def:ldd}).
A \ldd\ $\mathcal{D}$ then specifies a unique measure $\muGPH^{\mathcal{D}}$ on $C^\infty_0$. 
The push-forward of the measure $\muGPH^{\mathcal{D}}$ on $C^\infty_0$ under the fibration above gives rise to a measure on $\Ham(M,\omega)$,
which we denote by $\muGP^{\mathcal{D}}$.
For the sake of this summary, we will not go into the details of the choices that go into a \ldd .
However, we wish to highlight the regularity parameter $\reg$ which is part of a \ldd .
Intuitively, this parameter controls how much a typical Hamiltonian function sampled from the measure oscillates.
In particular, it makes sense to look at limits as $\reg \to 0$ or $\reg \to \infty$.
If we take such a limit, it is always assumed that the other choices in the \ldd\ are fixed.
With these preliminaries in place, we can now summarize the main results.
We begin with those results that are independent of the \ldd\ $\mathcal{D}$.
\begin{theorem}\label{thm:hofer-borel}
    Let $\mathcal{D}$ be any \ldd .
    Then the measure $\muGP^{\mathcal{D}}$ is a Hofer-Borel probability measure on $\Ham(M,\omega)$, i.e. $\muGP^{\mathcal{D}}(\Ham(M,\omega)) = 1$
    and any Hofer-open set $U \subset \Ham(M,\omega)$ is measurable.
\end{theorem}
Thus, we know that the measure is always compatible with the Hofer topology.
However, the measure is not just Hofer-Borel, it is also adapted to the Hofer norm in the following sense:
\begin{theorem}\label{thm:expected-hofer-norm-finite}
    Let $\mathcal{D}$ be any \ldd .
    Then, 
    \[\int_{\Ham(M,\omega)} \norm{\phi}_{\Hof} \, d\muGP^{\mathcal{D}}(\phi) < \infty.\]
\end{theorem}
\begin{corollary}\label{cor:hofer-lipschitz-finite-expectation}
    Let $\mathcal{D}$ be any \ldd .
    For any function $f: \Ham(M,\omega) \to \Reals$ that is Lipschitz with respect to the Hofer metric, we have 
    \[\int_{\Ham(M,\omega)} f(\phi) \, d\muGP^{\mathcal{D}}(\phi) < \infty.\]
\end{corollary}
These results establish some basic facts: We have obtained a family of Borel probability measures and the Hofer norm is a well-defined random variable with a finite expectation.
The same is true for all Hofer-Lipschitz measurements, e.g.\ for Hamiltonian spectral invariants.
The facts that we have so far established, are not very strong.
Indeed, the Dirac delta at $\id_M \in \Ham(M,\omega)$ is a probability measure that fulfills the same properties.
This is due to the fact that our construction is very general.
Indeed, it is possible to obtain $\delta_{\id_M}$ as $\muGP^{\mathcal{D}}$ for a choice of ``trivial'' \ldd .
Thus, any stronger result must restrict the types of \ldds\ to which it is applicable.
We will now present several such results, based on types of \ldds\ enumerated in Definition~\ref{def:ldd-types}.

We start by giving a quantitative estimate on the tails of the Hofer norm of a random Hamiltonian diffeomorphism.
We obtain the following estimate that can be seen as a \textit{concentration of measure} phenomenon:
\begin{theorem}\label{thm:hofer-norm-subgaussian}
    Let $\mathcal{D}$ be any centered \ldd.
    Then there exist constants $C,R > 0$ depending on $\mathcal{D}$ such that 
    we have 
    \begin{equation*}
        \muGP^{\mathcal{D}}(\{\phi \in \Ham(M,\omega) \mid \norm{\phi}_{\Hof} > R + u\}) \leq 2e^{-C^{-1}u^2}
    \end{equation*}
    for any $u > 0$, i.e.\ the tails of the distribution of the Hofer norm satisfy a sub-Gaussian type estimate.
\end{theorem}
This result would be trivially true if the measure $\muGP^{\mathcal{D}}$ were concentrated at a specific Hamiltonian diffeomorphism, such as $\id_M$.
The next theorem shows that the opposite is true for suitable \ldds .
\begin{theorem}\label{thm:full-support}
    Let $\mathcal{D}$ be an exhaustive or periodically exhaustive \ldd.
    Then $\muGP^{\mathcal{D}}$ has full support on $\Ham(M,\omega)$, i.e.
    $\supp \muGP^{\mathcal{D}} = \Ham(M,\omega)$.
\end{theorem}
Thus, any Hofer ball has a positive measure under $\muGP^{\mathcal{D}}$ when $\mathcal{D}$ is exhaustive in the sense of Definition~\ref{def:ldd-types}.
Note that the above theorem alone is not sufficient to make $\muGP^{\mathcal{D}}$ a reasonable model for a random Hamiltonian diffeomorphism.
Indeed, obtaining a probability measure with full support on $\Ham(M,\omega)$ is an easy exercise.
Since $\Ham(M,\omega)$ with the Hofer metric is separable, one can simply take a countable dense subset $\{\phi_n\}_{n \in \Nats} \subset \Ham(M,\omega)$
and some sequence of weights $\{w_n\}_{n \in \Nats} \subset (0,1)$ such that $\sum_{n=1}^\infty w_n = 1$.
Then the measure $\mu = \sum_{n=1}^\infty w_n \delta_{\phi_n}$ is a probability measure with full support on $\Ham(M,\omega)$.
Thus, Theorem~\ref{thm:full-support} should be seen as a necessary rather than sufficient condition for accepting $\muGP^{\mathcal{D}}$ as a model for random Hamiltonian diffeomorphisms.
However, the lower bounds on the measure of Hofer balls, see Corollary~\ref{cor:ball_estimates}, rule out such pathological constructions.
Theorem~\ref{thm:hofer-norm-subgaussian}, Corollary~\ref{cor:hofer-lipschitz-finite-expectation} and the next theorem additionally establish that this measure is compatible with the symplectic nature of $\Ham(M,\omega)$.
These theorems also illustrate characteristically Gaussian properties of the measure $\muGP^{\mathcal{D}}$ which it inherits from $\muGPH^{\mathcal{D}}$.
Interestingly, 
these estimates can be used to prove the following result about the smoothness (as measured by the decay of the coefficients with respect to an eigenbasis of the Laplace-Beltrami operator) of Hamiltonian functions generating Hamiltonian diffeomorphisms
that form an infinite-dimensional constellation in the Hofer metric.
It is noteworthy that while the proof of this statement makes use of the existence and properties of the measure $\muGP^{\mathcal{D}}$,
the result is not itself probabilistic in nature.
\begin{theorem}\label{thm:smoothness-quasi-flats} 
    Let $\{\phi_n\}_{n \in \Nats} \subset \Ham(M,\omega)$ be a sequence of Hamiltonian diffeomorphisms such that
    there exist $\eps,\delta > 0$ such that $\norm{\phi_n}_{\Hof} \leq \eps$ for all $n \in \Nats$
    and $d_{\Hof}(\phi_n,\phi_m) \geq 2\delta$ for all $n \neq m$.
    Then for any Hamiltonians $\{H_n\}_{n \in \Nats} \subset C^\infty_0$ 
    such that $\phi_{H_n}^1 = \phi_n$ for all $n \in \Nats$, we have
    \begin{equation*}
        \lim_{n \to \infty} \sum_{k=1}^{\infty} \exp\left(\eps \lambda_k \right) \langle H_n, e_k \rangle_{L^2} = \infty,
    \end{equation*}
    for all $\eps > 0$ and any almost complex structure $J$ compatible with $\omega$.
    Here $\{e_n\}_{n \in \Nats} \subset L^2([0,1] \times M)$ denotes an orthonormal basis of eigenfunctions of $\Laplace - \partial_t^2$
    and $\{\lambda_n\}_{n \in \Nats}$ denotes the corresponding eigenvalues.
    Here $\Laplace_g$ is the Laplace-Beltrami operator associated to the Riemannian metric $g = \omega(\cdot,J\cdot)$.
\end{theorem}
\begin{corollary}
    The conclusions of Theorem~\ref{thm:smoothness-quasi-flats}
    apply whenever there is a quasi-isometric embedding of $(\Reals^{\infty},d_{\infty})$ into $(\Ham(M,\omega),d_{\Hof})$.
\end{corollary}
\rmk{
    It seems likely that there is a more direct analytical way of proving this result.
    However, the fact that it arises as a consequence of the existence and properties of the measure $\muGP^{\mathcal{D}}$ is interesting in itself.
}
It is natural to ask how compatible the measure $\muGP^{\mathcal{D}}$ is with the group structure of $\Ham(M,\omega)$.
Since $\Ham(M,\omega)$ is not compact it is not possible for $\muGP$ to be invariant under the action of $\Ham(M,\omega)$ on itself by left or right multiplication.
Indeed, the behavior of this measure under composition with Hamiltonian diffeomorphisms is quite subtle, see Remark~\ref{rmk:symplectic_cameron_martin}.
However, we have the following theorem on the compatibility of $\muGP^{\mathcal{D}}$ with the group structure of $\Ham(M,\omega)$.
\begin{theorem}\label{thm:inversion-invariance}
    Assume that $\mathcal{D}$ is time-symmetric and centered.
    For any Hofer-Borel set $U \subset \Ham(M,\omega)$, we have that 
    \[\muGP^{\mathcal{D}}(U) = \muGP^{\mathcal{D}}\left(\left\{\phi^{-1} \mid \phi \in U\right\}\right),\]
    i.e., the measure $\muGP^{\mathcal{D}}$ is invariant under inversion.
\end{theorem}

It is a well-known result due to Banyaga, see~\cite{banyaga-1997}, that the group $\Ham(M,\omega)$ is generated by autonomous Hamiltonian diffeomorphisms,
i.e.\ by Hamiltonian diffeomorphisms that admit a time-independent Hamiltonian function.
We will use this fact to additionally propose a model for a random walk on $\Ham(M,\omega)$.

Recall that if $G$ is a group and $\mu$ is a probability measure on $G$, then
this gives rise to a random walk $X_n$ on $G$ defined by
$X_n = g_1 \cdots g_n$,
where $g_1, g_2, \ldots$ are independent and identically distributed $G$-valed random variables with law $\mu$.
For example $G = \Z$ and $\mu = \frac{1}{2}(\delta_1 + \delta_{-1})$ is a very simple example of a random walk on $\Z$.
Note that it makes sense for $\mu$ to not be supported on the whole group and instead to be supported on a generating set that is invariant under inversion.
The elements of this group are seen as the ``steps'' the random walker can take, as is illustrated by the above example.
The following theorem implies that it is possible to define a random walk on $\Ham(M,\omega)$ using $\Aut(M,\omega)$ as the set of ``steps''.
Indeed, there are suitable \ldds\ $\mathcal{D}$ such that a Hamiltonian diffeomorphism sampled from $\muGP^{\mathcal{D}}$ 
is almost-surely autonomous.
\begin{theorem}\label{thm:aut-measure}
    Let $\mathcal{D}$ be an autonomous \ldd.
    Then the support of the measure $\muGP^{\mathcal{D}}$ is contained in the Hofer-closure $\overline{\Aut}(M,\omega)$ of $\Aut(M,\omega) \subset \Ham(M,\omega)$.
    Furthermore, 
    there exists a Borel probability measure $\muAut^{\mathcal{D}}$ on $\Aut(M,\omega)$ such that 
    $\muGP^{\mathcal{D}}$ is the push-forward of $\muAut^{\mathcal{D}}$ under the inclusion $\Aut(M,\omega) \hookrightarrow \Ham(M,\omega)$.
\end{theorem}
This allows us to define a random walk on $\Ham(M,\omega)$.
Let $\mathcal{D}$ be an autonomous \ldd\ and let $\phi_0,\phi_1,\ldots$ be independent and identically distributed $\Aut(M,\omega)$-valued random variables with law $\muAut^{\mathcal{D}}$ 
(or equivalently with law $\muGP^{\mathcal{D}}$ when seen as $\Ham(M,\omega)$-valued random variables).
Then the random walk
\begin{equation}\label{eq:random-walk}
    \Phi_n \coloneqq \phi_n \circ \cdots \circ \phi_1
\end{equation}
is a $\Ham(M,\omega)$-valued random variable.
The first interesting thing that we can say about this random walk is that its law at any step $n \in \Nats$ 
is again given by a measure following the construction outlined above.
\begin{theorem}\label{thm:random-walk-law}
    Let $\mathcal{D}$ be an autonomous \ldd\ and $(\Phi_n)_{n \in \Nats}$ be the random walk defined in~\eqref{eq:random-walk}.
    Then for any $n \in \Nats$, the law of $\Phi_n$ is given by $\muGP^{\mathcal{D}_n}$ for some \ldd\ $\mathcal{D}_n$.
\end{theorem}
\begin{theorem}\label{thm:random-walk-symmetric}
    Let $\mathcal{D}$ be centered autonomous \ldd .
    Further, let $(\Phi_n)_{n \in \Nats}$ be the random walk defined in~\eqref{eq:random-walk}.
    Then for any $n \in \Nats$, the law of $\Phi_n$ is given by $\muGP^{\mathcal{D}_n}$ for a time-symmetric and centered \ldd\ $\mathcal{D}_n$.
    In particular, the law of $\Phi_n$ is invariant under inversion.
\end{theorem}
Furthermore, the random walk $(\Phi_n)_{n \in \Nats}$ explores all of $\Ham(M,\omega)$ for suitable $\mathcal{D}$,
as we will establish with the next theorem.
This is of course related to the fact that $\Aut(M,\omega)$ generates $\Ham(M,\omega)$.
Note that the word metric on $\Ham(M,\omega)$ with respect to the generating set $\Aut(M,\omega)$ is called the \textit{autonomous norm}
and defined by 
\begin{equation}
    \norm{\phi}_{\Aut} \coloneqq \inf\left\{n \in \Nats \mid \exists \phi_1,\ldots,\phi_n \in \Aut(M,\omega) \text{ s.t. } \phi = \phi_1 \circ \cdots \circ \phi_n\right\}.
\end{equation}
This induces a bi-invariant metric $d_{\Aut}(\phi,\psi) = \norm{\phi^{-1} \psi}_{\Aut}$ on $\Ham(M,\omega)$.
Many of the properties of this norm and its induced metric remain mysterious.
On the one hand, it is known that it is bounded in the case of $M = \Reals^{2n}$, see~\cite{brandenbursky-kedra-2017}.
On the other hand, it is known to be unbounded when $M$ is a surface, see~\cite{brandenbursky-2014, brandenbursky-kedra-shelukhin-2018}.
Of course, by definition of our random walk, we have $\norm{\Phi_n}_{\Aut} \leq n$ for all $n \in \Nats$.
Thus, the speed at which $\Phi_n$ can ``exhaust'' $\Ham(M,\omega)$ is directly related to the growth of the autonomous norm,
as is formalized by the following theorem:
\begin{theorem}\label{thm:random-walk-filling}
    Let $\mathcal{D}$ be an autonomously exhaustive \ldd\ and $(\Phi_n)_{n \in \Nats}$ be the random walk defined in~\eqref{eq:random-walk}.
    For any $\phi \in \Ham(M,\omega)$ and $\eps > 0$,
    \[\Prob\left[ \Phi_n \in B_{\Hof}(\phi,\eps)\right] > 0\]
    for all $n \geq \norm{\phi}_{\Aut}$.
\end{theorem}
\rmk{
    One hope is that the behavior of the random walk might be able to illuminate how quickly the shells $S(n) = \{\phi \in \Ham(M,\omega) \mid \norm{\phi}_{\Aut} = n\}$ fill up $\Ham(M,\omega)$.
    In general, this question has been hard to attack.
    In~\cite{polterovich-shelukhin-2016}, Polterovich and Shelukhin show that there are elements of $S(2)$ which are arbitrarily far away (in the Hofer distance)
    from $S(1)$. However, already the question of whether elements of $S(3)$ can be arbitrarily far from $S(2)$ is wide open.
}
\rmk{
We will see in Section~\ref{sec:existence-of-ldd}, there are canonical ways to choose a \ldd\ if $\reg$ and the almost complex structure $J$ are fixed.
In particular, suitable \ldds\ always exist.
Note that manifolds such as $\T^n, \C P^n$, and 
in general all K\"ahler manifolds $(M,\omega,I)$
come with a preferred complex structure.
It is of course sensible to choose this complex structure as part of the \ldd .
The law-defining datum $\mathcal{D}_2$ from the proof of Theorem~\ref{thm:existence-of-ldd} 
then gives rise to a measure $\muGP^{\mathcal{D}_2}$ that has full support on $\Ham(M,\omega)$
and satisfies all the properties we would want to apply the aforementioned theorems.
This is a construction that is now canonical up to the choice of the regularity parameter $\reg$.
Given some Hofer-Lipschitz function $f: \Ham(M,\omega) \to \Reals$,
e.g.\ a spectral invariant, the map $\reg \mapsto \int_{\Ham(M,\omega)} f(\phi) \, d\muGP^{\mathcal{D}_{2,\reg}}(\phi)$
can then be seen as a completely intrinsic invariant of the K\"ahler manifold $(M,\omega,I)$.
Here $\mathcal{D}_{2,\reg}$ is the \ldd\ obtained from $\mathcal{D}_2$ by replacing the given regularity parameter with $\reg$.
Ratios of expected spectral invariants could similarly be seen as intrinsic invariants of $(M,\omega,I)$.
}
\subsection{Limiting behavior and $C^0$-symplectic geometry}
Given a \ldd\ $\mathcal{D}$, one can consider the family of \ldds\ $\mathcal{D}_\reg$ for $\reg > 0$ which are obtained from $\mathcal{D}$ by replacing the given regularity parameter with $\reg$.
We now wish to understand what happens to the measures $\muGP^{\mathcal{D}_\reg}$ as $\reg \to 0$ or $\reg \to \infty$.

The limit as $\reg \to \infty$ is easily understood.
For $\reg$ large, the measure $\muGP^{\mathcal{D}_\reg}$ is concentrated more and more at the identity if $\mathcal{D}$ is centered.
The following theorem makes this precise.
\begin{theorem}\label{thm:limit-reg-infty}
    Let $\mathcal{D}$ be a centered \ldd .
    As $\reg \to \infty$, we have $\muGP^{\mathcal{D}_\reg} \weakto \delta_{\id_M}$, i.e.\ the measure $\muGP^{\mathcal{D}_\reg}$ converges weakly to $\delta_{\id_M}$.
\end{theorem} 

The limit $\reg \to 0$ is more subtle.
The regularity of the Hamiltonian functions becomes worse and worse in this limit, corresponding to more and more ``energy''
being added to the Hamiltonian function. We have to impose additional restrictions to obtain a well-defined limiting measure.
This is done by choosing a \textit{limiting type}, see Definition~\ref{def:ldd-types}.
Since the resulting limiting object will no longer be smooth Hamiltonian diffeomorphisms,
we have to choose a suitable completion or other enlargement of $\Ham(M,\omega)$.
We thereby enter the realm of $C^0$-symplectic geometry.
\begin{figure}[h]
    \begin{center}
        \begin{tikzcd}
        {\Ham(M,\omega)} \arrow[d, phantom, "\bigcap"] \arrow[drr, hook]                   &  &                                                    &  &                                   \\
        {\Hameo(M,\omega)} \arrow[d, hook] \arrow[rr, hook] &  & {\widehat{\Ham}{}^{\Hof}(M,\omega)} \arrow[rr, hook] &  & {\widehat{\Ham}^\gamma(M,\omega)} \\
        {\overline{\Ham}(M,\omega)}                          &  &                                                    &  &                                  
        \end{tikzcd}
    \end{center}
    \caption{Various enlargements of the space of Hamiltonian diffeomorphisms and their relations.
    }\label{fig:enlargements}
\end{figure}
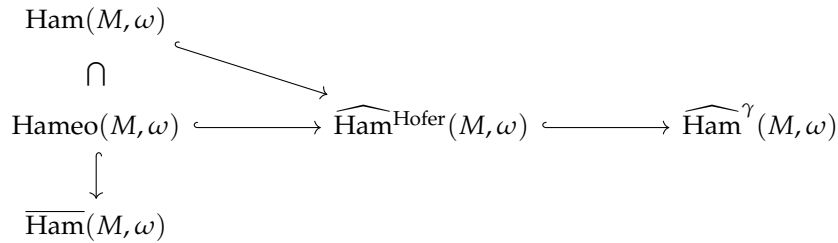

There are multiple reasonable choices, see Figure~\ref{fig:enlargements} for an illustration of some possible choices
of enlargements of $\Ham(M,\omega)$.
The easiest way to obtain a limiting object is to consider the metric completion of $\Ham(M,\omega)$ with respect to the Hofer metric.
This is a purely metric construction and no longer a space of maps.
This completion $\widehat{\Ham}{}^{\Hof}(M,\omega)$ has not been the object of intense study so far.
However, there is a finer completion of $\Ham(M,\omega)$ that has been studied in the literature,
namely the \textit{Humilière completion} $\widehat{\Ham}^\gamma(M,\omega)$
introduced in~\cite{humiliere-2008}.
This is the metric completion of $\Ham(M,\omega)$ with respect to the spectral metric $\gamma$.
Note that there is a continuous inclusion $\widehat{\Ham}{}^{\Hof}(M,\omega) \into \widehat{\Ham}^\gamma(M,\omega)$.
Even though the elements of 
these completions are not diffeomorphisms, 
some can still be interpreted geometrically.
This was first observed by Humilière in~\cite[Section 5]{humiliere-2008} and 
recently studied in more detail by Viterbo in~\cite{viterbo-2024}.
Thus, it is natural 
to ask when we obtain a limiting measure on $\widehat{\Ham}{}^{\Hof}(M,\omega)$.
By pushing this measure forward under the inclusion, we also obtain a probability measure on $\widehat{\Ham}^\gamma(M,\omega)$.
The following theorem and corollary make this precise:
\begin{theorem}\label{thm:limit-reg-0-C0}
    Let $\mathcal{D}$ be a \ldd\ of $C^0$-limiting type.
    Then there exists a Borel probability measure $\widehat{\muGP}^{\mathcal{D}_0}$
    on $\widehat{\Ham}{}^{\Hof}(M,\omega)$ such that
    as $\reg \to 0$, we have $\muGP^{\mathcal{D}_\reg} \weakto \widehat{\muGP}^{\mathcal{D}_0}$.
\end{theorem}
\begin{corollary}
    Let $\mathcal{D}$ be a \ldd\ of $C^0$-limiting type.
    Then there exists a Borel probability measure $\widehat{\muGP}^{\mathcal{D}_0}$
    on $\widehat{\Ham}^{\gamma}(M,\omega)$ such that
    as $\reg \to 0$, we have $\muGP^{\mathcal{D}_\reg} \weakto \widehat{\muGP}^{\mathcal{D}_0}$.
\end{corollary}

Another approach towards a limiting object for $\Ham(M,\omega)$
is to try to generalize the concept 
of a Hamiltonian diffeomorphism to the $C^0$-setting.
This idea was introduced by Müller and Oh in~\cite{oh-muller-2007}.
They define the group of \textit{Hamiltonian homeomorphisms} or \textit{Hameomorphisms}
denoted by $\Hameo(M,\omega)$.
We briefly recall the definition:
\begin{definition}
    Let $(\phi^t)_{t\in [0,1]}$ be a (not necessarily smooth) isotopy of $M$.
    We say that $\phi^t$ is a \textit{continuous Hamiltonian flow} (also called a \textit{hameotopy}), if there exists a sequence of smooth Hamiltonians $H_i:[0,1] \times M \to \Reals$ such that:
    \begin{enumerate}
        \item The sequence of flows $\phi_{H_i}^t$ converges to $\phi^t$ in the $C^0$-distance. Furthermore, we require this convergence to be uniform in $t$.
        \item The sequence of Hamiltonian functions $H_i$ converges uniformly to a $C^0$-function $H: [0,1] \times M \to \Reals$,
        i.e.\ $\norm{H_i - H}_{\infty} \to 0$ as $ i \to \infty$.
    \end{enumerate}
    We say that $H$ generates $\phi^t$, denote $\phi_H^t \coloneqq \phi^t$, and call $H$ a \textit{$C^0$-Hamiltonian}.
\end{definition}
\begin{definition}\label{def:hameo}
    A homeomorphism $\phi: M \to M$ is called a \textit{Hamiltonian homeomorphism} or \textit{hameomorphism}, if there exists a continuous Hamiltonian flow $\phi_H^t$ such
    that $\phi_H^1 = \phi$.
    We denote the set of all Hamiltonian homeomorphisms by $\Hameo(M,\omega)$.
\end{definition}
\rmk{
    A small notational remark: We use the notation $\norm{\cdot}_\infty$ to denote the $\sup$-norm, i.e.
    $\norm{f}_\infty = \sup_{x \in X} \norm{f(x)}$, where $f: X \to E$ maps from a topological space $X$ to a normed vector space $E$.
    Thus, the $\infty$-norm is the norm inducing the $C^0$-topology.
    We use the same notation for the $L^\infty$-norm, when dealing with $L^\infty$-functions.
}
It is established in~\cite{oh-muller-2007} that $\Hameo(M,\omega)$ is normal subgroup of the $C^0$-closure 
of $\Symp(M,\omega) \subset \Homeo(M,\omega)$.
Furthermore, it has been established in~\cite{oh-muller-2007} that 
a continuous Hamiltonian $H$ generates a unique continuous Hamiltonian flow. 
It is also true that a continuous Hamiltonian flow has a unique\footnote{The uniqueness holds once properly normalized. Otherwise, one can say that the function is unique up to addition of a function of time} 
generating $C^0$-Hamiltonian.
This was shown by Viterbo in~\cite{viterbo-2006} and Buhovsky-Seyfaddini in~\cite{buhovsky-seyfaddini-2013}.
Thus, the relationship between Hameomorphisms and $C^0$-Hamiltonian functions 
is much like in the smooth case.
Thus, it is again natural to ask under what conditions we obtain a limiting measure on $\Hameo(M,\omega)$.
The answer is given by the following theorem:
\begin{theorem}\label{thm:limit-reg-0-C1}
    Let $\mathcal{D}$ be a \ldd\ of $C^1$-limiting type.
    Then there exists a Borel probability measure $\overline{\muGP}^{\mathcal{D}_0}$
    on $\Hameo(M,\omega)$ such that  
    as $\reg \to 0$, we have $\muGP^{\mathcal{D}_\reg} \weakto \overline{\muGP}^{\mathcal{D}_0}$.
\end{theorem}
\begin{corollary}
    Let $\mathcal{D}$ be a \ldd\ of $C^1$-limiting type.
    Then there exists a Borel probability measure $\overline{\muGP}^{\mathcal{D}_0}$
    on $\overline{\Ham}(M,\omega)$ such that
    as $\reg \to 0$, we have $\muGP^{\mathcal{D}_\reg} \weakto \overline{\muGP}^{\mathcal{D}_0}$.
    Here $\overline{\Ham}(M,\omega)$ denotes the $C^0$-closure of $\Ham(M,\omega)$ in $\Homeo(M,\omega)$.
\end{corollary}
It is noteworthy that the proof of this result relies on a highly non-trivial result from~\cite{joksimovic-seyfaddini-2024} relating the Hofer and $C^0$-norms and the norm of the differential of a Hamiltonian diffeomorphism.
Thus, this theorem makes essential use of the symplectic nature of $\Ham(M,\omega)$
and $\Hameo(M,\omega)$.

\subsection{Previous approaches and related work} 
Gaussian processes have a rich interaction with geometry.
A comprehensive account of the theory of Gaussian processes on manifolds is given in~\cite{adler-taylor-2007}.
This book gives various ways in which the properties of Gaussian processes on manifolds 
can be related to topological and geometric properties of the underlying manifold.
Another prime example of a connection between smooth Gaussian processes and Riemannian geometry is given in~\cite{nicolaescu-2014},
where a specific model of a Gaussian process on a Riemannian manifold is constructed using the eigenfunctions of the Laplace-Beltrami operator.
Nicolaescu then shows that the asymptotic behavior of the statistics of these Gaussian processes completely recover the Riemannian metric as the regularity of the function becomes worse, see~\cite[Theorem 1.7]{nicolaescu-2014}.

Gaussian processes that model vector fields on a manifold have also been studied in the literature.
In particular, there has been recent interest in Gaussian Reeb and Beltrami vector fields, as in~\cite{cordona-2023, enciso-2023}. 
A particularly noteworthy result in this direction is that a canonical model of Gaussian random
Beltrami fields exhibits arbitrarily high computational complexity with probability 1, see~\cite[Theorem 4]{cordona-2023}.
The dynamics, specifically the non-integrability, of specific models of random Hamiltonian vector fields have been studied in~\cite{enciso-2024}.
Here the authors consider the Hamiltonian vector field generated by the physical Hamiltonian $H(q,p) = \frac{1}{2} p^2 + V(q)$, where $V$ is a random potential.
In particular, the case of a Gaussian random potential is considered.

More general models of random Hamiltonian functions are also considered in the symplectic literature.
In particular, Viterbo studied the Hamiltonian-Jacobi equation for a random Hamiltonian function $H: T^*\Reals^n \to \Reals$ in~\cite{viterbo-2025}.
The results of this paper are in some sense agnostic to the underlying distribution of the random Hamiltonian function, as long as some conditions are satisfied.
Another randomized approach to Hamiltonian flows in $\Reals^{2n}$ is taken in~\cite{pelayo-rezakhanlou-2025}.
Here the dynamics, in particular the periodic orbits, of a random Hamiltonian function on $\Reals^{2n}$ are studied.
Several physically-motivated models of random Hamiltonian functions are considered, see~\cite[Section 2]{pelayo-rezakhanlou-2025}.
The authors of~\cite{pelayo-rezakhanlou-2025} also consider the case of a Gaussian random Hamiltonian function~\cite[Remark 2.12]{pelayo-rezakhanlou-2025}.

\subsection{Future work}
This paper only introduces the measure $\muGP^{\mathcal{D}}$ and establishes 
the most basic properties of this class of probability measures.
Many questions remain unanswered.
For example the statistics of fixed points and periodic orbits of Hamiltonian diffeomorphisms sampled from $\muGP^{\mathcal{D}}$ are not studied here.
Additionally, in Section~\ref{sec:simulations}
we present some numerical simulations.
These simulations strongly suggest the presence of 
diffusion effects and Crofton type formulas for Lagrangian intersections.
We make an explicit conjecture for such a Crofton-type formula in Conjecture~\ref{conj:crofton}.
A follow-up paper will be dedicated to
the rigorous study of these phenomena.
\subsection{Acknowledgements}
I would like to thank my advisor Ivan Smith for many invaluable discussions and for his adamant support of this project. 
Furthermore, I would like to thank Richard Nickl for explaining basic Gaussian process theory to me and particularly for pointing out reference~\cite{castillo-et-al-2014}.
Additionally, I would like to thank Jonny Evans for an enlightening discussion about Crofton formulas for Lagrangian intersections.
During the writing
of this paper I was supported by EPSRC Horizon Europe Guarantee grant ``Floer theory beyond Floer (FloerPlus35)`` (project reference EP/X030660/1).
\section{Background}
\label{sec:background}

\subsection{Symplectic geometry}
\label{subsec:symplectic-geometry}
In this section we will review the necessary background on symplectic geometry.
A reader well-versed in the subject is advised to skip this section.
We will keep this review brief and refer to~\cite{cannas-da-silva-2008} and~\cite{mcduff-salamon-1998} for more details.
\begin{definition}\label{def:symplectic-manifold}
    Let $M$ be a $2n$-dimensional manifold and $\omega \in \Omega^2(M)$ be a closed 2-form on $M$.
    Then $(M,\omega)$ is called a \textit{symplectic manifold} if $\omega$ is non-degenerate, i.e.\ if $\omega^n$ is a volume form on $M$.
\end{definition}
\begin{definition}\label{def:symplectomorphism}
    Let $(M,\omega)$ be a symplectic manifold.
    A diffeomorphism $\phi: M \to M$ is called a \textit{symplectomorphism} if $\phi^* \omega = \omega$.
\end{definition}
We will later be most interested in closed symplectic manifolds. Thus, let us assume in the remainder of this section that $M$ is closed, i.e.\ compact and without boundary.
We will further always assume that $M$ is connected.
Then any smooth function on $M$ gives rise to a vector field under the symplectic form $\omega$.
We call such functions and vector fields Hamiltonian functions and Hamiltonian vector fields, respectively.
\begin{definition}
    Let $H \in C^\infty(M, \Reals)$ be a time-independent Hamiltonian function\footnote{Such a Hamiltonian is also called autonomous.}. 
    Then the vector field $X_H$ on $M$ given by the equation $\omega(\cdot, X_H) = dH$ is called the \textit{(autonomous) Hamiltonian vector field} of $H$.
    Let $H \in C^{\infty}([0,1] \times M,\Reals)$ be a time-dependent Hamiltonian function and let $H_t(x) \coloneqq H(t,x)$. Then the time-dependent vector field $X_H$ given by $\omega(\cdot,X_H(t,\cdot)) = dH_t$ for all $t \in [0,1]$ is called the \textit{Hamiltonian 
    vector field} of $H$. 
\end{definition}
Note that $H$ is not uniquely determined by $X_H$,
but that it is unique up to a constant.
Thus, it is convenient to consider \textit{normalized} Hamiltonian functions.
\begin{definition}\label{def:normalized-hamiltonian}
    A Hamiltonian function $H \in C^\infty([0,1] \times M,\Reals)$ is called \textit{normalized}
    if for all $t \in [0,1]$ we have 
    \[\int_M H_t \omega^n = 0.\]
\end{definition}
Note that there is a one-to-one correspondence between normalized autonomous Hamiltonian functions and Hamiltonian vector fields.
Recall from Section~\ref{subsec:general-strategy} that we introduced the notation $C^\infty_0$ for time-dependent Hamiltonian functions that are normalized.
We will be interested in Hamiltonian flows, i.e.\ the flows of (time-dependent) Hamiltonian vector fields.
\begin{definition}
    A diffeomorphism $\phi: M \to M$ is called \textit{Hamiltonian diffeomorphism} if it is the time-$1$ flow of a Hamiltonian vector field,
    i.e.\ if there exists a Hamiltonian function $H \in C^{\infty}([0,1] \times M,\Reals)$ and a smooth family $(\phi_t)_{t \in [0,1]}$ of diffeomorphisms such that 
    $\phi_1 = \phi$ and $\phi_0 = \id_M$ and such that for all $t \in (0,1)$ we have 
    \[\frac{d}{dt}\phi_t(x) = X_H(\phi_t(x))\]
    for all $x \in M$.
\end{definition} 
It is easy to see that Hamiltonian diffeomorphisms are symplectomorphisms.
We call the family $(\phi_t)_{t \in [0,1]}$ in the definition above the \textit{Hamiltonian flow} of $H$.
Given any Hamiltonian function $H \in C^{\infty}([0,1] \times M,\Reals)$, we denote the time-$t$ flow of $X_H$ by $\phi_H^t$.
The following standard lemma describes how we can manipulate Hamiltonian flows via their Hamiltonian functions. 
The proof of this lemma is by straightforward computation and is left to the reader. 
\begin{lemma}\label{lem:flows_composition_and_inversion} 
    Let $F,G \in C^\infty([0,1]\times M , \Reals)$ be arbitrary.
    Set 
    \begin{align*}
        F \sharp G(t,x) &\coloneqq F(t,x) + G(t,(\phi^t_F)^{-1}(x)), \\
        \overline{F}(t,x) &\coloneqq -F(t,\phi^t_F(x)) \\
        \hat{F}(t,x) &\coloneqq -F(1-t,x).
    \end{align*}
    Then for any $t \in [0,1]$ we have
    \begin{enumerate}
        \item $\phi_{F\sharp G}^t = \phi_F^t \circ \phi_G^t$;
        \item $\phi_{\overline{F}}^t = (\phi_F^t)^{-1}$; and 
        \item $\phi_{\hat{F}}^t = \phi_F^{1-t} \circ (\phi_F^1)^{-1}$.
    \end{enumerate}
\end{lemma}
This lemma implies that we can compose Hamiltonian flows in a natural way. In particular, the composition of two Hamiltonian diffeomorphisms is again a Hamiltonian diffeomorphism.
Furthermore, the inverse of a Hamiltonian diffeomorphism is again a Hamiltonian diffeomorphism.
Thus, we obtain the following standard result:
\begin{theorem}
    Let $(M,\omega)$ be a symplectic manifold. 
    The set of Hamiltonian diffeomorphisms $\Ham(M,\omega)$ is a subgroup of the diffeomorphism group.
\end{theorem}
We should note that for autonomous Hamiltonian functions,
there is an alternative formula for the Hamiltonian function of the composition.
Since we will make use of this fact, we formalize it through the following lemma:
\begin{lemma}\label{lem:composition_of_autonomous_flows}
    Let $H_1,\ldots,H_k$ be autonomous Hamiltonian functions
    and let $\beta: \Reals \to \Reals_{\geq 0}$
    be a smooth bump function such that $\supp \beta \Subset (0,1)$ and $\int_0^1 \beta(t) dt = 1$.

    Then for the Hamiltonian function $H$ given by
    \[
    H(t,x) = \sum_{n=1}^k k\cdot\beta(kt-n+1) \cdot H_n(x),
    \]
    we have $\phi_H ^ 1 = \phi_{H_k}^1 \circ \cdots \circ \phi_{H_1}^1$.
    \end{lemma}
\begin{proof}
    This lemma has a clear geometric interpretation: we simply run the dynamics generated by each $H_n$ in order.
    Time and energy are being rescaled to allow for all of this to happen within the unit time interval $[0,1]$.
    Let us now make this precise.
    Let $x \in M$ be arbitrary.
    Then we define a path $\gamma: [0,1] \to M$ starting at $x$ by
    $\gamma(t) = \phi^{\int_0^t k\beta(ks-n+1)ds}_{H_n}(\gamma((n-1)/k))$ for $t \in [(n-1)/k,n/k]$ and $n = 1,\ldots,k$.
    Because the support of $\beta$ is a compact subset of $(0,1)$, this is a smooth path and clearly 
    $(\phi_{H_k}^1 \circ \cdots \circ \phi_{H_1}^1)(x) = \gamma(1)$.
    Here we use that the integral of $\beta$ is one, and thus $\int_0^{1/k} k \beta(ks) ds = 1$.
    Now, we have that 
    \begin{equation*}
    \dot\gamma(t) = k \beta(kt-n+1)X_{H_n}\left(\phi^{\int_0^t k\beta(ks-n+1)ds}_{H_n}(\gamma((n-1)/k))\right)
    = k \beta(kt-n+1)X_{H_n}(\gamma(t))
    \end{equation*}
    for $t \in [(n-1)/k,n/k]$.
    Thus, 
    \begin{equation*}
        \dot\gamma(t) = \sum_{n=1}^k k\cdot\beta(kt-n+1) \cdot X_{H_n}(\gamma(t)),
    \end{equation*}
    which implies that $\dot\gamma(t) = X_H(\gamma(t))$ for all $t\in [0,1]$.
    Thus, $\gamma(1) = \phi^1_H(x)$.
    Since $x$ was arbitrary, we have established the claim.
\end{proof}

It is a deep fact, first discovered by Hofer in~\cite{hofer-1990}, that the group $\Ham(M, \omega)$ admits a bi-invariant metric, which is now known as the \textit{Hofer metric}.
Moreover, this metric can be seen as an infinite-dimensional Finsler metric on the space of Hamiltonian diffeomorphisms.
Recall that a Finsler structure on a manifold $M$ is smooth family of norms on the tangent spaces $T_xM$ for $x \in M$.
Now assume we have a smooth Hamiltonian flow (or path) $\phi_t: M \to M$ for $t \in [0,1]$.
Then at any time $t \in [0,1]$ we have $\frac{d}{dt}\phi_t = X_F$ for some autonomous Hamiltonian function $F$.
Thus, we can identify the tangent space $T_{\phi_t}\Ham(M,\omega)$ with the space of autonomous Hamiltonian functions $C^\infty(M,\Reals)/\Reals$ defined up to a constant.
This space is canonically identified with the space $C^\infty_0(M,\Reals)$ of smooth normalized functions on $M$.
Then we can define the following norm on $C^\infty_0(M,\Reals)$:
\[\norm{f}_{\osc} \coloneqq \max_{x \in M} f(x) - \min_{x \in M} f(x).\]
Then we can define the (Hofer) length of the Hamiltonian flow $\phi_t$ as
\[\text{length}(\phi_t) \coloneqq \int_0^1 \norm*{\frac{d}{dt}\phi_t}_{\osc} dt.\]
It is easy to see that the Hamiltonian paths from $\id_M$ to a specified Hamiltonian diffeomorphism $\phi$ are in one-to-one correspondence with the Hamiltonian functions $H \in C^\infty([0,1] \times M,\Reals)$ such that $\phi = \phi_H^1$.
This motivates the following definition:
\begin{definition}\label{def:osc}
    Let $H \in C^{\infty}([0,1] \times M,\Reals)$ be a Hamiltonian. Then the \textit{oscillation} of $H$ is given by 
    \[ \osc (H) \coloneqq \int_0^1 \norm{H_t}_{\osc} dt = \int_0^1 \left[ \max_{x \in M} H(t,x) - \min_{x \in M} H(t,x) \right] dt.\]
\end{definition}
On a Finsler manifold, the distance between two points is given by the infimum of the lengths of all paths connecting them.
We take same approach to defining the Hofer distance between two Hamiltonian diffeomorphisms. Indeed, it suffices to define the \textit{Hofer norm} of Hamiltonian diffeomorphisms.
\begin{definition}\label{def:hofer-norm}
    Let $\phi: M \to M$ be a Hamiltonian diffeomorphism. Then the \textit{Hofer norm} of $\phi$ is given by 
    \[\norm{\phi}_{\Hof} \coloneqq \inf \left\{ \osc (H) \mid H \in C^\infty([0,1]\times M, \Reals), \phi_H^1 = \phi \right\}.\]
    For any $\phi, \psi \in \Ham(M,\omega)$, we define the \textit{Hofer distance} between $\phi$ and $\psi$ as
    \[d_{\Hof}(\phi,\psi) \coloneqq \norm{\phi^{-1} \circ \psi}_{\Hof}.\]
\end{definition}
The following theorem is due to Hofer~\cite{hofer-1990} in the case of $\Reals^{2n}$ and was later generalized to all symplectic manifolds Lalonde and McDuff~\cite{lalonde-mcduff-1995}.
\begin{theorem}\label{thm:hofer-metric}
    Let $(M,\omega)$ be a symplectic manifold. 
    Then the Hofer distance $d_{\Hof}$ is a bi-invariant metric on the group $\Ham(M,\omega)$.
\end{theorem}
We refer the reader to the excellent survey~\cite{polterovich-2001} for more details on the Hofer metric and other aspects of the geometry of $\Ham(M,\omega)$.
\subsection{Sobolev spaces} 
We will now briefly review the necessary background on Sobolev spaces on Riemannian manifolds.
Since we will later only look at those Riemannian manifolds that arise from symplectic manifolds, we will directly assume that we have a symplectic manifold $(M,\omega)$.
\begin{definition}
    \label{def:almost-complex-structure}
    Let $(M,\omega)$ be a closed symplectic manifold.
    A smooth bundle map $J: TM \to TM$ is called an \textit{almost-complex structure} if it satisfies $J^2 = -\id_{TM}$.
    If $\omega(\cdot, J \cdot)$ is a Riemannian metric on $M$, then $J$ is called \textit{$\omega$-compatible}.
\end{definition}
In the following, we will assume that $(M,\omega)$ is a closed symplectic manifold.
Further, let $J$ be an $\omega$-compatible almost-complex structure.
We denote the induced Riemannian metric by $g_{\omega,J}$, i.e.\ for any $p \in M$ and $v,w \in T_pM$ we have $g_{\omega,J}(v,w) = \omega(v, Jw)$.
Note that the below holds for any Riemannian metric on $M$.

We denote the $L^2$-Sobolev space with regularity $2k$ on $M$ by $W^{2k,2}(M)$.
It is defined as follows 
\[W^{2k,2}(M) \coloneqq \{u\mid u, \Laplace u, \Laplace^2 u, \dots, \Laplace^k u \in L^2(M)\}\]
with Sobolev norm 
\begin{equation*}
\norm{u}_{W^{2k,2}}^2 \coloneqq \sum_{l = 0}^{k} \norm{\Laplace^l u}_{L^2}^2.
\end{equation*}
Here $\Laplace$ is the Laplace-Beltrami operator on $M$ induced by the Riemannian metric $g_{\omega,J}$, 
i.e.\ it is the densly-defined operator $\Laplace: D \subset L^2(M) \to L^2(M)$ given by $\Laplace = -\text{div}_g \circ \nabla_g$.
Here $\nabla_g$ is the gradient operator on $M$ and $\text{div}_g$ is the divergence operator on $M$, both with respect to the Riemannian metric $g_{\omega,J}$.

Under specific circumstances, the Sobolev space $W^{2k,2}(M)$ embeds continuously into the space $C^m(M)$ of $m$-times continuously differentiable functions on $M$.
This is made precise by the Sobolev embedding theorem. 
See e.g.~\cite[Thm. 7.1]{grigoryan-2012} for a proof of this theorem.
\begin{theorem}[Sobolev embedding theorem]\label{thm:sobolev-emb}
    For any $m \geq 0$ and $k > \frac{m}{2} + \frac{\dim M}{4}$ 
    there is a continuous embedding
    \begin{equation*}
        W^{2k,2}(M) \into C^m(M).
    \end{equation*}
\end{theorem}
In the following, we will be interested in time-dependent functions on $M$.
Thus, we introduce the densly-defined operator $\Laplace_t$ given by 
\[\Laplace_t \coloneqq \Laplace - \partial_t^2: D \subset L^2([0,1] \times M) \to L^2([0,1] \times M).\]
With this in place we obtain the Sobolev space 
\[W^{2k,2}([0,1] \times M) \coloneqq \{u: u, \Laplace_t u, \Laplace_t^2 u, \dots, \Laplace_t^k u \in L^2([0,1] \times M)\}\]
with Sobolev norm 
\begin{equation*}
\norm{u}_{W^{2k,2}}^2 \coloneqq \sum_{l = 0}^{k} \norm{\Laplace_t^l u}_{L^2}^2.
\end{equation*}
The Sobolev embedding theorem still gives us a continuous embedding of the following form:
\begin{corollary}
    For any $m \geq 0$ and $k > \frac{m}{2} + \frac{\dim M + 1}{4}$ 
    there is a continuous embedding
    \begin{equation*}
        W^{2k,2}([0,1] \times M) \into C^m([0,1] \times M).
    \end{equation*}
\end{corollary}
\subsection{Gaussian processes}
\label{subsec:gaussian-processes}
In this section we will review the necessary background on Gaussian processes.
A reader well-versed in the subject is advised again to skip this section.
We will only discuss those parts of the theory necessary for the present text.
We refer to~\cite[Part I]{adler-taylor-2007} and~\cite[Chapter 2]{gine-nickl-2016} for more details with a focus on Gaussian processes
and to~\cite{bogachev-1996} and~\cite{kuo-1975} for a more general treatment of Gaussian measures.
\subsubsection{Gaussian measures}
We will first recall the definition of the standard normal distribution on $\Reals$.
On $\Reals$ a Borel measure $\gamma_{\mu,\sigma^2}$ is called a \textit{Gaussian measure} with mean $\mu \in \Reals$ and variance $\sigma^2 > 0$ if it has the density
\[\frac{1}{\sqrt{2\pi \sigma^2}} \exp\left(-\frac{(x - \mu)^2}{2\sigma^2}\right)\]
with respect to the Lebesgue measure on $\Reals$.
Additionally, the Dirac measure $\delta_x$ at a point $x \in \Reals$ is also considered a Gaussian measure with mean $\mu = x$ and variance $\sigma^2 = 0$.
If $\gamma_{\mu,\sigma^2}$ is not a Dirac measure, i.e.\ when $\sigma^2 > 0$, then it is called a \textit{non-degenerate} Gaussian measure.
Let $(\Omega, \mathcal{F}, \Prob)$ be a probability space.
A random variable $X: \Omega \to \Reals$ is called a Gaussian random variable with mean $\mu \in \Reals$ and variance $\sigma^2 > 0$ if the law of $X$ (i.e. the push-forward measure $X_\sharp \Prob$) is the Gaussian measure $\gamma_{\mu,\sigma^2}$.
We also say that $X$ is normally distributed.
In this case, we write $X \sim N(\mu,\sigma^2)$.
If $X$ is a Gaussian random variable, then the expectation $\E[X]$ is equal to $\mu$ and the variance $\Var[X]$ is equal to $\sigma^2$.

We can extend this definition to a Gaussian measure on $\Reals^n$.
A Borel measure $\gamma_{\mu,\Sigma}$ on $\Reals^n$ is called a \textit{Gaussian measure} with mean $\mu \in \Reals^n$ and covariance matrix $\Sigma \in \GL_n(\Reals)$ such that $\Sigma$ is symmetric and positive-semi-definite, if for any linear map $A: \Reals^n \to \Reals$ the push-forward measure $A_\sharp \gamma_{\mu,\Sigma}$ is a Gaussian measure on $\Reals$ with mean $A\mu$ and variance $A\Sigma A^T$, where we identify $A$ with its matrix representation with respect to the standard basis of $\Reals^n$.
A random vector $X: \Omega \to \Reals^n$ is called a Gaussian random vector with mean $\mu \in \Reals^n$ and covariance matrix $\Sigma \in \GL_n(\Reals)$ if the law of $X$ (i.e. the push-forward measure $X_\sharp \Prob$) is the Gaussian measure $\gamma_{\mu,\Sigma}$.
We also say that $X$ has a multivariate normal distribution.
It is easy to see that $\E[X] = \mu$ and 
$Cov[X_i,X_j] = \Sigma_{ij}$ for $i,j = 1,\ldots,n$.
If $\Sigma$ is the identity matrix and $\mu = 0$, then we say that $X$ is a standard Gaussian random vector.

We can now give the general definition of a Gaussian measure on a locally convex topological vector space.
In particular, this extends the concept of a Gaussian measure to spaces such as $L^p(M), W^{k,p}(M), C^k(M)$, and $C^\infty(M)$ where $M$ is a smooth manifold.
\begin{definition}\label{def:gaussian-measure}
Let $X$ be a locally convex topological vector space. Then a Borel measure $\gamma$ on $X$ is called a \textit{Gaussian measure} if for any continuous linear functional $L \in X^*$ the push-forward measure $L_\sharp \gamma$ is a Gaussian measure on $\Reals$.
Additionally, we say that $\gamma$ is centered if for any continuous linear functional $L \in X^*$ the push-forward measure $L_\sharp \gamma$ has mean zero.
\end{definition}
\subsubsection{Elementary properties of Gaussian processes}
We will focus on Gaussian measures that are arise from so-called Gaussian processes.
Intuitively, a Gaussian process is a random function on some index set.
For us this index set will always be a manifold.
\begin{remark}
    The term Gaussian process originates from the study of processes indexed by time.
    Some authors, e.g.~\cite{adler-taylor-2007}, use the term Gaussian field for processes indexed by a higher-dimensional space.
    While this is more linguistically sensible than \textit{Gaussian process}, we will stick to the term for two reasons: (1) it is more common in the literature, and (2) there might be some confusion given the ample existing meanings of \textit{field} in geometry.
\end{remark}
To specify a Gaussian process on a (Riemannian) manifold $M$, it suffices to specify a map $\mu: M \to \Reals$ and a covariance function $\kappa: M \times M \to \Reals$, such that for any finite collection of points $x_1,\ldots,x_n \in M$ the matrix $\Sigma \in \GL_n(\Reals)$ defined by $\Sigma_{ij} = \kappa(x_i,x_j)$ is symmetric and positive-semidefinite.
Then by an elementary application of the Kolmogorov extension theorem (see e.g.~\cite[Theorem 12.1.2]{dudley-2002}), there exists a family $(f_p)_{p \in M}$ of random variables $f_p: \Omega \to \Reals$ defined on some probability space $(\Omega, \mathcal{F}, \Prob)$ such that for any finite collection of points $x_1,\ldots,x_n \in M$ the random vector $(f_{x_1},\ldots,f_{x_n})$ has a multivariate normal distribution with mean $(\mu(x_1),\ldots,\mu(x_n))$ and covariance matrix $\Sigma$ defined by $\Sigma_{ij} = \kappa(x_i,x_j)$. We call this family of random variables a \textit{Gaussian (random) process} on $M$ with mean $\mu$ and covariance function $\kappa$.
If $\mu$ is identically zero (which is the case we are most interested in), then we say that the Gaussian process is \textit{centered}.
\begin{example}\label{ex:white-noise}
Let $M$ be any manifold and let $\mu: M \to \Reals$ be identically zero and $\kappa: M \times M \to \Reals$ be the indicator function of the diagonal, i.e. $\kappa(x,y) = 1$ if $x = y$ and $\kappa(x,y) = 0$ otherwise.
Then the Gaussian process $(f_p)_{p \in \Reals}$ is simply a collection of independent standard Gaussian random variables.
Thus, it encodes absolutely no information about $M$. Such a process is called \textit{white noise} on $M$.
\end{example}
\begin{example}\label{ex:gaussian-process-ex1}
    Let $M = \Reals$ and let $\mu: \Reals \to \Reals$ be identically zero and $\kappa: \Reals \times \Reals \to \Reals$ be the covariance function given by
    \[\kappa(x,y) = \exp\left(-\alpha{(x - y)^2}\right),\]
    where $\alpha > 0$ is a fixed constant.
    Then the Gaussian process $(f_p)_{p \in \Reals}$ obtained from $\mu$ and $\kappa$ is a much nicer process than the white noise considered in the previous example.
    Indeed, one can show that its sample paths, i.e.\ the functions $p \mapsto f_p(\omega)$ for $\omega \in \Omega$, are almost surely continuous (indeed, almost surely smooth).
    We should note that $\kappa(x,y)$ only depends on the distance between $x$ and $y$.
    Such a process is called \textit{stationary}.
    The parameter $\alpha$ controls (intuitively speaking) how quickly the values of $f$ at points become uncorrelated as the distance between them increases.
    If $\alpha$ is large, then even points where $\abs{x-y}$ is small will be almost uncorrelated.
    If $\alpha$ is small, then sample paths of $f$ will only change very slowly as we move along $\Reals$.
    See Figure~\ref{fig:gaussian-process-ex1} for the typical sample paths of this process for different values of $\alpha$.
    \begin{figure}[htbp]
        \centering
        \includegraphics[width=0.95\textwidth]{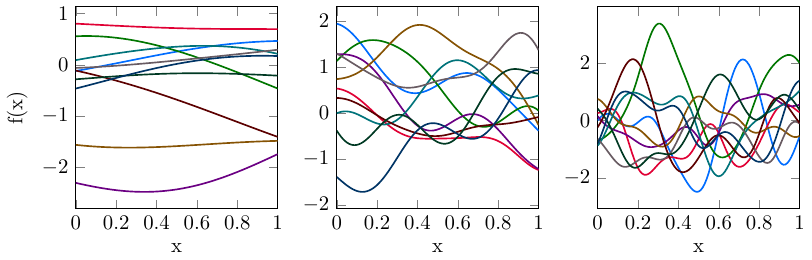}
        \caption{Sample paths of a Gaussian process described in Example~\ref{ex:gaussian-process-ex1}. The values of $\alpha$ are (from left to right): $1/2, 10$ and $50$.}\label{fig:gaussian-process-ex1} 
    \end{figure}
\end{example}
\begin{definition}\label{def:gaussian-process}
    Let $M$ be any index set\footnote{We use the letter $M$ here to emphasize that we are mostly interested in the manifold case. In particular, the $T$ most commonly found in the literature is too closely associated with \textit{time}.}.
    A family of random variables $(f_p)_{p \in M}: \Omega \to \Reals$ defined on a common probability space $(\Omega, \mathcal{F}, \Prob)$ is called a \textit{Gaussian process} if for any $n \in \Nats$, any $a_1,\ldots,a_n \in \Reals$, and any $p_1,\ldots,p_n \in M$ the random variable
    $\sum_{i=1}^n a_i f_{p_i}$
    is a (possibly degenerate) Gaussian random variable.
    The Gaussian process is called \textit{centered} if for any $p \in M$ the random variable $f_p$ has mean zero.
\end{definition}
Henceforth, we will usually use the notation $f(p)$ instead of $f_p$ to denote the random variable associated to the point $p \in M$.
This is since we seldom need to emphasize the probability space on which the process is defined
and are usually more interested in treating $f$ as a random function on $M$.
\define{
    Let $f_1,f_2$ be two Gaussian processes on an index set $M$.
    Then we say that $f_1$ and $f_2$ are \textit{versions of each other} if for any $p_1,\ldots,p_n \in M$ the random variables $(f_1(p_1),\ldots,f_1(p_n))$ and $(f_2(p_1),\ldots,f_2(p_n))$ have the same law.
    If $f_1$ and $f_2$ are defined on the same probability space $(\Omega, \mathcal{F}, \Prob)$, then we say that $f_1$ and $f_2$ are \textit{strict versions of each other}
    if $\Prob[f_1 = f_2] = 1$.
}
\rmk{
    Note that being a strict version of a Gaussian process is a much stronger condition than being a version of it.
    In particular, two Gaussian processes can be versions of each other while being independent.
}
It is easy to see that a Gaussian process constructed from a mean function $\mu: M \to \Reals$ and a covariance function $\kappa: M \times M \to \Reals$ as described above fits Definition~\ref{def:gaussian-process}.
Further, given a Gaussian process $f$ on $M$, we can construct the mean function $\mu: M \to \Reals$ and the covariance function $\kappa: M \times M \to \Reals$ by
\[\mu(p) = \E[f(p)]\]
and
\[\kappa(p,q) = \Cov[f(p),f(q)] = \E[(f(p) - \mu(p))(f(q) - \mu(q))].\]
We can then reconstruct the Gaussian process $f$ from $\mu$ and $\kappa$ as described above.
Note that when doing so, we obtain an independent version of $f$ since the covariance function 
only encodes the covariance structure within $f$ and not any further dependencies that might exist.
However, the law of $f$ is completely determined by $\mu$ and $\kappa$.
Note that for a centered Gaussian process, i.e.\ $\mu \equiv 0$, the covariance function $\kappa$ completely determines the behavior of the process.

In practice, we are most often interested in Gaussian processes with samples that lie almost-surely in some Banach space 
of functions on $M$.
Note that if $B$ is a Banach space of functions on $M$, and $f$ is a Gaussian process on $M$ with samples 
that lie almost-surely in $B$, then there exists a version of $f$ such that its samples lie in $B$ outright.
This gives rise to the following lemma:
\begin{lemma}\label{lem:gaussian-measure-from-GP}
    Let $B$ be a Banach space of functions on $M$, and let $f$ be a Gaussian process on $M$ with samples that lie almost-surely in $B$. 
    Then, there exists a version $\tilde{f}$ of $f$ such that $\tilde{f}$ is a Gaussian measure on $B$.
\end{lemma}
Later it will be necessary to compute the probability of the supremum of the absolute value of a Gaussian process being larger than some value $a > 0$.
The following example illustrates how this can be done in a simple case.
\begin{example}\label{ex:gaussian-process-ex2}
    Let $M = S^1$ be the standard circle identified with the interval $[0,1]$ with the endpoints identified.
    Then let 
    \begin{equation*}
        f(t) = \sum_{n=1}^\infty \frac{1}{n^2} X_n \sin(2\pi \cdot n \cdot t),
    \end{equation*}
    where $(X_n)_{n \in \Nats}$ is a sequence of independent standard Gaussian random variables.
    Then the process $f$ is a centered Gaussian process on $S^1$.
    We can now try to estimate the probability that the supremum of $\abs{f}$ (i.e.\ the $\infty$-norm of $f$) is larger or smaller than some value $a > 0$.
    First note that by the triangle inequality we have
    \begin{align*}
        \norm{f}_{\infty} &= \sup_{t \in S^1} \abs{f(t)} \leq \sum_{n=1}^\infty \frac{1}{n^2} \abs{X_n},
    \end{align*}
    since $\norm{\sin}_\infty = 1$.
    Let $\ell = \frac{a}{\zeta(\frac{3}{2})}$, where $\zeta$ is the Riemann zeta function.
    If we assume that $X_n \in (-\ell \sqrt{n}, \ell \sqrt{n})$ for all $n \in \Nats$, then we have
    \begin{align*}
        \norm{f}_{\infty} &\leq \sum_{n=1}^\infty \frac{1}{n^2} \abs{X_n} \leq \frac{a}{\zeta(\frac{3}{2})}\sum_{n=1}^\infty {n^{-\frac{3}{2}}}  = a.
    \end{align*}
    Thus, 
    \begin{align*}
    \Prob\left[\norm{f}_{\infty} \leq a\right] &\geq \Prob\left[\abs{X_n} \leq \ell \sqrt{n} \text{ for all } n \in \Nats\right]
    = \prod_{n=1}^\infty \Prob\left[\abs{X_n} \leq \ell \sqrt{n}\right] \\
    &= \prod_{n=1}^\infty \frac{1}{\sqrt{2\pi}} \int_{-\ell \sqrt{n}}^{\ell \sqrt{n}} \exp\left(-\frac{x^2}{2}\right) dx,
    \end{align*}
    where we use the independence of the $X_n$.
    Since the integrals in the product rapidly converge to $1$ as $n \to \infty$, 
    it is easy to see that the right-hand side is strictly positive and mainly depends on the first few $X_n$.
    Thus, $\norm{f}_{\infty}$ almost behaves like the absolute value of a Gaussian random variable itself.
    This is true more generally for Gaussian processes on compact manifolds, see e.g.~\cite[Chapter 2]{adler-taylor-2007} for more details.
    The following lemma, known as the \textit{Borell-TIS} inequality, is a more general tool for such estimates.
\end{example}
\begin{lemma}[Borell-TIS inequality]\label{lem:Borell-TIS}
    Let $f$ be a centered Gaussian process on a compact manifold $M$ that is bounded almost-surely.
    Then 
    \begin{equation*}
        \E\left[ \sup_{p \in M} f(p) \right] < \infty
    \end{equation*}
    and for any $u > 0$,
    \begin{equation*}
        \Prob\left[\sup_{p \in M} f(p) - \E\left[\sup_{p \in M} f(p)\right] > u\right] \leq \exp\left(-\frac{u^2}{2\sup_{p \in M} \Var\left[f(p)\right]}\right).
    \end{equation*}
\end{lemma}
We refer the reader to~\cite[Theorem 2.1.1]{adler-taylor-2007} for a proof. 
We note the following useful variant of the above inequality for the $\infty$-norm:
\begin{lemma}\label{lem:Borell-TIS-infinity}
    Let $f$ be a (possibly non-centered) Gaussian process on a compact manifold $M$ that is bounded almost-surely.
    Let $\sigma^2 = \sup_{p \in M} \Var\left[f(p)\right]$.
    Then for any $u > 0$, we have
    \begin{align*}
        \Prob\left[ \norm{f - \E[f]}_\infty \ge \E\left[\norm{f- \E[f]}_\infty\right] + u\right] &\leq \exp\left(-\frac{u^2}{2\sigma^2}\right) \\
        \Prob\left[ \norm{f- \E[f]}_\infty \le \E\left[\norm{f- \E[f]}_\infty\right] - u\right] &\leq \exp\left(-\frac{u^2}{2\sigma^2}\right) \\
        \Prob\left[ \abs{\norm{f- \E[f]}_\infty- \E\left[\norm{f- \E[f]}_\infty\right]} \ge u\right] &\leq 2\exp\left(-\frac{u^2}{2\sigma^2}\right).
    \end{align*}
\end{lemma}
See e.g.~\cite[Theorem 2.5.8]{gine-nickl-2016} for a proof.
Intuitively, the supremum or $\infty$-norm of a Gaussian process on a compact manifold depends on infinitely many independent random variables.
This can be seen explicitly using an orthogonal expansion of the process, see~\cite[Chapter 3]{adler-taylor-2007}.
Thus, the ``concentration of measure'' phenomenon articulated by Talagrand in~\cite{talagrand-1996} as ``\textit{[a] random variable that depends (in a smooth) way on many independent variables (but not too much on any of them) is essentially constant}''
applies and is expressed in the above inequalities.

We conclude these remarks with a brief discussion of the \textit{reproducing kernel Hilbert spaces} (RKHS) associated to Gaussian processes. 
Any Gaussian process defined on any index set gives rise to a Hilbert space called the \textit{reproducing kernel Hilbert space} (RKHS) of the Gaussian process.
Sometimes this space is also called the \textit{Cameron-Martin space}.
We will only give a very brief introduction and
refer the reader to~\cite[Chapter 2]{gine-nickl-2016} and~\cite[Chapter 3]{adler-taylor-2007} for a more detailed discussion of reproducing kernel Hilbert spaces.
In the following, let $f$ be a centered Gaussian process on a compact manifold $M$.
Let $(\Omega, \mathcal{F}, \Prob)$ be the underlying probability space of the Gaussian process.
Then consider the vector space
\begin{equation*}
    F = \left\{ \sum_{k=1}^m a_k f(p_k) \Big\vert m \in \Nats, a_1, \dots, a_m \in \Reals, p_1, \dots, p_m \in M \right\} \subset L^2(\Omega, \mathcal{F}, \Prob).
\end{equation*}
This space consists of all finite linear combinations of evaluations of the Gaussian process $f$ at points in $M$.
We denote the closure of $F$ in $L^2(\Omega, \mathcal{F}, \Prob)$ by $\bar{F}$.
We can now define the RKHS $\Hil$ of the Gaussian process $f$ as follows:
\begin{definition}\label{def:RKHS}
    Let $f$ be a centered Gaussian process on a compact manifold $M$. The Hilbert space 
    \begin{equation*}
        \Hil = \left\{ p \mapsto \E\left[h \cdot f(p)\right] \mid h \in \bar{F}\right\}
    \end{equation*}
    endowed with the inner product
    \begin{equation*}
        \langle p \mapsto \E\left[h_1 \cdot f(p)\right], p \mapsto \E\left[h_2 \cdot f(p)\right] \rangle_\Hil = \E\left[h_1h_2\right]
    \end{equation*}
    is called the \textit{reconstructing kernel Hilbert space} or \textit{RKHS} of the Gaussian process $f$.
\end{definition}
\rmk{
    It is clear from the definition that $\Hil$ consists of functions on $M$.
    Roughly speaking, the RKHS $\Hil$ consists of all functions on $M$ that have the same smoothness properties as the covariance function $\kappa(p,q) = \Cov[f(p),f(q)]$.
}
\rmk{
    The name \textit{reconstructing} kernel Hilbert space was not chosen randomly. 
    Let $p \in M$ be any point on $M$.
    Note that the random variable $f(p)$ is an element of $F$.
    It follows that $\kappa(p,\cdot) = \Cov[f(p),f(\cdot)] = \E[f(p)f(\cdot)]$, 
    where we use that $f$ is centered.
    Thus, $\kappa(p, \cdot) \in \Hil$.
    Now let $u \in \Hil$ be arbitrary.
    By definition of $\Hil$, there exists $h_u \in \bar{F}$ such that $u(p) = \E[h_u \cdot f(p)]$.
    Thus, 
    \begin{align*}
        \langle u, \kappa(p,\cdot) \rangle_\Hil &= \E\left[h_u f(p) \right] = u(p).
    \end{align*}
    This property is called the \textit{reconstructing property} of the RKHS.
}
We are often interested in the case where the samples of $f$ lie in some separable Banach space $B$ of functions on $M$, e.g. $B = L^2(M), C^0(M), W^{k,p}(M)$.
Note that by Lemma~\ref{lem:gaussian-measure-from-GP}, the law of $f$ induces a Gaussian measure on $B$.
We will  now give an alternative description of the RKHS in this setting.
We define 
\begin{equation*}
    G = \left\{ g(f) \mid g \in B^*\right\} \subset L^2(\Omega, \mathcal{F}, \Prob)
\end{equation*}
and denote by $\bar{G}$ the closure of $G$ in $L^2(\Omega, \mathcal{F}, \Prob)$.
Using this set, we can define the RKHS $\Hil$ as follows:
\begin{equation*}
    \Hil = \left\{ \E\left[h \cdot f\right] \mid h \in \bar{G}\right\} \subset B
\end{equation*}
endowed with the inner product
\begin{equation*}
    \langle \E[h_1 f], \E[h_2 f] \rangle_\Hil = \E\left[h_1h_2\right].
\end{equation*}
\rmk{
    Showing that these two definitions of the RKHS are equivalent is beyond the scope of this introduction. 
    However, it can be seen relatively easily if $B = C^0(M)$ (which is the most important case for us).
    Recall that the topological dual $C^0(M)^*$ is given by the space of $\Reals$-valued Radon measures on $M$.
    The elements of $F$ can be seen as linear combinations of Dirac measures in that language and thus are also included 
    in $C^0(M)^*$.
    However, the span of Dirac measures at arbitrary points is dense in $C^0(M)^*$ in the weak-* topology.
    Now let $\nu \in C^0(M)^*$ be arbitrary and let $(\nu_n)_{n \in \Nats}$ be a sequence of linear combinations of Dirac measures that converges to $\nu$ in the weak-* topology.
    Clearly, the random variables $\nu_n(f)$ converge to $\nu(f)$ in $L^2(\Omega, \mathcal{F}, \Prob)$ by the definition of weak-* convergence.
    Note that $\nu_n(f)$ lies in $F$ for all $n \in \Nats$.
    Since any linear combination of Dirac measures trivially lies in $G$, 
    we obtain that $\bar{F} = \bar{G}$. 
}
The relevance of the RKHS for us lies primarily in the next two results.
\begin{lemma}\label{lem:RKHS-and-support}
    Let $B$ be a separable Banach space of functions on $M$ and $f$ a centered Gaussian process on $M$ with samples that lie almost-surely in $B$.
    Then the closure of the RKHS $\Hil$ in $B$ (with respect to the norm of $B$) is the support of the law of $f$.
\end{lemma}
We refer the reader to~\cite[Corollary 2.6.17]{gine-nickl-2016} for a proof.
So far, we have only needed the RKHS as a set.
However, the next theorem shows the usefulness of the RKHS norm.
The following result, known as the \textit{Cameron-Martin theorem}, is a fundamental result in the theory of Gaussian processes.
It describes how the law of a Gaussian process behaves under translation. We refer the reader to~\cite[Theorem 2.6.13]{gine-nickl-2016} for a proof.
\begin{theorem}[Cameron-Martin theorem]\label{thm:cameron-martin}
    Let $B$ be a separable Banach space of functions on $M$ and $f$ a centered Gaussian process on $M$ with samples that lie almost-surely in $B$.
    Let $\mu$ be the law of $f$ on $B$ and let $\Hil$ be the RKHS.
    Then for any $h \in B$, the push-forward measure $\mu_h$ of $\mu$ under translation by $h$ is absolutely continuous with respect to $\mu$ if and only if $h \in \Hil$.
    Note that $\mu_h(A) = \mu(\{x \in B \mid x + h \in A\})$ for any measurable set $A \subset B$.
    Furthermore, the Radon-Nikodym derivative $\frac{d\mu_h}{d\mu}$ can be given explicitly. Let $h \in \Hil$ and let $g \in \bar{G}$ be such that $h = \E[g f]$, 
    then
    \begin{equation*}
        \frac{d\mu_h}{d\mu}(\cdot) = \exp\left(g(\cdot) - \frac{1}{2} \|h\|_\Hil^2\right).
    \end{equation*}
\end{theorem}
\rmk{
    It is useful to consider an example to fully grasp the relevance of the RKHS in this context.
    We might think of a Gaussian process with samples in $L^2([0,1])$ such that the samples are almost-surely smooth.
    In this case, the RKHS will consist only of smooth functions.
    Then, if we take a non-smooth $L^2$-function $v$, the samples of the push-forward Gaussian process under translation by $v$ will
    almost-surely not be smooth.
    Thus, the two measures have to be mutually singular.
    If $v$ is smooth, then the samples of the push-forward Gaussian process under translation by $v$ will still be smooth, 
    and the two measures might be absolutely continuous with respect to each other.
    However, if the ``typical'' smoothness of a sample path of $f$ and that of $v$ do not agree, 
    the measures will still be mutually singular. The next example illustrates this.}
\begin{example}\label{ex:RKHS-example}
    Fix $\eps > 0$.
    Let $M = S^1$ and let $f$ be a Gaussian process on $M$ defined by 
    \begin{align*}
        f(t) = X_0 + {\sqrt{2}}\sum_{k=1}^\infty \exp\left(-\eps\pi^2k^2\right) \cdot \left( X_k^{(1)} \cos(2\pi k t) + X_k^{(2)} \sin(2\pi k t)\right),
    \end{align*}
    where $X_0,X_1^{(1)},X_1^{(2)},X_2^{(1)},X_2^{(2)},\ldots$ are i.i.d.\ standard normal random variables.
    Due to the almost-surely exponential decay of the Fourier coefficients, we immediately see that $f$ lies in $L^2(S^1)$ almost surely.
    Using the above with $B = L^2(S^1)$, we find 
    \begin{equation*}
        G = \left\{ \langle g,f \rangle_{L^2} \mid g \in L^2(S^1)\right\} \subset L^2(\Omega, \mathcal{F}, \Prob)
    \end{equation*}
    where we use the fact that $L^2(S^1)$ is its own topological dual.
    This greatly simplifies the following computations.
    Note that $\{1\} \cup \{\sqrt{2}\cos(2\pi k \cdot), \sqrt{2}\sin(2\pi k \cdot) \}_{k \in \Nats}$
    is an orthonormal basis of $L^2(S^1)$.
    Thus, 
    \begin{align*}
        G &= \left\{ \left\langle a_0 + \sum_{k=1}^\infty a_k \sqrt{2}\cos(2\pi k \cdot) + b_k \sqrt{2}\sin(2\pi k \cdot),f \right\rangle_{L^2} \Big\vert \{a_k\}_{k \in \Nats}, \{b_k\}_{k \in \Nats_{>0}}, \sum_{k \geq 1}{a_k^2 + b_k^2} < \infty \right\} \\
        &= \left\{ a_0 X_0 + \sum_{k=1}^\infty a_k X_k^{(1)} + b_k X_k^{(2)} \Big\vert \{a_k\}_{k \in \Nats}, \{b_k\}_{k \in \Nats_{>0}}, \sum_{k \geq 1}{a_k^2 + b_k^2} < \infty \right\}.
    \end{align*}
    Notice that this set is already closed in $L^2(\Omega, \mathcal{F}, \Prob)$, thus $\bar{G} = G$.
    Further note that for any $\{a_k\},\{b_k\} \in \ell^2$, 
    we have
    \begin{align*}\E&\Bigg[ \left(a_0 X_0 + \sum_{k=1}^\infty a_k X_k^{(1)} + b_k X_k^{(2)} \right) \left(X_0 + {\sqrt{2}}\sum_{k=1}^\infty e^{-\eps\pi^2k^2} \cdot \left( X_k^{(1)} \cos(2\pi k \cdot) + X_k^{(2)} \sin(2\pi k \cdot)\right)\right) \Bigg]\\
        &=a_0 + {\sqrt{2}}\sum_{k=1}^\infty \exp\left(-\eps\pi^2k^2\right) \cdot \left( a_k \cos(2\pi k \cdot) + b_k \sin(2\pi k \cdot)\right),
    \end{align*}
    where we use the fact that $X_k^{(1)}$ and $X_k^{(2)}$ are independent and have variance $1$.
    With this it becomes easy to compute the RKHS $\Hil$ of $f$:
    \begin{align*}
        \Hil &= \left\{ \E[h f] \mid h \in \bar{G} \right\}\\ 
        &= \left\{
            a_0 + \sum_{k=1}^\infty  a_k \cos(2\pi k \cdot) + b_k \sin(2\pi k \cdot)
            \Big\vert \{a_k\},\{b_k\} \in \ell^2,\sum_{k=1}^\infty {\sqrt{2}}e^{2\eps\pi^2k^2} (a_k^2 + b_k^2) < \infty\right\}.
    \end{align*}
    It is easy to see that the set $\Hil \subset L^2(S^1)$ is dense.
    However, $\Hil$ only contains \textit{very} smooth functions, namely function with Fourier coefficients that decay faster than $e^{-\eps \pi^2 k^2}$ as $k \to \infty$.
    Intuitively, the RKHS norm penalizes higher modes exponentially. 
    Thus, even though the reproducing kernel Hilbert space $\Hil$ is dense in $L^2(S^1)$ and consists purely of smooth functions,
    it does not contain all smooth functions.
    For example the function $g(s) = \sum_{k=1}^\infty k^{-\log(k)} \sin(2\pi k s)$ is not contained in the RKHS $\Hil$,
    since its Fourier coefficients decay slower than $e^{-\eps \pi^2 k^2}$ for any $\eps > 0$.
    Since these coefficients decay faster than any polynomial, the function $g$ however is smooth.
\end{example} %
\section{Construction of the Measure}
\label{sec:construction}
We will begin with the construction of a Gaussian measure on $C^\infty_0$.
As a first step, we carefully lay out all the choices that constitute a \ldd .

\subsection{Law-defining data}
Henceforth, we will assume that we have a symplectic manifold $(M,\omega)$, which is closed and connected.
We will now fix a universal probability space $(\Omega, \Sigma, \Prob)$
and henceforth all Gaussian processes will be assumed (unless explicitly stated otherwise) to be defined on this space.
\begin{definition}\label{def:ldd}
    A \textit{\ldd} $\mathcal{D}$ is a tuple $(\reg, J, \{e_n\}_{n \in \Nats_{>0}},  \{Z_n\}_{n \in \Nats_{>0}})$ such that 
    \begin{enumerate}
        \item $\reg > 0$, called the \textit{regularity parameter}, is a positive real number;
        \item $J$ is an $\omega$-compatible almost-complex structure on $M$;
        \item $\{\frac{1}{\vol_g(M)}, e_1, e_2, \ldots\}$ is an orthonormal basis of $L^2(M,g)$ and any $e_n$ is an eigenfunction of the Laplace-Beltrami operator $\Laplace_g: W^{2,2}(M, g) \to L^2(M, g)$ associated to the metric $g(\cdot,\cdot)=\omega(\cdot, J \cdot)$.
        Further, let $0 = \lambda_0 < \lambda_1 \leq \lambda_2 \leq \ldots < \infty$ be the eigenvalues of $\Laplace_g$ counted with multiplicities. Then $\{e_n\}_{n \in \Nats_{>0}}$ is ordered by increasing eigenvalue, i.e.\ $e_n$ is an eigenfunction with eigenvalue $\lambda_n$;
        \item $\{Z_n\}_{n \in \Nats_{>0}}$ is a family of Gaussian processes on $[0,1]$
        such that there exist constants $C,D \geq 0$ such that for all $n \in \Nats_{>0}$ and $t \in [0,1]$
        we have $\Cov[Z_n(t),Z_n(t)] < C$ and 
        $\E[\norm{Z_n}_\infty] < D$.
        We further require that $t_1, t_2 \mapsto \Cov[Z_n(t_1),Z_n(t_2)]$ and $t \mapsto \E[Z_n(t)]$ are smooth functions for any $n \in \Nats_{>0}$, which implies that the samples paths of $Z_n$ are almost-surely smooth.
        We say that the process $Z_n$ is \textit{associated to the eigenfunction $e_n$} in $\mathcal{D}$.
    \end{enumerate}
\end{definition}
\define{\label{def:ldd-types}
    A \ldd\ $\mathcal{D}$ is said to be
    \begin{enumerate}
        \item \textit{time-symmetric}, if for any $n \in \Nats_{>0}$ we have that the Gaussian processes $Z_n(\cdot)$ and $Z_n(1-\cdot)$
        are versions of each other; 
        \item \textit{centered}, if for any $n \in \Nats_{>0}$ we have that $Z_n$ is a centered Gaussian process.
    \end{enumerate}
    These are fairly general properties.
    We can also impose conditions to restrict what kind of Hamiltonian function in modelled by $\mathcal{D}$.
    We say that $\mathcal{D}$ is
    \begin{enumerate}\setcounter{enumi}{2}
        \item \textit{autonomous}, if for any $n \in \Nats_{>0}$ we have that the Gaussian process $Z_n$ satisfies for any $t_1, t_2 \in [0,1]$ that
        \[\Cov[Z_n(t_1), Z_n(t_2)] = \Var[Z_n(t_1)] = \Var[Z_n(t_2)],\]
        which, in particular, implies that the sample paths of $Z_n$ are almost-surely constant on $[0,1]$;
        \item \textit{periodic}, if for any $n \in \Nats_{>0}$ we have that $\Prob[D^k Z_n(0) = D^k Z_n(1)] = 1$ for all $k \in \Nats_0$, i.e., the Gaussian process $Z_n$ descends to a smooth function on $S^1$ almost-surely.
    \end{enumerate}
    The following conditions are related to the support of the resulting measure on $\Ham(M,\omega)$. We say that $\mathcal{D}$ is
    \begin{enumerate}\setcounter{enumi}{4}
        \item \textit{exhaustive}, if for any $n \in \Nats_{>0}$ we have that the RKHS of $Z_n$ is dense\footnote{Here we implicitly use the fact that functions in the RKHS are smooth (in particular continuous) since the covariance functions of all $Z_n$ are smooth.} in $(C^0([0,1],\R),\norm{\cdot}_{\infty})$;
        \item \textit{autonomously exhaustive}, if it is autonomous and for any $n \in \Nats_{>0}$ we have $\Var[Z_n] > 0$;
        \item \textit{periodically exhaustive}, if for any $n \in \Nats_{>0}$ we have that the closure of the RKHS of $Z_n$ in $(C^0([0,1],\R),\norm{\cdot}_{\infty})$ contains all smooth periodic functions.
    \end{enumerate}
    For technical reasons, we sometimes want to impose conditions on the relationship between the different Gaussian processes $Z_n$.
    The most useful restrictions are listed below.
    We say that $\mathcal{D}$ is
    \begin{enumerate}\setcounter{enumi}{7}
        \item \textit{frequency independent} if for any $n,m \in \Nats_{>0}$ we have that the Gaussian processes $Z_n$ and $Z_m$ are independent;
        \item \textit{frequency unbiased} if for any $n,m \in \Nats_{>0}$ we have that the Gaussian processes $Z_n$ and $Z_m$ are independent versions of each other;
        \item \textit{weakly frequency unbiased} if for any $n,m \in \Nats_{>0}$ there exists some $\alpha \in \Reals$ such that $Z_n$ and $\alpha \cdot Z_m$ are independent versions of each other.
    \end{enumerate}
    Finally, we need to impose conditions on $\mathcal{D}$ to ensure that the resulting measure has a well-defined limit as the regularity parameter $\reg$ tends to $0$.
    For this purpose we say that $\mathcal{D}$ is
    \begin{enumerate}\setcounter{enumi}{10}
        \item \textit{of $C^0$-limiting type} if it is centered, weakly frequency unbiased, and if $\sum_{n=1}^\infty \abs{\alpha_n}
        \cdot \norm{e_n}_\infty < \infty$, where $\alpha_n \in \Reals$ is such that $Z_n$ and $\alpha_n Z_1$ are versions of each other;
        \item \textit{of $C^1$-limiting type} if it is centered, weakly frequency unbiased, and if $\sum_{n=1}^\infty \abs{\alpha_n}
        \cdot \lambda_n \cdot \norm{e_n}_\infty < \infty$, where $\alpha_n \in \Reals$ is such that $Z_n$ and $\alpha_n Z_1$ are versions of each other.
    \end{enumerate}
}
\rmk{
    We should note that the definition of \ldd\ above is purposely broad.
    There is a canonical way to choose all of this data, see Section~\ref{sec:existence-of-ldd}.
    However, we do not want to artificially constrain ourselves.
    The reward for this is that our construction will encompass Dirac measures, measures supported on $\Aut$, and measures with full support, all within the same framework.
    For this generality, we have to make a sacrifice by introducing the avalanche of choices encoded in a \ldd.
}
\rmk{
    We have chosen to make the regularity parameter $\reg > 0$ a part of the \ldd\ and crucially non-random 
    since we are most interested in the existence and properties of probability measures on $\Ham(M,\omega)$.
    Since we also wish to study limits as $\reg \to \{0,\infty\}$, it seems unnatural to make $\reg$ a random variable.
    However, from a statistical point of view, this is the wrong decision.
    Thus, if one wants to adapt the construction towards doing 
    actual statistical inference on $\Ham(M,\omega)$, one should make $\reg$ a random variable.
    The reasons for this are outlined in~\cite{castillo-et-al-2014}.
} %
\subsection{Gaussian measure on $C^\infty_0$}\label{subsec:gaussian-measure-on-Hamiltonian-functions}
We will now construct a Gaussian process on the space of Hamiltonian functions.
This construction is based on the construction in~\cite{castillo-et-al-2014} and~\cite{nicolaescu-2014}.

In the following let $\mathcal{D}=(\reg, J, \{e_n\}_{n \in \Nats_{>0}},  \{Z_n\}_{n \in \Nats_{>0}})$ be a fixed \ldd.
Let 
\[\Laplace_g: W^{2,2}(M, g) \to L^2(M, g)\]
be the Laplace-Beltrami operator associated to the metric $g(\cdot, \cdot) = \omega(\cdot,J\cdot)$.
We denote the eigenvalues of $\Laplace_g$ by $0 = \lambda_0 < \lambda_1 \le \lambda_2 \le \cdots < \infty$,
where we count the eigenvalues with multiplicities.
Our \ldd\ gives us a preferred orthonormal basis of $L^2(M, g)$, namely ${\{e_n\}}_{n \in \Nats_{>0}}$.

We can then define our Gaussian process as a random combination of eigenfunctions of the Laplace-Beltrami operator with appropriate decay of the (time-dependent) coefficients. 
To encode the appropriate decay, we define the following series of weights:
\[w_n = \exp\left(-\frac{1}{2}\lambda_n \cdot \reg\right).\] 
Note that as $\reg \to \infty$ all $w_n$ converge to $0$ and as $\reg \to 0$ all $w_n$ converge to $1$.
Note that a faster decay might be encoded in the \ldd, see Definition~\ref{def:ldd-types}.
With this decay condition in place, we can define the Gaussian process that will serve as our model for a random Hamiltonian function, ultimately leading to a notion of random Hamiltonian diffeomorphisms:
\begin{equation}\label{eq:GP_H}
    H(t,x) = \sum _{n \ge 1} w_n \cdot Z_n(t) \cdot e_n(x).
\end{equation} 
\begin{lemma}\label{lem:GP_H}
    The random function $H: [0,1] \times M \to \Reals$ defined in~\eqref{eq:GP_H} is a Gaussian process.
    If $\mathcal{D}$ is centered, then $H$ is a centered Gaussian process.
\end{lemma}
\begin{proof}
    Since $H$ is by definition a series with Gaussian coefficients, it is a Gaussian process provided that the series converges almost-surely.
    Recall that we can estimate the supremum norm of eigenfunctions of the Laplacian as follows:
    \begin{equation}\label{eq:C0-estimate-eigenfunctions}
    \norm{e_n}_{\infty} \leq C' \cdot \norm{e_n}_{W^{\dim M,2}} \leq C'' \lambda_n^{\frac{\dim M}{2}} \cdot \norm{e_n}_{L^2}  = C'' \lambda_n^{\frac{\dim M}{2}},
    \end{equation}
    where $C',C'' > 0$ are constants that only depend on $(M,\omega)$ and $J$.
    Here we use the Sobolev embedding theorem (Theorem~\ref{thm:sobolev-emb}) and the definition of the Sobolev norm.
    When we combine~\eqref{eq:C0-estimate-eigenfunctions} 
    with the decay of the coefficients $w_n$ and the fact that the variances and means of $Z_n(t)$ are uniformly bounded in $n$ and $t$ by the assumptions on the \ldd\ $\mathcal{D}$,
    we obtain that the series converges almost-surely uniformly on $[0,1] \times M$.

    To obtain that $H$ is centered, if $\mathcal{D}$ is centered, we compute the mean of $H$ at any point $(t,x) \in [0,1] \times M$:
    \begin{align*}
        \E[H(t,x)] &= \E\left[\sum _{n \ge 1} w_n \cdot Z_n(t) e_n(x) \right] \\
                   &= \sum_{n \ge 1} w_n \cdot \E[Z_n(t)] \cdot e_n(x),
    \end{align*}
    where we again use the almost-sure convergence of the series to interchange the expectation and the sum.
    Clearly this quantity is $0$ if $\mathcal{D}$ is centered, since $\E[Z_n(t)] = 0$ for all $n \in \Nats_{>0}$ and $t \in [0,1]$  by assumption in that case.
\end{proof}
\begin{remark}
    Assume that $\mathcal{D}$ is not centered and let 
    \[\bar{H}(t,x) = \E[H(t,x)] = \sum _{n \ge 1} w_n \E[Z_n(t)] e_n(x).\]
    Then the process $H - \bar{H}$ is centered and can be obtained from~\eqref{eq:GP_H} for the centered $\ldd$
    \[\mathcal{D}' = (\reg, J, \{e_n\}_{n \in \Nats_{>0}},  \{Z_n - \E[Z_n]\}_{n \in \Nats_{>0}}).\]
    Thus, in the remainder of this section we will assume that $\mathcal{D}$ is centered.
\end{remark}

Next we wish to establish some requisite regularity properties of the Gaussian process $H$.
To define a Hamiltonian vector field, we need that $H$ is at least $C^1$-regular in the space variable $x \in M$.
However, we will actually show that $H$ is almost-surely smooth in both time and space.
First we will show that $H$ is almost-surely an $L^2$-function on $[0,1] \times M$.
\begin{lemma}\label{lem:GP_H-L2}
    The sample paths of the Gaussian process $H$ defined in~\eqref{eq:GP_H} are almost-surely in $L^2([0,1] \times M, \Reals)$.
    In particular, the law of $H$ induces a well-defined Gaussian measure on $L^2([0,1] \times M, \Reals)$.
\end{lemma}
\begin{proof}
    We want to show that $\norm{H}_{L^2}^2 < \infty$ almost-surely.
    Note that for any measurable $f:[0,1] \times M \to \Reals \cup \{\infty\}$,
    \[\norm{f}_{L^2}^2 = \int_{M \times [0,1]} f^2 d\vol_g dt = \int_0^1 \norm{f(t,\cdot)}_{L^2}^2 dt\]
    by the Fubini-Tonelli theorem.
    Now by combining the above with Parseval's identity, we obtain
    \begin{align*}
        \norm{H}_{L^2}^2 &= \int_0^1 \norm{H(t,\cdot)}_{L^2}^2 dt 
        = \int_0^1 \sum_{n \ge 1} \left(w_n \cdot Z_n(t) \right)^2 dt.
    \end{align*}
    We can compute the expectation of this quantity as follows:
    \begin{align*}
        \E\left[\norm{H}_{L^2}^2\right] = \E\left[\int_0^1 \sum_{n \ge 1} \left(w_n \cdot Z_n(t) \right)^2 dt\right]
        &= \int_0^1 \sum_{n \ge 1} \E\left[\left(w_n \cdot Z_n(t) \right)^2\right] dt,
    \end{align*}
    where we use the Fubini-Tonelli theorem again to interchange the expectation, the integral, and the sum.
    Furthermore, we have
    \[\E\left[\left(w_n \cdot Z_n(t) \right)^2\right] = w_n^2 \cdot \E[Z_n(t)^2] = w_n^2 \cdot \left( \Var[Z_n(t)] - \E[Z_n(t)]^2 \right).\]
    By the definition of a \ldd, we have that there exists a constant $C > 0$ such that $\Var[Z_n(t)] - \E[Z_n(t)]^2 < C$ for all $n \in \Nats_{>0}$ and $t \in [0,1]$.
    Combining this with our previous computations, we obtain
    \begin{align*}
        \E\left[\norm{H}_{L^2}^2\right]
        &= \int_0^1 \sum_{n \ge 1} w_n^2 \cdot \left( \Var[Z_n(t)] - \E[Z_n(t)]^2 \right) dt \\
        &\leq C \cdot \sum_{n \ge 1} w_n^2 
        = C \cdot \sum_{n \ge 1} \exp(-\lambda_n \cdot \reg) < \infty,
    \end{align*}
    where we use that $\lambda_n \to \infty$ sufficiently quickly as $n \to \infty$ by Weyl's law~\cite{weyl-1911}.
    Since this implies $\Prob[\norm{H}_{L^2} < \infty] = \Prob[\norm{H}^2_{L^2} < \infty] = 1$, we obtain that sought-after result.
\end{proof}
We have now established that $H$ gives rise to a well-defined Gaussian measure on $L^2([0,1] \times M, \Reals)$.
To use $H$ as a Hamiltonian function, we need to show two more things: (1) that $H$ is almost-surely smooth and (2) that $H$ is normalized, i.e.\ $\int_M H(t,x) d\vol_g(x) = 0$ for all $t \in [0,1]$.
\begin{lemma}\label{lem:GP_H-Sobolev}
    The sample paths of the Gaussian process $H$ defined in~\eqref{eq:GP_H} are almost-surely in $W^{2k,2}([0,1] \times M)$ for any $k \in \Nats$.
\end{lemma}
\begin{proof} 
    We wish to use the same proof strategy as in Lemma~\ref{lem:GP_H-L2}, i.e. 
    we wish to show that 
    \[ \E\left[\norm{H}_{W^{2k,2}}^2\right] < \infty,\]
    thereby obtaining that this quantity is finite almost-surely, which implies that the sample paths of $H$ are almost-surely in $W^{2k,2}([0,1] \times M)$.
    Recall that the Sobolev norm is defined as $\norm{H}_{W^{2k,2}}^2 = \sum_{j=0}^{k} \norm{\Laplace_t^j H}_{L^2}^2$.
    We can compute these expressions individually.
    First we note that in the $j = 0$ case we have already shown in Lemma~\ref{lem:GP_H-L2} that $\norm{H}_{L^2}^2 < \infty$ almost-surely.
    Thus, we can focus solely on $j > 0$. First we note that when $j = 1$, we have
    \begin{align*}
        \Laplace_t H(t,x) = (\Laplace_g - \partial_t^2) H(t,x) =  \sum_{n\ge 1} \exp\left(-{\frac{1}{2}\lambda_n}\right) (\lambda_n Z_n(t) - \partial_t^2 Z_n(t)) e_k(x),
    \end{align*}
    since $e_k$ is an eigenfunction of $\Laplace_g$ with eigenvalue $\lambda_n$.
    It is an elementary computation to see that this implies 
    \begin{align*}
        \Laplace_t^j H(t,x) = \sum_{n\ge 1} \exp\left(-{\frac{1}{2}\lambda_n}\right) \left[(\lambda_n - \partial_t^2)^j Z_n(t) \right] e_k(x),
    \end{align*}
    which in particular implies that $\norm{H}_{W^{2k,2}}^2 = \sum_{j=0}^{k} \norm{\Laplace_t^j H}_{L^2}^2$ is measurable.
    Here we use that the derivatives of a Gaussian process are again Gaussian processes (if they exist), see e.g.~\cite[Section 1.4.2]{adler-taylor-2007}. 
    Furthermore, this implies that 
    \begin{align*}
        \norm{\Laplace_t^j H}_{L^2}^2 &= \int_0^1 \norm{\Laplace_t^j H(t,\cdot)}_{L^2}^2 dt = \int_0^1 \sum_{n\ge 1} \exp\left(-\lambda_n\right) \left[(\lambda_n - \partial_t^2)^j Z_n(t) \right]^2 dt \\
        &= \sum_{n\ge 1}  \exp\left(-\lambda_n\right) \left( \int_{0}^{1} \left[(\lambda_n - \partial_t^2)^j Z_n(t) \right]^2 dt \right),
    \end{align*}  
    where we use Parseval's identity and the Fubini-Tonelli theorem.
    Now for any $t \in [0,1]$, we have 
    \begin{align*}
        \E\left[ ((\lambda_n - \partial_t^2)^j Z_n(t))^2\right] = \Var\left[ (\lambda_n - \partial_t^2)^j Z_n(t) \right] + \left( \E\left[ (\lambda_n - \partial_t^2)^j Z_n(t)\right]\right)^2.
    \end{align*}
    We can now estimate these two terms separately.
    First, we have that 
    \begin{align*}
        \Var\left[ (\lambda_n - \partial_t^2)^j Z_n(t) \right] &= \Var\left[ \sum_{i=0}^j {{j}\choose{i}} \lambda_n^i (-\partial_t^2)^{j-i} Z_n(t) \right] \\
        &= \sum_{i=0}^j\sum_{l=0}^j \Cov\left[ {{j}\choose{i}} \lambda_n^i (-\partial_t^2)^{j-i} Z_n(t), {{j}\choose{l}} \lambda_n^l (-\partial_t^2)^{j-l} Z_n(t)\right] \\
        &= \sum_{i=0}^j\sum_{l=0}^j {{j}\choose{i}} {{j}\choose{l}} \lambda_n^{i+l} {(-1)}^{2j-i-l} \Cov\left[(\partial_t^2)^{j-i} Z_n(t),  (\partial_t^2)^{j-l} Z_n(t)\right] \\
        &= \sum_{i=0}^j\sum_{l=0}^j {{j}\choose{i}} {{j}\choose{l}} \lambda_n^{i+l} {(-1)}^{2j-i-l} \partial_{t_1}^{2j-2i} \partial_{t_2}^{2j-2l} \kappa_n(t_1,t_2) \vert_{t_1=t,t_2=t}\\ 
        &\leq \sum_{i=0}^j\sum_{l=0}^j {{j}\choose{i}} {{j}\choose{l}} \lambda_n^{i+l} \norm{\kappa_n}_{C^{2j}} \\
        &\leq \sum_{i=0}^j\sum_{l=0}^j (j!)^2 \lambda_n^{2j} \norm{\kappa_n}_{C^{2j}} \leq j^2 \cdot (j!)^2 \cdot \lambda_n^{2j} \cdot \norm{\kappa_n}_{C^{2j}},
    \end{align*}
    where we use the fact that the covariance of derivatives of a Gaussian process is given by the derivatives of the covariance function, see e.g.~\cite[Section 5.5]{adler-taylor-2007}.
    We know that these derivatives exist (almost-surely) because the samples paths of $Z_n$ are almost-surely smooth (see Definition~\ref{def:ldd}).
    The second term can be estimated as follows:
    \begin{align*}
        \left( \E\left[ (\lambda_n - \partial_t^2)^j Z_n(t)\right]\right)^2 &= \left( (\lambda_n - \partial_t^2)^j\E\left[  Z_n(t)\right]\right)^2 \leq C_j \cdot \lambda_n^{2j} \cdot \norm{\E[Z_n]}_{C^{2j}},
    \end{align*}
    where we use the linearity of the expectation and the fact that $\E[Z_n(t)]$ is a smooth function by assumption on the \ldd\ $\mathcal{D}$. 
    Here $C_j > 0$ is some constant depending on only on $j$.
    Combining these two estimates, we obtain that
    \[ \E\left[ ((\lambda_n - \partial_t^2)^j Z_n(t))^2\right] \leq  \lambda_n^{2j} \cdot \left( j^2 (j!)^2 \norm{\kappa_n}_{C^{2j}} + C_j \norm{\E[Z_n]}_{C^{2j}} \right) \]
    Then by the definition a \ldd, there is a uniform bound (uniform meaning uniform in $n$) for $\norm{\kappa_n}_{C^{2j}}$ and $\norm{\E[Z_n]}_{C^{2j}}$.
    Thus, the above estimate simplifies to $\E\left[ ((\lambda_n - \partial_t^2)^j Z_n(t))^2\right] \leq C'_j \cdot \lambda_n^{2j}$,
    where $C'_j > 0$ is a constant that depends only on $j$.
    We can now combine this with our previous computations to obtain
    \begin{align*}
        \E\left[\norm{\Laplace_t^j H}_{L^2}^2 \right] &= \sum_{n\ge 1}  \exp\left(-\lambda_n\right) C'_j \cdot \lambda_n^{2j} < \infty,
    \end{align*}
    where we again use that $\lambda_n \to \infty$ as $n \to \infty$ by Weyl's law~\cite{weyl-1911}.
    It now suffices to note that 
    \begin{align*}
        \E\left[\norm{H}_{W^{2k,2}}^2\right] &= \sum_{j=0}^{k} \E\left[\norm{\Laplace_t^j H}_{L^2}^2\right] < \infty,
    \end{align*}
    to conclude the proof.
\end{proof} 

We can now utilize the Sobolev embedding theorem to show that $H$ is almost-surely smooth.
To formalize this, denote by $H_\theta$ the sample path of the Gaussian process $H$ associated to the element $\theta \in \Omega$ in the underlying probability space $(\Omega, \mathcal{F}, \Prob)$.
Thus, we can define the event 
\[\mathcal{S} \coloneqq \bigcap_{k = 1}^\infty \{\theta \in \Omega \mid H_{\theta} \in W^{2k,2}([0,1] \times M)\}.\]
If $\theta \in \mathcal{S}$ this means that $H_{\theta} \in W^{2k,2}$ for all $k \in \Nats$.
This implies $H_{\theta} \in C^\infty([0,1] \times M)$ by standard Sobolev embedding arguments, see Theorem~\ref{thm:sobolev-emb}.
Note that $\mathcal{S}$ is a countable intersection of full-measure sets, and thus has full measure itself.
Thus, we obtain the following corollary of Lemma~\ref{lem:GP_H-Sobolev}:
\begin{corollary}\label{cor:GP_H-smooth}
    With the notation from above, $\Prob[\mathcal{S}] =1$. Thus, a sample path of $H$ is smooth almost-surely.
\end{corollary}
\begin{remark}
    Indeed, we can get more. 
    By restricting to $\mathcal{S}$ we can obtain a random variable \textit{with} smooth samples, i.e.\ one valued in $C^\infty([0,1]\times M)$.
    Indeed, on the probability space $(\mathcal{S}, \mathcal{F}\vert_{\mathcal{S}}, \Prob\vert_{\mathcal{S}})$
    there is the random variable
    \[H\vert_{\mathcal{S}}:\mathcal{S} \to C^\infty([0,1] \times M, \Reals).\]
    In some ways whether we look at $H\vert_{\mathcal{S}}$ or $H$ 
    is simply a matter of convention or notational preference.
    Since $\Prob[\mathcal{S}] = 1$ the probability of any event\footnote{Or to be precise, the probability of the event and the probability of the intersection of the event and $\mathcal{S}$.} is unchanged by the choice we make.
\end{remark}
We now still have to show that $H$ is normalized, i.e.\ that $\int_M H(t,x) d\vol_g(x) = 0$ for all $t \in [0,1]$.
\begin{lemma}\label{lem:GP_H-normalized}
    The Gaussian process $H$ defined in~\eqref{eq:GP_H} is normalized in the sense that 
    \[\int_M H(t,x) d\vol_g(x) = 0\]
    for all $t \in [0,1]$ almost-surely. In particular, the law of $H$ induces a well-defined Gaussian measure on $C_0^\infty$.
\end{lemma}
\begin{remark}
    We should note that in the context of this section (and this section only!) we think of $C_0^\infty$ simply as $C_0^\infty \subset L^2([0,1] \times M, \Reals)$.
    In particular, we use the induced topology from $L^2([0,1] \times M, \Reals)$ on $C_0^\infty$.
    Of course, this is quite heretical, and we will deal with issues of topology in the next section.
\end{remark}
Before we prove this lemma, we will quickly comment on different modes of normalization.
Recall that a (time independent) Hamiltonian vector field determines an autonomous Hamiltonian function up to a constant.
It follows that a time-dependent Hamiltonian vector field fixes
the time-dependent Hamiltonian up to a time dependent constant.
To be precise, if $H,G: [0,1] \times M \to \Reals$ 
are two Hamiltonian functions with $X_H(t,\cdot) = X_G(t,\cdot)$ for all $t \in [0,1]$
Then $H(t,x) = G(t,x) + C(t)$ for all $x \in M, t\in [0,1]$ and some $C \in C^\infty([0,1])$.
We resolve this ambiguity by specifying a normalization condition.
The general convention in the field is to require that 
\[\int_{M} H_t \omega^n = 0,\]
for all $t \in [0,1]$.
Since we work in an $L^2$-setting, it is a natural choice to normalize using 
\begin{equation*}
    \langle 1, H_t \rangle_{L^2} =   \int_M H_t d\vol_g = 0. 
\end{equation*}
However, this is actually the same normalization.
Indeed, we have
   \[\frac{1}{n!}\int_{M} H_t \omega^n = \int_M H_t d\vol_g,\]
   since $n! \cdot d\vol_g = \omega^n$. This can be seen via a local computation in Darboux coordinates.
   Here we use that $g(\cdot, \cdot) = \omega(\cdot, J\cdot)$ for some almost-complex structure $J: TM \to TM$ compatible with $\omega$.
Note that for any eigenfunction $e \in L^2(M)$ of $\Laplace_g$ with a non-zero eigenvalue $\lambda$, we have 
\begin{align*}
    \langle 1, e \rangle_{L^2} = \lambda^{-1} \langle 1, \Laplace_g e \rangle_{L^2} = \lambda^{-1} \langle \Laplace_g 1, e \rangle_{L^2} = 0.
\end{align*}
Thus, by building our GP out of eigenfunction with non-zero eigenvalues, we have obtained the above normalization by default.
In particular, 
\begin{align*}
    \langle 1, H_t \rangle_{L^2} &= \left\langle 1, \sum_{n \geq 1} w_n Z_n(t) e_n \right\rangle_{L^2} 
    = \sum_{n \geq 1} w_n Z_n(t) \langle 1, e_n \rangle_{L^2} 
    = 0,
\end{align*}
whenever the series $H_t = H(t,\cdot)$ defined in~\ref{eq:GP_H} converges. So in particular, this holds on $\mathcal{S}$ and almost-surely in general.
In the last step we use that the series only involves eigenfunctions with non-zero eigenvalues by construction and thus $\langle 1, e_n \rangle_{L^2} = 0$ for all $n \geq 1$.
Thus, we have shown Lemma~\ref{lem:GP_H-normalized}.

\define{\label{def:equivalence-of-ldds}
    Let $\mathcal{D}_1$ and $\mathcal{D}_2$ be two \ldds .
    We say that $\mathcal{D}_1$ and $\mathcal{D}_2$ are \textbf{equivalent} if the Gaussian processes $H^{(1)}$ and $H^{(2)}$ defined in~\eqref{eq:GP_H} using $\mathcal{D}_1$ and $\mathcal{D}_2$ respectively are versions of each other.
}
After these setup steps,
we are now ready to finally define the Gaussian measure on normalized Hamiltonian functions.
\define{\label{def:muGPH}
    Let $\mathcal{D}$ be a \ldd.
    We define the \textit{Gaussian measure on Hamiltonian functions} $\muGPH^{\mathcal{D}}$ as the restriction of the law of the Gaussian process $H$ defined in~\eqref{eq:GP_H} with respect to the \ldd\ $\mathcal{D}$ to the space $C_0^\infty$
    of normalized Hamiltonian functions.
}
Note that this measure is well-defined, since by Corollary~\ref{cor:GP_H-smooth} the sample paths of $H$ are almost-surely smooth and by Lemma~\ref{lem:GP_H-normalized} the sample paths are normalized.
\subsection{Measure on the Hamiltonian diffeomorphism group}\label{subsec:ham-measure}
In the last section we defined a Gaussian measure on the space of normalized time-dependent Hamiltonian functions.
We now use this measure to define a measure on the group of Hamiltonian diffeomorphisms.
This idea is very simple, we will push-forward the measure $\muGPH$ defined in Definition~\ref{def:muGPH} under the (path-space) fibration
\begin{center}
    \begin{tikzcd}
        C^\infty_{0}  \cong \ham(M,\omega) \arrow[rr, two heads] &  & {\Ham(M,\omega)}.
    \end{tikzcd}
\end{center}
Since we wish to construct a Borel measure on $\Ham(M,\omega)$, we need to specify how we topologize these spaces.
In the following, we wish to topologize $\Ham(M,\omega)$ with the topology induced by the Hofer metric, see Definition~\ref{def:hofer-norm}.
This is a natural choice on $\Ham(M,\omega)$.

Note that we can introduce the intermediate space $\widetilde{\Ham}(M,\omega)$, i.e.\ the universal cover of $\Ham(M,\omega)$.
Elements of $\widetilde{\Ham}(M,\omega)$ are paths in $\Ham(M,\omega)$ that start at the identity and which we consider to be equivalent up to homotopy relative to the endpoints.
The space $\widetilde{\Ham}(M,\omega)$ is also endowed with a Hofer (pseudo-)norm which is obtained similar to the usual Hofer norm (see Definition~\ref{def:hofer-norm}) 
but instead of taking an infimum over all Hamiltonian functions, we take an infimum over all that generate a flow in the specified homotopy class.
The space $\widetilde{\Ham}(M,\omega)$ is a topological group and the Hofer (pseudo-)norm gives rise to a Hofer (pseudo-)metric $\tilde{d}_{\Hof}$ on $\widetilde{\Ham}(M,\omega)$.
Clearly, the covering map $(\widetilde{\Ham}(M,\omega), \tilde{d}_{\Hof}) \to (\Ham(M,\omega), d_{\Hof})$ is 1-Lipschitz.
We can now extend the fibration above to include the universal cover as follows:
\begin{center}
    \begin{tikzcd}
        C^\infty_{0} \arrow[rr, "\cong"] &  & {\ham(M,\omega)} \arrow[rr, two heads]         &  & {\widetilde{\Ham} (M,\omega)} \arrow[rr, two heads] &  & {\Ham(M,\omega)} \\
        H \arrow[rr, maps to]            &  & {(\phi_H^t)_{t \in [0,1]}} \arrow[rr, maps to] &  & {[\phi_H^t]_{/\sim}} \arrow[rr, maps to]           &  & \phi_H^1,
    \end{tikzcd}
\end{center}
where $[\cdot]_{/\sim}$ denotes the equivalence class of paths up to homotopy relative to the endpoints.
Since Lipschitz maps are particularly compatible with Gaussian measures (see~\cite{gine-nickl-2016} for many examples of this general philosophy),
we wish to metrize things in such a way that the maps above are all Lipschitz (and therefore so is their composition).
The following lemma (well-known in the field) gives us a clear indication as to which metric to use on $C^\infty_0$.
\begin{lemma}\label{lem:H_to_phi_H_Lipschitz}
    The map 
    \begin{align*}
        C^\infty_0 &\to \widetilde{\Ham}(M,\omega) \\
                        H &\mapsto [\phi_H^1]_{/\sim}
    \end{align*}
    is Lipschitz with Lipschitz constant $2$ with respect to the $C^0$-norm $\norm{\cdot}_\infty$ on $C^\infty_0$ and the Hofer norm $\tilde{d}_{\Hof}$ on $\widetilde{\Ham}(M,\omega)$.
\end{lemma}
\begin{proof}
    Let $H_1,H_2 \in C^\infty_0$ be arbitrary.
    We want to show 
    \begin{equation}\label{eq:Lipschitz}
        \tilde{d}_{\Hof}([\phi_{H_1}^t]_{/\sim}, [\phi_{H_2}^t]_{/\sim}) \leq 2 \cdot \norm{H_2 - H_1}_\infty.
    \end{equation}
    For this, we first note that 
    \begin{align*}
        \tilde{d}_{\Hof}([\phi_{H_1}^t]_{/\sim}, [\phi_{H_2}^t]_{/\sim}) =
        \tilde{d}_{\Hof}\left(\left[\left(\phi_{H_1}^t\right)^{-1}\phi_{H_2}^t\right]_{/\sim}, \id\right)
        = \norm*{\left[\left(\phi_{H_1}^t\right)^{-1}\phi_{H_2}^t\right]_{/\sim}}_{\Hof},
    \end{align*}
    by the bi-invariance of the Hofer metric.
    Now by Lemma~\ref{lem:flows_composition_and_inversion}, we obtain 
    \[\left(\phi_{H_1}^t\right)^{-1}\phi_{H_2}^t = \phi_{\overline{H}_1\sharp H_2}^t\]
    for any $t \in [0,1]$.
    Thus, 
    \begin{equation}\label{eq:Lipschitz_upper_bound}
        \tilde{d}_{\Hof}([\phi_{H_1}^t]_{/\sim}, [\phi_{H_2}^t]_{/\sim}) \leq
        \int_0^1 \osc \left[ (\overline{H}_1\sharp H_2)(t,\cdot)\right] dt.
    \end{equation}
    Now notice that 
    \[(\overline{H}_1\sharp H_2)(t,x) = \overline{H}_1(t,x) + H_2(t,(\phi_{\overline{H}_1}^t)^{-1}(x))
      = H_2(t,\phi_{H_1}^t(x)) -H_1(t,\phi_{H_1}^t(x)),\]
    where we again use Lemma~\ref{lem:flows_composition_and_inversion}.
    Thus, 
    \begin{align*}
        \int_0^1 \osc \left[ (\overline{H}_1\sharp H_2)(t, \cdot)\right] dt
        &= 
            \int_0^1 \Big[ \max_{x \in M}\left( H_2(t,\phi_{H_1}^t(x)) -H_1(t,\phi_{H_1}^t(x))\right) \\
            & \qquad - \min_{x \in M}\left( H_2(t,\phi_{H_1}^t(x)) -H_1(t,\phi_{H_1}^t(x))\right) \Big] dt \\
        &= \int_0^1 \left[ \max_{x \in M}\left( H_2(t,x) -H_1(t,x)\right)
        -  \min_{x \in M}\left( H_2(t,x) -H_1(t,x)\right) \right] dt \\
        &\leq \int_0^1 \left[ 2 \max_{x \in M}\abs*{ H_2(t,x) -H_1(t,x)}
       \right] dt \\
        &\leq 2 \max_{t \in [0,1]} \max_{x\in M}\abs{H_2(t,x) -H_1(t,x) }
        \\ &= 2 \norm{H_2 - H_1}_\infty,
    \end{align*}
    where we use that $\phi^t_{H_1}$ is a diffeomorphism (and thus in particular bijective) for the second equality.
    By plugging this inequality into~\eqref{eq:Lipschitz_upper_bound} we obtain~\eqref{eq:Lipschitz}, thus completing the proof.
\end{proof}
Thus, we wish to show that the measure $\muGPH$ defined in Definition~\ref{def:muGPH} is a Borel measure on $C_0^\infty$ with respect to the $C^0$-topology.
\begin{lemma}\label{lem:muGPH_Borel}
    Let $\mathcal{D}$ be a \ldd.
    Then the probability measure $\muGPH^{\mathcal{D}}$ defined in Definition~\ref{def:muGPH} is a Borel measure on $C_0^\infty$ with respect to the $C^0$-topology.
\end{lemma}
\begin{proof}
    Recall that $\muGPH^{\mathcal{D}}$ is defined as the law of the Gaussian process $H$ defined in~\eqref{eq:GP_H} with respect to the \ldd\ $\mathcal{D}$.
    To be completely precise, it is the law of the random variable $H: \mathcal{S} \to C^\infty_0$ defined on the restriction $\mathcal{S} \subset \Omega$ of our universal probability space $(\Omega, \mathcal{F}, \Prob)$.
    This is due to the fact that $H$ is actually valued in $C^\infty_0$ when we restrict to $\mathcal{S}$, instead of samples being in $C_0^\infty$ almost-surely when working with $\Omega$.

    To now prove our claim, it suffices to show that the function 
    \begin{align*}
    c_0:\mathcal{S}&\to \Reals \\
    \theta &\mapsto \sup_{\substack{t \in [0,1] \\ x \in M}} H_{\theta}(t,x)
    \end{align*}
    is measurable.
    It follows from the estimate~\eqref{eq:C0-estimate-eigenfunctions} that for any fixed $t \in [0,1]$, the series~\eqref{eq:GP_H} converges uniformly on $M$ almost-surely,
    and thus we obtain that $H(t,\cdot)$ is a $C^0$-valued random variable.
    In particular, the function $\theta \mapsto \norm{H_{\theta}(t,\cdot)}_{C^0}$ is measurable.
    Since we already know that all sample paths of $H$ are smooth after restricting to $\mathcal{S}$, we can use the fact that the supremum of a countable family of measurable functions is measurable.
    Indeed, we can complete the proof by writing the function from above as 
    \begin{align*}
    c_0(\theta) = \sup_{t \in [0,1] \cap \Q}\norm{H_{\theta}(t,\cdot)}_{C^0},
    \end{align*}
    thus showing that $c_0$ is measurable with respect to the Borel $\sigma$-algebra on $\Reals$ and $\mathcal{F}\vert_{\mathcal{S}}$.
    Thus, the law of $H$ is a Borel measure on $C_0^\infty$ with respect to the $C^0$-topology.
\end{proof}
We now come to the heart of this paper, the construction of a probability measure on $\widetilde{\Ham}(M,\omega)$ and $\Ham(M,\omega)$ from a \ldd.
The following pair of definition and corollary is the main result of this section and proves Theorem~\ref{thm:hofer-borel}.
\begin{definition}\label{def:measure-on-Ham}
    Let $\mathcal{D}$ be a \ldd.
    We define the \textit{associated probability measure} on $\widetilde{\Ham}(M,\omega)$, denoted $\widetilde{\muGP}^{\mathcal{D}}$, as the push-forward of the measure $\muGPH^{\mathcal{D}}$ (as given in Definition~\ref{def:muGPH}) under the map
    \begin{align*}
        C^\infty_0 &\to \widetilde{\Ham}(M,\omega) \\
        H &\mapsto [\phi_H^t]_{/\sim}.
    \end{align*}
    Furthermore, we define the \textit{associated probability measure} on $\Ham(M,\omega)$, denoted $\muGP^{\mathcal{D}}$, as the push-forward of the measure $\widetilde{\muGP}^{\mathcal{D}}$ under the covering map $\widetilde{\Ham}(M,\omega) \to \Ham(M,\omega)$.
\end{definition}
\begin{corollary}
    Let $\mathcal{D}$ be a \ldd.
    Then the following holds:
    \begin{enumerate}
        \item The measure $\widetilde{\muGP}^{\mathcal{D}}$ is a Borel measure on $\widetilde{\Ham}(M,\omega)$ with respect to the Hofer metric $\tilde{d}_{\Hof}$.
        Furthermore, the measure $\muGP^{\mathcal{D}}$ is a Borel measure on $\Ham(M,\omega)$ with respect to the Hofer metric $d_{\Hof}$.
        \item These measures are probability measures, i.e.\ $\widetilde{\muGP}^{\mathcal{D}}(\widetilde{\Ham}(M,\omega)) = 1$ and $\muGP^{\mathcal{D}}(\Ham(M,\omega)) = 1$.
        \item The measure $\muGP^{\mathcal{D}}$ can alternatively be obtained as the push-forward of the measure $\muGPH^{\mathcal{D}}$ under the map $C^\infty_0 \to \Ham(M,\omega)$ given by $H \mapsto \phi_H^1$.
    \end{enumerate}
\end{corollary} %
\subsection{Measure on the autonomous Hamiltonian diffeomorphisms}\label{subsec:aut-measure}
We conclude this section by proving Theorem~\ref{thm:aut-measure}.
This theorem allows us to construct measures on the space of autonomous Hamiltonian diffeomorphisms.
In particular, this is necessary to construct the random walk on $\Ham(M,\omega)$ in Section~\ref{sec:random-walk}.
\begin{proof}[Proof of Theorem~\ref{thm:aut-measure}]
    Assume that $\phi \in \supp \muGP^{\mathcal{D}}$.
    Thus, for all $\eps > 0$ we have $\muGP^{\mathcal{D}}(\{\psi \in \Ham(M,\omega) \mid d_{\Hof}(\phi,\psi) < \eps\}) > 0$.
    Since $\mathcal{D}$ is autonomous, the Gaussian processes $Z_n$ are constant on $[0,1]$ almost-surely and thus 
    the random Hamiltonian function $H$ given by~\eqref{eq:GP_H} is time-independent almost-surely.
    Assume that 
    \[\{\psi \in \Ham(M,\omega) \mid d_{\Hof}(\phi,\psi) < \eps\} \cap \Aut(M,\omega) = \emptyset.\]
    Thus, the open set $\mathcal{F}_\eps(\phi) = \{F \in C^\infty_0 \mid d_{\Hof}(\phi_F^1, \phi) < \eps\}$ does not contain any time-independent 
    Hamiltonian functions.
    Now by definition of $\muGP^{\mathcal{D}}$, we have that
    \begin{align*}
        0 < \muGP^{\mathcal{D}}(\{\psi \in \Ham(M,\omega) \mid d_{\Hof}(\phi,\psi) < \eps\})
        = \muGPH^{\mathcal{D}}(\mathcal{F}_\eps(\phi)).
    \end{align*}
    Since $\muGPH^{\mathcal{D}}$ is the law of $H$, this is a contradiction to the fact that $H$ is time-independent almost-surely.
    (This fact implies that $\muGPH^{\mathcal{D}}(\mathcal{F}_\eps(\phi)) = 0$ since the set $\mathcal{F}_\eps(\phi)$ does not contain any time-independent Hamiltonian functions.)
    Since this holds for all $\eps > 0$, we have that $\phi \in \overline{\Aut}(M,\omega)$.
    To show the second part of the theorem, 
    note that the event $\mathcal{A} = \{H(\cdot,x) \text{ is constant on } [0,1] \text{ for all } x \in M\}$ is measurable and has probability $1$.
    Thus, by restricting the probability space to $\mathcal{A}$, we obtain that $\phi_H^1$ is an $\Aut(M,\omega)$-valued random variable.
    Let $\muAut^{\mathcal{D}}$ be its law.
    It follows from the construction that $\muGP^{\mathcal{D}}$ is the push-forward of $\muAut^{\mathcal{D}}$ under the inclusion $\Aut(M,\omega) \hookrightarrow \Ham(M,\omega)$.
\end{proof} %
\section{Existence of \ldds}\label{sec:existence-of-ldd}
We will now prove that \ldds\ --- with all combinations of properties that we would like to have --- exist.
This is crucial, since most of the results in this paper would otherwise be vacuous.
\begin{theorem}\label{thm:existence-of-ldd}
    Let $(M,\omega)$ be a closed symplectic manifold.
    Further, let $J$ be an $\omega$-compatible almost complex structure
    and $\reg > 0$.
    Then there exists
    \begin{enumerate}
        \item a \ldd\ $\mathcal{D}_1$ that is centered, time-symmetric, exhaustive, and frequency unbiased; and
        \item a \ldd\ $\mathcal{D}_2$ that is centered, time-symmetric, periodic, periodically exhaustive, and frequency unbiased; and
        \item a \ldd\ $\mathcal{D}_3$ that is centered, autonomous, autonomously exhaustive, and frequency unbiased.
    \end{enumerate}
    Furthermore, for both limiting types there exists
    \begin{enumerate}
        \item a \ldd\ $\mathcal{D}_4$ that is centered, time-symmetric, exhaustive, and weakly frequency unbiased; 
        \item a \ldd\ $\mathcal{D}_5$ that is centered, time-symmetric periodic, periodically exhaustive, and weakly frequency unbiased; and
        \item a \ldd\ $\mathcal{D}_6$ that is centered, autonomous, autonomously exhaustive, and frequency unbiased,
    \end{enumerate}
    such that $\mathcal{D}_4,\mathcal{D}_5,\mathcal{D}_6$ are of the specified limiting type.
    Furthermore, $\mathcal{D}_1,\ldots,\mathcal{D}_6$ all have the same regularity parameter $\reg$ and almost-complex structure $J$.
\end{theorem}
\begin{proof}
    Clearly, 
    we can use $J$ as the almost-complex structure
    and $\reg$ as the regularity parameter for all six \ldds . 
    We then choose --- once and for all --- an eigenbasis $\{e_n\}_{n=0}^\infty$ of $\Laplace_g$, 
    where $g$ is the metric defined by $\omega$ and $J$.
    Thus, it only remains to specify the Gaussian processes associated with each eigenfunction.
    We do this on a case-by-case basis.
    \begin{enumerate}
        \item To obtain a centered, time-symmetric, exhaustive, and frequency unbiased \ldd\ $\mathcal{D}_1$,
                we let each $Z_n$ for $n \geq 1$ be an independent version of the process 
                from Example~\ref{ex:gaussian-process-ex1} with $\alpha = \reg$,
                i.e.\ the Gaussian process $Z_n$ has covariance function
                $\kappa(t,s) = \exp(-\reg(s-t)^2)$ and mean zero.
                It follows directly that $\mathcal{D}_1$ is centered, time-symmetric, and frequency unbiased.
                To show that it is exhaustive, we need to show that the of this Gaussian process
                RKHS is dense in $C^0([0,1])$.
                This is well-known in the literature, we refer the reader to~\cite[Chapter 10]{wendland-2004}
                for a complete proof.
        \item To obtain a centered, time-symmetric, periodic, periodically exhaustive, and frequency unbiased \ldd\ $\mathcal{D}_2$,
                we let each $Z_n$ for $n \geq 1$ be an independent version of the process
                from Example~\ref{ex:RKHS-example} with $\eps = 2\reg$,
                i.e.\ we set
                \begin{equation*}
                    Z_n(t) = X^{(n,1)}_0 + {\sqrt{2}}\sum_{k=1}^\infty \exp\left(-2\reg\pi^2k^2\right) \cdot \left( X_k^{(n,1)} \cos(2\pi k t) + X_k^{(n,2)} \sin(2\pi k t)\right),
                \end{equation*}
                where $\{X_k^{(n,1)}, X_k^{(n,2)}\}_{n,k \in \Nats_0}$
                are independent standard Gaussian random variables.
                We have already seen in Example~\ref{ex:RKHS-example} that the RKHS of $Z_n$ is dense in $C^0(S^1)$, 
                where $S^1 = [0,1]/_{0 \sim 1}$.
                Thus, $\mathcal{D}_2$ is periodically exhaustive.
                It is also clear that $\E[Z_n] \equiv 0$, thus $\mathcal{D}_2$ is centered.
                Finally, one can compute that 
                \begin{align*}
                    \Cov[Z_n(t_1),Z_n(t_2)] 
        &= 1 + \sum_{k=1}^\infty  2\cdot\exp\left(-4\reg\pi^2k^2\right) \cdot \cos(2\pi k (t_1 - t_2)),
                \end{align*}
                thus $Z_n$ and $Z_n(1-\cdot)$ have the same law (since they have the same mean and covariance function).
                Thus, $\mathcal{D}_2$ is time-symmetric.
        \item To obtain a \ldd\ $\mathcal{D}_3$ that is centered, autonomous, autonomously exhaustive, and frequency unbiased,
                let $\{X_n\}_{n \in \Nats}$ be independent standard Gaussian random variables.
                Then let $Z_n(t) = X_n$ for all $t \in [0,1]$.
                Clearly, these processes are centered and almost-surely constant.
                Thus, $\mathcal{D}_3$ is centered and autonomous.
                Furthermore, all $Z_n$ have positive variance.
                Therefore, we obtain that $\mathcal{D}_3$ is autonomously exhaustive.
                The i.i.d.\ assumption on the $X_n$ implies that $\mathcal{D}_3$ is frequency unbiased.
    \end{enumerate}
    The remainder of the theorem follows by rescaling the processes $Z_n$ in the above constructions
    to obtain the correct limiting type.
\end{proof}
\rmk{\label{rmk:D2-RKHS}
    The \ldd\ $\mathcal{D_2}$ in the above construction is especially useful.
    See Figure~\ref{fig:simulation-1} for some simulations of Hamiltonian vector fields
    obtained from $\mathcal{D}_2$ on the two-dimensional torus with $\reg = 0.08$.
    Note that if we consider the manifold (with boundary) $[0,1] \times M$,
    with the metric given by the product of the Euclidean metric on $[0,1]$ and the metric $g$ on $M$,
    then the eigenfunctions of the Laplace-Beltrami operator on $[0,1] \times M$ are given by
    products of $\{1,\sqrt{2}\cos(2\pi k t), \sqrt{2}\sin(2\pi k t)\}_{k \in \Nats}$
    and $\{e_n\}_{n=0}^\infty$, where the latter is some eigenbasis of $\Laplace_g$.
    Here $t$ is the coordinate on $[0,1]$.
    In particular, 
    we obtain the following formula for the Gaussian process $H$ defined in~\eqref{eq:GP_H}:
    \begin{align*}
        H(t,x) &= \sum _{n \ge 1} w_n \cdot Z_n(t) \cdot e_n(x)\\
        &= \sum_{n \geq 1} w_n \cdot \left(
        X^{(n,1)}_0 + {\sqrt{2}}\sum_{k=1}^\infty \exp\left(-2\reg\pi^2k^2\right) \cdot \left( X_k^{(n,1)} \cos(2\pi k t) + X_k^{(n,2)} \sin(2\pi k t)\right)
        \right) e_n(x) \\
        &= \sum_{n \geq 1} w_n X^{(n,1)} e_n(x)\\ &\qquad  + \sum_{\substack{n \ge 1 \\ k \ge 1}} \exp\left( - \frac{\reg}{2} {\left( 4\pi^2k^2 + \lambda_n \right)} \right) 
\cdot \Big( X_k^{(n,1)} \cdot \sqrt{2} \cos(2\pi k t) \cdot e_n(x) + 
            X_k^{(n,2)} \cdot \sqrt{2} \sin(2\pi k t) \cdot e_n(x) \Big).
    \end{align*}
    Thus, we obtain the law of $H(t,x)$ as an orthogonal expansion with respect to the preferred eigenbasis of the Laplace-Beltrami operator on $[0,1] \times M$.
    Note that all coefficients are i.i.d.\ standard Gaussian random variables, 
    scaled by an exponential decay factor based on the eigenvalue.
    Note that $4\pi^2k^2 + \lambda_n$ is the eigenvalue corresponding to the eigenfunction $\sqrt{2}\cos(2\pi k t) e_n(x)$ or $\sqrt{2}\sin(2\pi k t) e_n(x)$.
    Thus, the law of $H$ is very similar to that of the process considered in Example~\ref{ex:RKHS-example}.
    This also means that much of the analysis carried out in~\cite{castillo-et-al-2014} and~\cite{nicolaescu-2014}
    can be applied for $H$ (in the case of $\mathcal{D}_2$) as well.
    In particular, 
    we obtain the following description of the RKHS $\Hil$ of $H$:
    \begin{align*}
        \Hil = \Bigg\{ & \sum_{\substack{n \ge 1 \\ k \ge 0}}
    \left( a_{k,n} \cdot \sqrt{2} \cos(2\pi k t) \cdot e_n(x) + 
            b_{k,n} \cdot \sqrt{2} \sin(2\pi k t) \cdot e_n(x) \right)
            \Bigg\vert \\ &\qquad a,b \in \ell^2, \sum_{\substack{n \ge 1 \\ k \ge 0}} \exp\left(\reg (4\pi^2k^2 + \lambda_n) \right) \cdot (a_{k,n}^2 + b_{k,n}^2) < \infty
            \Bigg\}.
    \end{align*}
    Furthermore, for two elements of the RKHS as above, their inner product is given by
    \begin{align*}
        \Bigg\langle&\sum_{\substack{n \ge 1 \\ k \ge 0}}
    \left( a_{k,n} \cdot \sqrt{2} \cos(2\pi k t) \cdot e_n(x) + 
            b_{k,n} \cdot \sqrt{2} \sin(2\pi k t) \cdot e_n(x) \right),\\ &\qquad \sum_{\substack{n \ge 1 \\ k \ge 0}}
    \left( c_{k,n} \cdot \sqrt{2} \cos(2\pi k t) \cdot e_n(x) + 
            d_{k,n} \cdot \sqrt{2} \sin(2\pi k t) \cdot e_n(x) \right)
        \Bigg\rangle_{\Hil}
        \\ &= \sum_{\substack{n \ge 1 \\ k \ge 0}}
        \exp\left(\reg (4\pi^2k^2 + \lambda_n) \right) \cdot (a_{k,n}c_{k,n} + b_{k,n}d_{k,n}).
    \end{align*}
}
\begin{figure}[h]
    \begin{center}
        \includegraphics[width=0.9\textwidth]{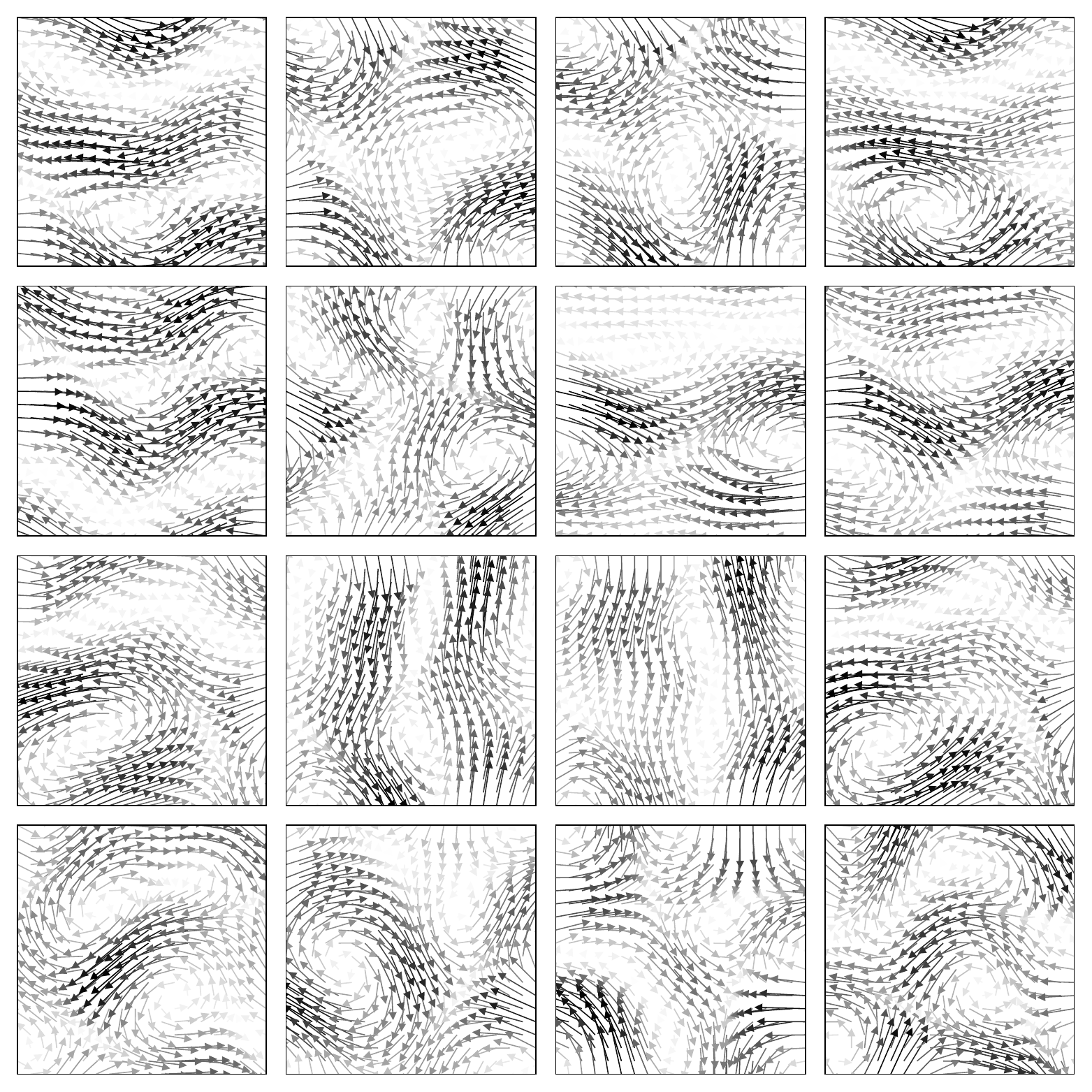}
        \caption{Visualization of $X_H(t,\cdot)$ for $\mathcal{D}_2$ on $\T^2$. The rows represent independent draws. The different columns represent evaluations of the Hamiltonian vector fields
    at $t=0,0.25,0.5,0.75$ from left to right. An animated version of this figure is available online at \href{https://www.dpmms.cam.ac.uk/~apd55/random-hamiltonians/}{dpmms.cam.ac.uk/\~{}apd55/random-hamiltonians}.}\label{fig:simulation-1}
    \end{center}
\end{figure} %
\section{Properties}
Now that we have completed the construction of $\muGPH^{\mathcal{D}}$ and $\muGP^{\mathcal{D}}$, we can study some elementary properties of these measures.
We will first focus on those properties that do not depend on the \ldd\ $\mathcal{D}$.
\begin{lemma}\label{lem:expected-osc-norm}
    Let $\mathcal{D}$ be a \ldd.
    Then the expected oscillation norm of the Hamiltonian function is finite, i.e.,
    \[
    \int_{C^\infty_0} \osc(H) d\muGPH^{\mathcal{D}}(H) = \int_{C^\infty_0} \int_0^1 \max_{x \in M} H(t,x) - \min_{x\in M}
    H(t,x) dt  d\muGPH^{\mathcal{D}}(H) < \infty.
    \]
\end{lemma}
\begin{proof}
    We first note that $\osc(H(t,\cdot)) \geq 0$ and thus, by the Fubini-Tonelli theorem, we can interchange the order of integration:
    \begin{align*}
    \int_{C^\infty_0} \int_0^1 \left[ \max_{x \in M} H(t,x) - \min_{x\in M} H(t,x)\right] dt  d\muGPH^{\mathcal{D}}(H) =
    \int_0^1 \int_{C^\infty_0} \left[ \max_{x \in M} H(t,x) - \min_{x\in M} H(t,x)\right] d\muGPH^{\mathcal{D}}(H) dt.
    \end{align*}       
    Fix now some $t \in [0,1]$.
    Then by definition, we have that
    \begin{align*}
    \int_{C^\infty_0} \left[ \max_{x \in M} H(t,x) - \min_{x\in M} H(t,x)\right] d\muGPH^{\mathcal{D}}(H) &= \E\left[\max_{x \in M} H(t,x) - \min_{x\in M} H(t,x)\right] ,
    \end{align*}
    where $H$ is the Gaussian process defined in~\eqref{eq:GP_H}.
    Now by linearity of expectation, we can write
    \begin{align*}
    \E\left[\max_{x \in M} H(t,x) - \min_{x\in M} H(t,x)\right] &= \E\left[\max_{x \in M} H(t,x)\right] - \E\left[\min_{x\in M} H(t,x)\right].
    \end{align*}
    It is easier in this case to work with a centered process, thus we set $\widetilde{H}(t,x) = H(t,x) - \E[H(t,x)]$.
    This is a centered Gaussian process and thus $\E[\max_{x \in M}\widetilde{H}(t,x)] = -\E[\min_{x \in M}\widetilde{H}(t,x)]$.
    It follows that 
    \begin{align*}
        \E\left[\max_{x \in M} H(t,x) - \min_{x\in M} H(t,x)\right] &= 
        \E\left[\max_{x \in M} \left(\widetilde{H}(t,x) + \E[H(t,x)]\right) - \min_{x\in M}\left( \widetilde{H}(t,x) - \E[H(t,x)]\right)\right] \\
        &\leq \E\left[\max_{x \in M} \widetilde{H}(t,x) + \max_{x \in M}\E[H(t,x)] - \min_{x\in M}\widetilde{H}(t,x) - \min_{x\in M}\E[H(t,x)]\right] \\
        &= \E\left[\osc \widetilde{H}_t \right] + \osc(\E[H_t]).
    \end{align*}
    Since $\osc(\E[H_t])$ is finite and also bounded in $t$ (since $\E[H(t,\cdot)]$ is smooth in $t$), it suffices to focus on the term $\E[\osc \widetilde{H}_t]$.
    A simple application of the triangle inequality gives us 
    \begin{align*}
    \E\left[ \osc \widetilde{H}_t\right] &\leq 2\E\left[\norm{\widetilde{H}(t,\cdot)}_{\infty}\right] = 2\E\left[ \norm*{\sum _{n \ge 1} w_n \cdot \abs{Z_n(t)} \cdot e_n}_\infty \right] \\
    &\leq \sum_{n \ge 1} w_n \cdot 2\E\left[\abs{Z_n(t)}\right] \cdot \norm{e_n}_\infty
    =\sum_{n \ge 1} w_n \cdot {\frac{2\sqrt{2}}{\sqrt{\pi}}} \sqrt{\Var[Z_n(t)]} \cdot \norm{e_n}_\infty,
    \end{align*}
    where we also use the standard fact that the mean of the absolute value of a centered Gaussian random variable is $\sqrt{\frac{2}{\pi}}$ times its standard deviation.
    By the definition of a \ldd, there exists a constant $C > 0$ such that for all $n \ge 1$ and $t \in [0,1]$ we have $\Var[Z_n(t)] \leq C^2$.
    Thus, 
    \begin{align*}
    \E\left[\osc \widetilde{H}_t\right] &\leq {\frac{2\sqrt{2}}{\sqrt{\pi}}} \cdot C \cdot \sum_{n \ge 1} w_n \norm{e_n}_\infty 
    \leq C' \cdot \sum_{n \ge 1} w_n \lambda^{\dim M}_n,
    \end{align*}
    where $C'$ is some constant that does only depend on $(M,\omega,J)$ and $\mathcal{D}$.
    Here we use the bound on the $C^0$-norm of the eigenfunctions from~\eqref{eq:C0-estimate-eigenfunctions}.
    Due to the rapid decay of the weights $w_n$, this sum converges.
    Thus, there is some constant $C'' > 0$ such that $\E\left[\max_{x \in M} \widetilde{H}(t,x)\right] < C''$. Combining all these estimates, we obtain 
    \begin{align*}
        \int_{C^\infty_0} \osc(H) d\muGPH^{\mathcal{D}}(H)
        &\leq \int_0^1 \left(\E\left[\osc \widetilde{H}_t\right] + \osc(\E[H_t])\right) dt \\
        &\leq \int_0^1 C'' dt + \int_0^1 \osc(\E[H(t,\cdot)]) dt \\
        &\leq C'' + \osc(\E[H]) < \infty.
    \end{align*}
\end{proof}
We directly obtain the following corollary, which also proves Theorem~\ref{thm:expected-hofer-norm-finite} and Corollary~\ref{cor:hofer-lipschitz-finite-expectation}.
\begin{corollary}\label{cor:expected-hofer-norm-finite}
    Let $\mathcal{D}$ be a \ldd.
    Then the expected Hofer norm is finite, i.e.,
    \begin{align*}
    \int_{\Ham(M,\omega)} \norm{\phi}_{\Hof} d{\muGP}^{\mathcal{D}}(\phi) &\leq  \int_{\widetilde{\Ham}(M,\omega)} \norm{[\phi_t]_{/\sim}}_{\Hof} d\widetilde{\muGP}^{\mathcal{D}}([\phi_t]_{/\sim}) \\
    &\leq \int_{C^\infty_0} \osc(H) d\muGPH^{\mathcal{D}}(H) \\ &< \infty.
    \end{align*}
\end{corollary}
\begin{proof}
    This follows directly from the fact that for any $H \in C^\infty_0$, we have 
    \begin{equation*}
        \norm{\phi^1_H}_{\Hof} \leq \norm*{{[\phi_H^t]}_{/\sim}}_{\Hof} \leq \osc H
    \end{equation*}
    by the definition of the Hofer norm.
\end{proof}
We have shown that the expected Hofer norm (or indeed any Hofer-Lipschitz measurement) is finite.
However, we wish to further characterize its distribution.
Theorem~\ref{thm:hofer-norm-subgaussian} states that the tail of the Hofer norm is sub-Gaussian, placing it in a well-studied and well-understood class 
of random variables.
\begin{proof}[Proof of Theorem~\ref{thm:hofer-norm-subgaussian}]
    Let $\mathcal{D}$ be a centered \ldd.
    Note that for any $H \in C_0^\infty$, 
    we have $\norm{\phi_H^1} \leq \osc H$, and 
    thus
    \begin{align*}
        \muGP^{\mathcal{D}}(\{\phi \in \Ham(M,\omega) \mid \norm{\phi}_{\Hof} > u\})
        &\leq 
        \muGPH^{\mathcal{D}}(\{H \in C_0^\infty \mid \osc H > u\}) \\
        &= \Prob[\osc H > u] \leq \Prob[\norm{H}_\infty > u],
    \end{align*}
    where $H$ is the Gaussian process defined in~\eqref{eq:GP_H} and $(\Omega, \mathcal{F}, \Prob)$ is the probability space on which it is defined.
    Let $\sigma^2 = \max_{t \in [0,1], x \in M} \Var[H(t,x)]$ be the maximum of the variances of the process $H$.
    Now, by the $\norm{\cdot}_\infty$-version of the Borell-TIS inequality (Lemma~\ref{lem:Borell-TIS-infinity}), we have
    \begin{align*}
        \Prob[\norm{H}_\infty - \E[\norm{H}_\infty] > u] \leq \exp\left(-\frac{u^2}{2\sigma^2}\right).
    \end{align*}
    Thus, if we set $C = 2\sigma^2$ and $R = \E[\norm{H}_\infty]$
    we obtain 
    \begin{align*}
         \muGP^{\mathcal{D}}(\{\phi \in \Ham(M,\omega) \mid \norm{\phi}_{\Hof} > R + u\})&\leq \Prob[\norm{H}_\infty > R + u] \\
         &=
        \Prob[\norm{H}_\infty - \E[\norm{H}_\infty] > u] \leq \exp\left(-C^{-1}u^2\right),
    \end{align*}
    which completes the proof.
\end{proof}

We have now established that the Hofer norm is a sub-Gaussian type random variable.
However, this is only an interesting statement if it is non-trivial.
Of course, currently our construction allows for the very simple \ldd\ $\mathcal{D}$ with $Z_n \equiv 0$ for all $n \in \Nats_{>0}$.
In this case, the expected Hofer norm is trivially $0$ since $\muGPH^{\mathcal{D}}$ is the Dirac measure at the zero Hamiltonian function.
However, under some assumptions on the \ldd, we can show that the expected Hofer norm is strictly positive.
We will obtain this as a corollary of Theorem~\ref{thm:full-support}, which we will prove using the following lemma.
\begin{lemma}\label{lem:GP_H_full_support}
    Let $\mathcal{D}$ be an exhaustive, frequency independent \ldd.
    Then the support of the measure $\muGPH^{\mathcal{D}}$ is the whole space $C^\infty_0$ topologized with the $C^0$-topology.
\end{lemma}
\begin{proof} 
    We start by noting that we can actually assume that $\mathcal{D}$ is additionally centered, and thus, by Lemma~\ref{lem:GP_H}, we have $\E[H] = 0$.
    If this is not the case, we can consider the centered process $\tilde{H} = H - \E[H]$ instead.
    If we assume that the support of $\tilde{H}$ is the whole space $C^\infty_0$, 
    it follows that this was already true for $H$.
    Thus, it suffices to consider the case $\E[H] = 0$.

    Let $F \in C^\infty_0$ be arbitrary. We need to show that for some $\eps > 0$, we have $\muGPH^{\mathcal{D}}(\{G \mid \norm{F-G}_\infty < \eps\}) > 0$, i.e.\ the $\eps$-ball around $F$ (with respect to the $\infty$-norm) has positive measure.
    We can define the time-dependent coefficients 
    \[f_n(t) = \frac{1}{w_n}\left\langle F(t,\cdot), e_n \right\rangle,\]
    for any $n \in \Nats_{>0}$.
    Since $F$ is normalized, and $\{1,e_1,e_2,\ldots\}$ is an orthonormal basis of $L^2(M,g)$, we have that $\sum_{n \ge 1}w_n f_n(t)e_n(x) = F(t,x)$
    for all $t \in [0,1]$ and $x \in M$.
    Now notice that for any $g_1,\ldots \in C^\infty([0,1])$, we have that 
    \begin{align*}
        \norm*{F - \sum_{n \ge 1} w_n g_n(\cdot_t)e_n(\cdot_x)}_\infty &=
        \norm*{\sum_{n \ge 1} w_n \left( f_n(\cdot_t) - g_n(\cdot_t)\right)e_n(\cdot_x)}_\infty 
        = \max_{t \in [0,1]} \norm*{\sum_{n \ge 1} w_n \left( f_n(t) - g_n(t)\right)e_n(\cdot_x)}_\infty \\
        &\leq \max_{t \in [0,1]} \sum_{n \ge 1} w_n \abs{f_n(t) - g_n(t)} \cdot \norm{e_n}_\infty \\
        &\leq  \sum_{n \ge 1} w_n \left(\max_{t \in [0,1]} \abs{f_n(t) - g_n(t)}\right) \cdot \norm{e_n}_\infty \\
        &\leq \sum_{n \ge 1} w_n\cdot \norm{f_n - g_n}_{\infty} \cdot \norm{e_n}_\infty,
    \end{align*}
    whenever $\sum_{n \ge 1} w_n g_n(t)e_n(x)$ converges uniformly.
    We now wish to process similarly in spirit to Example~\ref{ex:gaussian-process-ex2}.
    Let us assume $\dim M = 2d$ for $d \in \Nats_{>0}$.
    Note that by the Weyl law, there exists a constant $C > 0$ such that for all $n \in \Nats_{>0}$, we have $\lambda_n \leq C \cdot n^{\frac{1}{d}}$.
    Furthermore, by the estimate of $\infty$-norms of the eigenfunctions in~\eqref{eq:C0-estimate-eigenfunctions}, we have that
    $\norm{e_n}_\infty \leq C' \cdot \lambda_n^{d}$ for some constant $C' > 0$.
    Thus, $\norm{e_n}_\infty \leq C' \cdot (C \cdot n^{\frac{1}{d}})^{d} \leq C'' \cdot n$ for some constant $C''> 0$
    which only depends on $(M,\omega,J)$.
    Now assume that for all $n \in \Nats_{>0}$, we have $\norm{f_n - g_n}_{\infty} < \exp(\frac{1}{2}\lambda_n \reg) \cdot \frac{1}{n^3} = w_n^{-1} n^{-3}$.
    Then we have that 
    \begin{align*}
        \norm*{F - \sum_{n \ge 1} w_n g_n(\cdot_t)e_n(\cdot_x)}_\infty 
        &\leq \sum_{n \ge 1} w_n\cdot \norm{f_n - g_n}_{\infty} \cdot \norm{e_n}_\infty \\
        &< \sum_{n \ge 1} \frac{w_n}{n^3} \cdot w_n^{-1} \cdot \norm{e_n}_\infty 
        = \sum_{n \ge 1} \frac{1}{n^3} \cdot \norm{e_n}_\infty \\
        &\leq C'' \cdot \sum_{n \ge 1} \frac{1}{n^2} = C'' \cdot \frac{\pi^2}{6}.
    \end{align*}
    It follows directly that whenever $\norm{f_n - g_n}_{\infty} < \frac{6}{C''\pi^2 w_n n^3} \eps $,
    we have that \[\norm*{F - \sum_{n \ge 1} w_n g_n(t)e_n(x)}_\infty < \eps\]
    for any $\eps 
    > 0$.
    By construction of $\muGPH^{\mathcal{D}}$, as the law of the Gaussian process $H$ defined in~\eqref{eq:GP_H}, this implies that 
    \begin{align*}
        \muGPH^{\mathcal{D}}(\{G \mid \norm{F-G}_\infty < \eps\})
        &= \Prob\left[ \norm{H - F}_\infty < \eps \right] \\
        &\geq 
        \Prob\left[\bigcap_{n \ge 1}               \left\{ \norm{f_n - Z_n}_{\infty} < \frac{6}{C'' \pi^2 w_n n^3} \eps  \right\} \right]
        \\ &= \prod_{n \ge 1}                  \Prob\left[ \norm{f_n - Z_n}_{\infty} < \frac{6}{C'' \pi^2 w_n n^3} \eps  \right],
        \\ &= \prod_{n \ge 1} 1- \underbrace{\Prob\left[\norm{f_n - Z_n}_{\infty} \geq \frac{6}{C'' \pi^2 w_n n^3} \eps  \right]}_{\eqqcolon p_n},
    \end{align*}
    where we use the assumption that $\mathcal{D}$ is frequency independent in the third line.
    Thus, our problem of showing that $\muGPH^{\mathcal{D}}(\{G \mid \norm{F-G}_\infty < \eps\}) > 0$ reduces to showing that the product $\prod_{n \ge 1} (1 - p_n)$ is positive.
    This is equivalent to showing that $\sum_{n \ge 1} p_n < \infty$ and $p_n < 1$ for all $n \in \Nats_{>0}$.
    To show the latter, we note that since $\mathcal{D}$ is exhaustive, 
    the support of the law of $Z_n$ in $(C^0([0,1],\R),\norm{\cdot}_{\infty})$ is the whole space $C^0([0,1],\R)$.
    However, if $p_n = 1$ this would imply that a ball of radius ${3\eps}{(C'' \pi^2 w_n n^3)^{-1}}$ around $f_n$ lies outside the support of
    $Z_n$, which is a contradiction.
    Thus, for any $n \in \Nats_{>0}$, we have $p_n < 1$.
    It now remains to show that $\sum_{n \ge 1} p_n < \infty$.
    Notice the following inequality, which holds for any radius $R > 2\norm{f_n}_\infty$:
    \begin{align*}
            \Prob[\norm{f_n - Z_n} \geq R] \leq \Prob\left[\norm{Z_n} \geq \frac{R}{2} + \norm{f_n}_\infty\right]
    \end{align*}
    This is a direct consequence of the triangle inequality and $\E[Z_n] = 0$.
    Now, since $F$ is a smooth function $\norm{\langle F(t,\cdot), e_n \rangle}_{\infty}$ decays faster than any polynomial in $n$.
    In particular, faster than $\frac{3\eps}{C''\pi^2} n^{-3}$.
    It follows that there is an $N \in \Nats$ such that for all $n \ge N$, 
    we have 
    \begin{equation*}
        \frac{3\eps}{C''\pi^2w_n n^3} >\frac{1}{w_n} \norm{\langle F(t,\cdot), e_n \rangle}_{\infty} = \norm{f_n}_\infty.
    \end{equation*}
    This allows us to make the following estimate:
    \begin{align*}
        p_n &= \Prob\left[\norm{f_n - Z_n}_{\infty} \geq \frac{6\eps}{\pi^2C''w_n n^3} \right] \\
        &\leq \Prob\left[\norm{Z_n} \geq \frac{3\eps}{C''\pi^2w_n n^3} + \norm{f_n}_\infty\right] 
        \leq \Prob\left[\norm{Z_n} \geq \frac{3\eps}{C''\pi^2w_n n^3} \right] \\
        &= \Prob\left[\norm{Z_n} - \E[\norm{Z_n}_\infty]\geq \frac{3\eps}{C''\pi^2w_n n^3} - \E[\norm{Z_n}_\infty]\right]
    \end{align*}
    Without loss of generality, we can assume that $N$ was chosen large enough that $\frac{3\eps}{C''\pi^2w_n n^3} > \E[\norm{Z_n}_\infty]$.
    This follows since $w_n$ decays exponentially in $n$ and $\{\E[\norm{Z_n}_\infty]\}_{n \in \Nats_{>0}}$ is a bounded sequence by the definition of a \ldd .
    To complete the proof, we can use the Borell-TIS inequality (see Lemma~\ref{lem:Borell-TIS-infinity}) to estimate $p_n$ as follows:
    \begin{align*}
        p_n &\le \Prob\left[\norm{Z_n} - \E[\norm{Z_n}_\infty]\geq \frac{3\eps}{C''\pi^2w_n n^3} - \E[\norm{Z_n}_\infty]\right] \\
        &\le \exp\left(-\frac{\left(\frac{3\eps}{C''\pi^2 w_n n^3} - \E[\norm{Z_n}_\infty]\right)^2}{2\sigma_n^2}\right) \\
        &\le \exp\left(-\frac{1}{2\sigma_n^2}\left(\frac{3\eps}{C''\pi^2 n^3} \exp\left(\frac{1}{2}\lambda_n \reg\right) - \E[\norm{Z_n}_\infty]\right)^2\right),
    \end{align*}
    where $\sigma_n^2 = \max_{t \in [0,1]} \Var[Z_n(t)]$.
    Since $\kappa_n(t,t) = \Var[Z_n(t)]$ is bounded uniformly in $n$ and $t \in [0,1]$ by the definition of a \ldd, this implies that $\sum_{n \ge 1} p_n < \sum_{n \ge N} p_n < \infty$
    converges since the argument of the exponential rapidly diverges to $-\infty$ above as $n \to \infty$.
    The latter follows immediately from Weyl's law for the eigenvalues $\lambda_n$.
    This completes the proof.
    \color{black}
\end{proof}
\begin{corollary}\label{cor:full_support_periodically}
    Let $\mathcal{D}$ be a periodically exhaustive, frequency independent \ldd.
    Then the support of the measure $\muGP^{\mathcal{D}}$ is the whole group $\Ham(M,\omega)$.
\end{corollary}
\begin{proof}
    Note that any Hamiltonian diffeomorphism $\phi \in \Ham(M,\omega)$ admits a periodic Hamiltonian function $H \in C^\infty_0$ such that $\phi = \phi^1_H$. 
    This follows in particular from the fact that $\Ham(M,\omega)$ is generated by $\Aut(M,\omega)$, but one can also check this directly.
    Thus, we need to show that any periodic Hamiltonian function $H \in C^\infty_0$ is in the support of $\muGPH^{\mathcal{D}}$.
    This follows immediately by making the necessary modifications to the proof of Lemma~\ref{lem:GP_H_full_support},
    since the coefficients as defined in the proof of that lemma are periodic functions for periodic Hamiltonian functions.
    Thus, the same argument applies.
\end{proof}
Together, these show Theorem~\ref{thm:full-support}.
It still remains to show the equivalent statement for autonomous Hamiltonian functions:
\begin{theorem}
    Let $\mathcal{D}$ be an autonomously exhaustive, frequency independent \ldd.
    Then the support of the measure $\muGP^{\mathcal{D}}$ is exactly the Hofer-closure $\overline{\Aut}(M,\omega)$ of $\Aut(M,\omega) \subset \Ham(M,\omega)$.
\end{theorem}
\begin{proof}
    This is essentially the same proof as the proof of Lemma~\ref{lem:GP_H_full_support}.
    First note that it suffices to show that any $\phi \in \Aut(M,\omega)$ is surrounded by an arbitrarily small ball of positive measure.
    Let $\phi \in \Aut(M,\omega)$ and $\eps > 0$ be arbitrary and let $F: M \to \Reals$ be an autonomous Hamiltonian function generating $\phi$.
    Then we can define the coefficients 
    \[f_n = \langle e_n, F \rangle_{L^2(M,\omega)}\]
    for all $n \in \Nats_{>0}$.
    Since $\mathcal{D}$ is autonomous, we have that the Gaussian processes $Z_n$ are constant on $[0,1]$ almost-surely.
    We therefore treat them as real-valued random variables in the following.
    Thus, by the same reasoning as in the proof of Lemma~\ref{lem:GP_H_full_support}, 
    \begin{align*}
        \muGPH^{\mathcal{D}}(\{G \mid \norm{F-G}_\infty < \eps\})
        &= \Prob\left[ \norm{H - F}_\infty < \eps \right] \\
        &\geq 
        \Prob\left[\bigcap_{n \ge 1} \left\{ \abs{f_n - Z_n} <              \frac{6\eps}{C''\pi^2w_n n^3} \right\} \right]
        \\ &= \prod_{n \ge 1} 1- \underbrace{\Prob\left[\abs{f_n - Z_n} \geq\frac{6\eps}{C''\pi^2w_n n^3} \right]}_{\eqqcolon p_n},
    \end{align*}
    where $C'' > 0$ is the constant from Lemma~\ref{lem:GP_H_full_support} and we use the fact that the $Z_n$ are independent by assumption.
    To show that this lower bound is positive, we again need to show that $\sum_{n \ge 1} p_n < \infty$ and $p_n < 1$ for all $n \in \Nats_{>0}$.
    Showing $p_n < 1$ is easy since by assumption, all $Z_n$ have positive variance, i.e.\ $\Var[Z_n] > 0$. 
    We need to estimate the probabilities $p_n$ from above to show that $\sum_{n \ge 1} p_n < \infty$.
    Note that for any $s > 0$ we have
    \begin{equation*}
        \Prob[\abs{f_n - Z_n} \geq s] \leq \Prob[\abs{Z_n - \E[Z_n]} \geq s + \abs{f_n - \E[Z_n]}],
    \end{equation*}
    since $\abs{Z_n - \E[Z_n]} \geq s + \abs{f_n - \E[Z_n]}$ implies $\abs{f_n - Z_n} \geq s$.
    Further, since 
    $Z_n - \E[Z_n]$ is a centered Gaussian variable with variance $\Var[Z_n]$, we obtain 
    \begin{align*}
        \Prob[\abs{Z_n - \E[Z_n]} \geq s + \abs{f_n - \E[Z_n]}]
        \leq 2\exp\left(-\frac{1}{2\Var[Z_n]}\left( s + \abs{f_n - \E[Z_n]}\right)^2 \right)
        \leq 2\exp\left(-\frac{1}{2D}s^2\right),
    \end{align*}
    where we use the fact that there exists some constant $D > 0$ such that $\sup_{n \in \Nats_{>0}} \Var[Z_n] < D$.
    This holds by the definition of a \ldd.
    Now, by our previous estimate of $p_n$, we have
    \begin{align*}
        p_n \leq 2\exp\left(-\frac{1}{2D}\left( \frac{6\eps}{C''\pi^2w_n n^3} - \norm{f_n - \E[Z_n]}\right)^2\right)
        ,
    \end{align*}
    for $n$ large enough so that $\frac{6\eps}{C''\pi^2w_n n^3} > \norm{f_n - \E[Z_n]}$.
    Note that $\{\E[Z_n]\}_{n \in \Nats_{>0}}$ is a bounded sequence and $\lim_{n \to \infty} f_n = 0$ since $F$ is smooth.
    Thus, the argument in the exponential above rapidly diverges to $-\infty$ as $n \to \infty$ since $w_n$ decays exponentially in $n$.
    Thus, $\sum_{n\ge q}p_n < \infty$, which completes the proof.
    \color{black}
\end{proof}
After studying the support of $\muGP^{\mathcal{D}}$, we now wish to give some lower bounds 
on the measure of balls.
In the theory of Gaussian processes, such ``small ball estimates'' are a well-studied topic.
The focus is on understanding how the measure of a small ball around $0$ behaves.
The following lemma is simply an adaptation of these bounds to our setting.
\begin{lemma}\label{lem:centered_ball_estimate}
    Let $\mathcal{D}$ be a centered \ldd .
    Then there exist constants $C,D \geq 0$ depending only on $\mathcal{D}$ such that for every $r > 0$, we have
    \begin{equation*}
        \muGP^{\mathcal{D}}\left(B_{\Hof}(\id_M,r + C)\right) \geq 1 - \exp\left(-\frac{r^2}{D}\right). 
    \end{equation*}
\end{lemma}
\begin{proof}
    This is direct consequence of the Borell-TIS inequality for the $\infty$-norm (see Lemma~\ref{lem:Borell-TIS-infinity})
    and 
    Lemma~\ref{lem:H_to_phi_H_Lipschitz}.
    The latter implies that 
    \begin{equation*}
        \muGP^{\mathcal{D}}\left(B_{\Hof}(\id_M,u)\right) \geq
        \muGPH^{\mathcal{D}}\left(\left\{G \in C^\infty_0 \Big\vert \norm{G}_\infty < \frac{u}{2}\right\}\right).
    \end{equation*}
    Here $u > 0$ is arbitrary.
    Now, Borell-TIS gives us the following estimate:
    \begin{align*}
        \muGPH^{\mathcal{D}}\left(\left\{G \in C^\infty_0 \Big\vert \norm{G}_\infty < \frac{u}{2}\right\}\right)
        &= \Prob\left[\norm{H}_\infty < \frac{u}{2}\right] 
        = 1- \Prob\left[\norm{H}_\infty \geq \frac{u}{2}\right] \\
        &\geq 1 - \exp\left( - \frac{1}{2 \sup_{t,x} \Var[H(t,x)]} \cdot \left( u - \E\left[\norm{H}_\infty\right]\right)^2 \right),
    \end{align*}
    when $u > \E\left[\sup_{t,x}{H(t,x)}\right]$.
    Here we use Borell-TIS in the last line and the fact that $\mathcal{D}$ is centered in the second line.
    Setting $C = {2} \E\left[\sup_{(t,x) \in [0,1] \times M}{H(t,x)}\right]$
    and $D = 8 \sup_{(t,x) \in [0,1] \times M} \Var[H(t,x)]$ gives us the desired estimate.
    Clearly, $C$ and $D$ only depend on the \ldd\ $\mathcal{D}$.
\end{proof}
It is also interesting to give estimates for balls centered at arbitrary elements $\phi \in \Ham(M,\omega)$.
For this, we first need to introduce a measure of the complexity of such a $\phi$ that roughly encodes how likely it is to occur.
Recall that the Gaussian process $H$ defined in~\eqref{eq:GP_H} has an associated reproducing kernel Hilbert space (from now on referred to as RKHS) $\Hil$.
As sets $\Hil \subset C^\infty_0$, 
but the Hilbert space structure is \textit{not} inherited from $L^2([0,1] \times M)$.
We denote the RKHS norm by $\norm{\cdot}_{\Hil}$.
With these ingredients, we make the following definition:
\begin{definition}
    Let $\mathcal{D}$ be a centered \ldd .
    Then we call the quantity 
    \begin{equation*}
        \norm{\phi}_{\mathcal{D}} \coloneqq \inf \left\{ \norm{H}_{\Hil} \Big\vert H \in C^\infty_0 \cap \Hil, \phi_H^1 = \phi \right\} \in (0,+\infty],
    \end{equation*}
    the \textit{$\mathcal{D}$-complexity} of $\phi$. We set $\inf \emptyset = \infty$.
\end{definition}
It is now possible to give a lower bound on the measure of balls centered at $\phi$ in terms of the $\mathcal{D}$-complexity of $\phi$
and the estimate for balls centered at the identity.
\begin{lemma}\label{lem:ball_estimate}
    Let $\mathcal{D}$ be a centered \ldd\ and let $\phi \in \Ham(M,\omega)$.
    For every $r > 0$,
    \begin{equation*}
        \muGP^{\mathcal{D}}(B_{\Hof}(\phi,r)) \ge \exp\left(-\frac{1}{2} \norm{\phi}_{\mathcal{D}}^{2}\right) \cdot p_r
    \end{equation*}
    where $p_r = \muGPH^{\mathcal{D}}(\{G \vert G \in C^\infty_0, \norm{G}_\infty < \frac{r}{2}\}) > 0$ is the $\frac{r}{2}$-ball probability of the Gaussian process $H$ defined in~\eqref{eq:GP_H} with respect to $\mathcal{D}$.
\end{lemma}
\begin{proof}
    Assume that $\norm{\phi}_{\mathcal{D}} < \infty$, since the statement is trivial otherwise.
    Let $\delta > 0$ be small and $F_\delta \in C^\infty_0$ such that $\phi = \phi_{F_\delta}^1$ and $\norm{F_\delta}_{\Hil} < \norm{\phi}_{\mathcal{D}} + \delta$.
    Note that 
    \begin{align*}
        \muGP^{\mathcal{D}}(B_{\Hof}(\phi,r+\delta)) &= \muGPH^{\mathcal{D}}\left(\left\{F \in C^\infty_0 \vert \phi_F^1 \in B_{\Hof}(\phi,r+\delta)\right\}\right) 
        \\ &\ge  \muGPH^{\mathcal{D}}\left( \left\{{F}_{\delta} \sharp G \vert G \in C^\infty_0, \phi_G^1 \in B_{\Hof}(\id,r)\right\}\right).
    \end{align*}
    Recall that Lemma~\ref{lem:H_to_phi_H_Lipschitz} implies that for any $G \in C_0^\infty$ with $\norm{G}_\infty < \frac{r}{2}$, we have $\norm{\phi_G^1}_{\Hof} < r$.
    Thus,
    \begin{align*}
        \left\{{F}_{\delta} \sharp G \vert G \in C^\infty_0, \phi_G^1 \in B_{\Hof}(\id,r)\right\}
        \supset 
        \left\{{F}_{\delta} \sharp G \vert G \in C^\infty_0, \norm{G}_\infty < \frac{r}{2}\right\}.
    \end{align*}
    Now notice that 
    $F_\delta \sharp G(t,x) \coloneqq F_\delta(t,x) + G(t,(\phi^t_{F_\delta})^{-1}(x))$,
    and that $\norm{t,x \mapsto G(t,(\phi^t_{F_\delta})^{-1}(x))}_\infty = \norm{G}_\infty$ since $\phi^t_{F_\delta}$ is a diffeomorphism for all $t \in [0,1]$.
    Thus,
    \begin{equation*}
        \left\{{F}_{\delta} \sharp G \vert G \in C^\infty_0, \norm{G}_\infty < \frac{r}{2}\right\} = \left\{{F}_{\delta} + G \vert G \in C^\infty_0, \norm{G}_\infty < \frac{r}{2}\right\}
    \end{equation*}
    To estimate this quantity, we notice that 
    \begin{align*}
        \norm{G}_\infty &= \max_{t,x} \left|G(t,x)\right| = \max_{t,x} \left|G(t,(\phi_{{F}_{\delta}}^t)^{-1}(x))\right| = \norm{{F}_{\delta}\sharp G}_\infty
    \end{align*}
    since $\phi_{{F}_{\delta}}^t$ is a diffeomorphism.
    This allows us to apply the Cameron-Martin theorem (Theorem~\ref{thm:cameron-martin}), to estimate 
    \begin{align*}
        \muGP^{\mathcal{D}}(B_{\Hof}(\phi,r+2\delta)) 
        &\ge \muGPH^{\mathcal{D}}\left(\left\{\tilde{F}_{\delta} + G \Big\vert G \in C^\infty_0, \norm{G}_\infty < \frac{r}{2}\right\}\right) \\
        &\ge \exp\left(-\frac{1}{2}\norm{{F}_{\delta}}_{\Hil}^2\right) \muGPH^{\mathcal{D}}\left(\left\{G \Big\vert G \in C^\infty_0, \norm{G}_\infty < \frac{r}{2}\right\}\right) \\
        &= \exp\left(-\frac{1}{2}\norm{{F}_{\delta}}_{\Hil}^2\right) \cdot p_{r}.
    \end{align*}
    Then, the claim follows as $\delta \to 0$, since 
    \begin{align*}
        \muGP^{\mathcal{D}}(B_{\Hof}(\phi,r)) &= \lim_{\delta \to 0} \muGP^{\mathcal{D}}(B_{\Hof}(\phi,r+\delta)) \\ 
        &\ge \lim_{\delta \to 0} \left( \exp\left(-\frac{1}{2}\norm{{F}_{\delta}}_{\Hil}^2\right) \cdot p_{r}\right) \\
        &= \exp\left(-\frac{1}{2} \norm{\phi}_{\mathcal{D}}^{2}\right) \cdot p_r,
    \end{align*}
    which completes the proof.
\end{proof}
By combining the proofs of Lemma~\ref{lem:centered_ball_estimate} and Lemma~\ref{lem:ball_estimate}, 
we obtain the following corollary, which gives a useful lower bound on the measure of balls centered at arbitrary elements of $\Ham(M,\omega)$.
\begin{corollary}\label{cor:ball_estimates}
    Let $\mathcal{D}$ be a centered \ldd .
    Then there exist constants $C,D \geq 0$ depending only on $\mathcal{D}$ such that for every $r > 0$ and $\phi \in \Ham(M,\omega)$, we have
    \begin{equation*}
        \muGP^{\mathcal{D}}\left(B_{\Hof}(\phi,C+r)\right) \geq \exp\left(-\frac{1}{2} \norm{\phi}_{\mathcal{D}}^{2}\right) \left( 1- \exp\left(-\frac{r^2}{D}\right)\right).
    \end{equation*}
\end{corollary}
\rmk{\label{rmk:symplectic_cameron_martin}
    Note that in the proof of Lemma~\ref{lem:ball_estimate}, we use the Cameron-Martin theorem for Gaussian measures on Banach spaces.
    At first glance, one might hope that there exists a similar theorem --- a kind of \textit{symplectic Cameron-Martin theorem} --- for $\muGP^{\mathcal{D}}$.
    This would be a statement as follows:
    There exists a constant $C > 0$ such that for any $\eps > 0$, we have that
    \[\muGP^{\mathcal{D}}(B_{\Hof}(\phi,\eps)) \ge \exp(-C \norm{\phi}_{?}^{2}) \muGP^{\mathcal{D}}(B_{\Hof}(\id,\eps)),\]
    where $\norm{\cdot}_{?}$ is either the Hofer norm or $d_{C^0}(\cdot,\id_M)$.
    However, not only can we not prove such a theorem, the statement is in fact false.
 
    First consider the $d_{C^0}$-case.
    Recall that $d_{C^0}(\phi, \id) = \sup_{x \in M} d(\phi(x), x)$ is the $C^0$-distance to the identity.
    We can consider for example the torus $\T^2$ with the standard symplectic form $\omega = dx \wedge dy$.
    Then there is a Hamiltonian diffeomorphism $\phi$ supported in an arbitrarily small neighborhood of the non-displacable curve $S^1 \times \{0\} \subset \T^2$
    that has arbitrarily large Hofer norm, see~\cite{polterovich-1998}.
    However, the $C^0$-distance to the identity is bounded by the diameter of $S^1 \times \{0\}$ (which is $1$ up to rescaling of the torus) plus a small error term 
    related to the size of the support of $\phi$.
    Let $\phi_a$ be such a Hamiltonian diffeomorphism with $\norm{\phi_a}_{\Hof} > a$.
    We also have $d_{C^0}(\phi_a, \id) \leq 1 + \eps$ for some $\eps > 0$ independent of $a$.
    Now, thus, if we had
    \[\muGP(B_\delta(\phi_a)) \geq \exp(C \norm{\phi_a}_{C^0} )\muGP(B_\delta(\id)) \geq C' \muGP(B_\delta(\id)),\]
    we would have a universal lower bound on the measure of the Hofer ball around $\phi_a$.
    This is not only a contradiction to Theorem~\ref{thm:expected-hofer-norm-finite}, but also to the fact that $\muGP$ is a probability measure since these Hofer balls are all disjoint 
    for small $\delta > 0$.

    One might expect that this is due to the fact that the $C^0$-distance does not see the Hamiltonian nature of $\phi_a$.
    However, even the Hofer norm does not admit such an estimate to be made.
    In fact, one might assume that $M$ is such a manifold that admits a bi-Lipschitz embedding of $\Reals^\infty$ into $\Ham(M,\omega)$.
    Such an embedding is for example given in~\cite{usher-2013} for some closed symplectic manifolds.
    Then, pick $\phi_1,\dots \in \Ham(M,\omega)$
    such that $\norm{\phi_i}_{\Hof} = r$ for some $r > 0$ and $d_{\Hof}(\phi_i, \phi_j) \geq \eps$ for all $i \not= j$ and some fixed $0 < \eps < \frac{r}{4}$.
    Now assume that we had an estimate of the form 
    \[ \muGP^{\mathcal{D}}(\phi U)\geq \exp(C \norm{\phi}_{\Hof}^2) \muGP^{\mathcal{D}}(U). \]
    Then, take the ball of radius $\frac{\eps}{2}$ around the identity.
    Clearly, the measure of this ball under $\muGP^{\mathcal{D}}$ is positive, since it has full support.
    An estimate of the form above would now imply that the measure of all $B_{\Hof}(\phi_i, \frac{\eps}{2})$
    under $\muGP^{\mathcal{D}}$
    is bound from below by $\exp(-Cr^2) > 0$.
    Since all of these balls are disjoint and there are infinitely many of them, this is clearly a contradiction.
    Thus, no such estimate (in terms of the Hofer norm) is possible.
    This also illustrates why the RKHS norm has to show up in the estimate of Lemma~\ref{lem:ball_estimate}:
    This norm rules out counterexamples like the one above, since it behaves qualitatively different from the $\infty$- or $\osc$-norms.
    For example, the unit ball with respect to the RKHS norm is compact in $C^\infty_0$ equipped with the $\infty$-norm (!), see~\cite[Proposition 2.6.9]{gine-nickl-2016}.
}
There is an interesting follow-up to Remark~\ref{rmk:symplectic_cameron_martin}:
It tells us something about those Hamiltonian functions which generate infinite-dimensional
quasi-flats (i.e.\ quasi-isometric embeddings of $(\Reals^{\infty},d_{\infty})$ into $(\Ham(M,\omega),d_{\Hof})$).
Namely, it follows
that the RKHS norm of such Hamiltonians must be unbounded.
Notice that this is true not only for some family 
of Hamiltonian functions 
that generate an infinite-dimensional quasi-flat,
but for any family of Hamiltonian functions that generate Hamiltonian diffeomorphisms
forming an infinite-dimensional quasi-flat in $\Ham(M,\omega)$.
The following theorem makes this precise:
\begin{theorem}\label{thm:quasi-flats-RKHS-norm}
    Let $\mathcal{D}$ be a centered \ldd .
    Further, let $\{\phi_n\}_{n \in \Nats} \subset \Ham(M,\omega)$ be a sequence of Hamiltonian diffeomorphisms such that
    there exist $\eps,\delta > 0$ such that $\norm{\phi_n}_{\Hof} \leq \eps$ for all $n \in \Nats$
    and $d_{\Hof}(\phi_n,\phi_m) \geq 2\delta$ for all $n \neq m$.
    Then,
    \begin{equation*}
        \lim_{n \to \infty} \norm{\phi_n}_{\mathcal{D}} = \infty.
    \end{equation*}
    In particular, for any $\{H_n\}_{n \in \Nats} \subset C^\infty_0$ such that $\phi_{H_n}^1 = \phi_n$ for all $n \in \Nats$, we have
    \begin{equation*}
        \lim_{n \to \infty} \norm{H_n}_{\Hil} = \infty.
    \end{equation*}
\end{theorem}
\begin{proof}
    Assume the contrary, i.e.\
    there is a subsequence of $\{\phi_{n_k}\}_{k \in \Nats}$, for 
    which $\norm{\phi_{n_k}}_{\mathcal{D}} \leq C$ for some $C > 0$ and all $k \in \Nats$.
    Then, by Lemma~\ref{lem:ball_estimate}, we have that
    \begin{align*}
        \muGP^{\mathcal{D}}\left(\bigcup_{k \in \Nats} B_{\Hof}\left(\phi_{n_k}, \delta\right)\right) 
        &= \sum_{k \in \Nats} \muGP^{\mathcal{D}}\left(B_{\Hof}\left(\phi_{n_k}, \delta\right)\right) \\
        &\ge \sum_{k \in \Nats}
        \exp\left(-\frac{1}{2} \norm{\phi_{n_k}}_{\mathcal{D}}^{2}\right) \underbrace{p_\delta}_{>0} 
        \ge \sum_{k \in \Nats}
        \exp\left(-\frac{1}{2}C^{2}\right) p_\delta = \infty,
    \end{align*}
    where we use the fact that the balls $B_{\Hof}(\phi_{n_k},\delta)$ are disjoint by assumption in the first line.
    Clearly, this is a contradiction to the fact that $\muGP^{\mathcal{D}}$ is a probability measure.
    The second part of the theorem follows directly from the definition of $\norm{\cdot}_{\mathcal{D}}$.
\end{proof}
Note that for some \ldds\ the explicit computation of the RKHS norm is possible. 
Indeed, Remark~\ref{rmk:D2-RKHS} outlines a particularly simple case.
Using this, we can obtain Theorem~\ref{thm:smoothness-quasi-flats}, which notably makes no mention of anything probabilistic.
\begin{proof}[Proof of Theorem~\ref{thm:smoothness-quasi-flats}.]
    This theorem really is a corollary of a combination of previous results.   
    First note that by Theorem~\ref{thm:existence-of-ldd},
    there is a \ldd\ $\mathcal{D}_2$ that is centered, time-symmetric, periodic, periodically exhaustive, and frequency unbiased.
    Furthermore, we have computed the RKHS norm of the Gaussian process $H$ defined in~\eqref{eq:GP_H} with respect to $\mathcal{D}_2$ 
    explicitly in Remark~\ref{rmk:D2-RKHS}.
    Combining this explicit formula for $\norm{\cdot}_{\Hil}$ with Theorem~\ref{thm:quasi-flats-RKHS-norm} 
    yields the desired result.
\end{proof}
\rmk{
    With a little more effort, one could prove Theorem~\ref{thm:smoothness-quasi-flats}
    with any super-polynomial (in $k$) weights instead of $\exp\left(\eps \lambda_k \right)$.
    However, since this is not allowed in the current framework of \ldds\ (see Definition~\ref{def:ldd}), we stick to the exponential weights.
}

To conclude this section, we prove Theorem~\ref{thm:inversion-invariance}, which states that the measure $\muGP^{\mathcal{D}}$ is invariant under inversion if $\mathcal{D}$ 
is time-symmetric and centered.
\begin{proof}[Proof of Theorem~\ref{thm:inversion-invariance}.]
    Let $\mathcal{D}$ be a time-symmetric and centered \ldd\ and let us consider an arbitrary $U \subset \Ham(M,\omega)$ in the Borel $\sigma$-field generated by the Hofer topology.
    We need to show that $\muGP^{\mathcal{D}}(U) =\muGP^{\mathcal{D}}(U^{-1})$ where $ U^{-1} \coloneqq \left\{\phi^{-1}\mid \phi \in U\right\}$.
    First, we define the set $\mathcal{F}_U \coloneqq \left\{F \in C^\infty_0 \vert \phi_F^1 \in U\right\}$, of 
    all Hamiltonian functions $F$ such that the time-1 flow $\phi_F^1$ is in $U$.
    Let $F \in \mathcal{F}_U$, then by Lemma~\ref{lem:flows_composition_and_inversion}, we have that $\hat{F}(t,x) = -F(1-t,x)$ has time-$1$ flow $\phi_{\hat{F}}^1 = \phi_F^{-1}$.
    Thus, $F \in \mathcal{F}_{U}$ implies $\hat{F} \in \mathcal{F}_{U^{-1}}$.
    Conversely, if $F \in \mathcal{F}_{U^{-1}}$, then by the same reasoning, we have that $\hat{F} \in \mathcal{F}_{U}$.
    Since $-\hat{F}(1-t,x) = F(t,x)$, we have that $\hat{F} \in \mathcal{F}_{U}$ if and only if $F \in \mathcal{F}_{U^{-1}}$.
    Thus, it suffices to show that $\muGPH^{\mathcal{D}}(\mathcal{F}_U) = \muGPH^{\mathcal{D}}(\mathcal{F}_{U^{-1}})$ to obtain the desired result.
    By construction, we have that $\muGPH^{\mathcal{D}}$ is the law of the Gaussian process $H$ defined in~\eqref{eq:GP_H}.
    We make the following observation, which completes the proof: 
    \begin{align*}
        \muGPH^{\mathcal{D}}(\mathcal{F}_U) &= \Prob[H \in \mathcal{F}_U] = \Prob[\hat{H} \in \mathcal{F}_{U^{-1}}] 
        = \Prob\left[\left(t,x \mapsto -\sum_{n \ge 1} w_n \cdot Z_n(1-t) \cdot e_n(x)\right) \in \mathcal{F}_{U^{-1}}\right] \\
        &= \Prob\left[\left(t,x \mapsto -\sum_{n \ge 1} w_n \cdot Z_n(t) \cdot e_n(x)\right) \in \mathcal{F}_{U^{-1}}\right] \\
        &= \Prob\left[\left(t,x \mapsto \sum_{n \ge 1} w_n \cdot Z_n(t) \cdot e_n(x)\right) \in \mathcal{F}_{U^{-1}}\right] = \muGPH^{\mathcal{D}}(\mathcal{F}_{U^{-1}}),
    \end{align*}
    where we use the $Z_n(\cdot)$ and $Z_n(1-\cdot)$ have the same law because $\mathcal{D}$ is time-symmetric in the second line and 
    that $Z_n$ and $-Z_n$ have the same law since $\mathcal{D}$ is centered in the third line.
\end{proof}
\section{Random walks on the Hamiltonian diffeomorphism group}\label{sec:random-walk}
We now wish to prove Theorem~\ref{thm:random-walk-law} and Theorem~\ref{thm:random-walk-filling}.
Recall that if $\mathcal{D}$ is an autonomous \ldd, 
then the associated random walk on the Hamiltonian diffeomorphism group is constructed
as follows:
Let $\phi_0,\phi_1,\ldots$ be independent and identically distributed $\Aut(M,\omega)$-valued random variables with law $\muAut^{\mathcal{D}}$.
Then the random walk $\Phi_n \coloneqq \phi_n \circ \cdots \circ \phi_1$ is a $\Ham(M,\omega)$-valued random variable for any $n \in \Nats$.
\begin{proof}[Proof of Theorem~\ref{thm:random-walk-law}]
    Let $n \in \Nats$ be fixed.
    Now let $\mathcal{D}^{(1)},\dots,\mathcal{D}^{(n)}$ be independent copies of $\mathcal{D}$,
    by this we mean that the coefficients $Z_k^{(i)}$ of $\mathcal{D}^{(i)}$ and $Z_k^{(j)}$ of $\mathcal{D}^{(j)}$ are independent and identically distributed 
    for all $k \ge 1$ if $i \neq j$.
    In particular, 
    we assume that all $Z_k^{(i)}$ are defined on the same probability space $(\Omega,\mathcal{F},\Prob)$.
    Recall that since $\mathcal{D}$ is autonomous, $Z_k^{(i)}$ is almost-surely constant for all $k \ge 1$ and $i \in \{1,\ldots,n\}$.
    Thus, we treat them as real-valued random variables without loss of generality.
    Now, 
    we notice that 
    if we 
    set 
    \begin{equation*}
        H^{(i)}(t,x) = \sum _{n \ge 1} w_n \cdot Z_n^{(i)} \cdot e_n(x).
    \end{equation*}
    based on~\eqref{eq:GP_H},
    then $\Phi_n$ and $\phi_{H^{(n)}}^1 \circ \cdots \circ \phi_{H^{(1)}}^1$ have the same law.
    Now by Lemma~\ref{lem:composition_of_autonomous_flows}, 
    we can also write down a Hamiltonian function that generates $\phi_{H^{(n)}}^1 \circ \cdots \circ \phi_{H^{(1)}}^1$
    as its time-$1$ flow.
    Let $\beta: \Reals \to \Reals_{\geq 0}$
    be a smooth bump function such that $\supp \beta \Subset (0,1)$ and $\int_0^1 \beta(t) dt = 1$.
    Then,
    the Hamiltonian function
    \begin{equation*}
        H^{\text{all}}(t,x) = 
        \sum_{i=1}^n n\cdot\beta(nt-i+1) \cdot H^{(i)}(x),
    \end{equation*}
    satisfies $\phi_{H^{(n)}}^1 \circ \cdots \circ \phi_{H^{(1)}}^1 = \phi_{H^{\text{all}}}^1$.
    We are now almost done.
    Note that 
    \begin{align*}
        H^{\text{all}}(t,x) &= 
        \sum_{i=1}^n n\cdot\beta(nt-i+1) \cdot \left(\sum _{k \ge 1} w_k \cdot Z_k^{(i)} \cdot e_k(x)\right) \\
        &= \sum _{k \ge 1} w_k \cdot \left( \sum_{i=1}^n n\cdot\beta(nt-i+1) \cdot Z_k^{(i)} \right)\cdot e_k(x),
    \end{align*}
    where we use the rapid decay of the weights $w_n$ to exchange the order of summation.
    Thus, if we set
    \begin{equation*}
        \tilde{Z}_k(t) = \sum_{i=1}^n n\cdot\beta(nt-i+1) \cdot Z_k^{(i)},
    \end{equation*}
    for $k \ge 1$ we obtain a sequence of Gaussian processes $\tilde{Z}_k$ on $[0,1]$.
    Let $\mathcal{D}_n$ be the \ldd\ defined by replacing the coefficient processes of $\mathcal{D}$ 
    with $\tilde{Z}_k$.
    Note that we have 
    \begin{equation*}
        H^{\text{all}}(t,x) = \sum _{k \ge 1} w_k \cdot \tilde{Z}_k \cdot e_k(x) = H(t,x),
    \end{equation*}
    where $H$ is the Gaussian process~\eqref{eq:GP_H} defined with respect to $\mathcal{D}_n$.
    Thus, $\phi_H^1$ and $\Phi_n$ have the same law.
    By construction, this implies that the law of $\Phi_n$ is given by $\muGP^{\mathcal{D}_n}$.
\end{proof}
Under some light assumptions, namely that the initial \ldd\ $\mathcal{D}$ is centered,
we can obtain a symmetry condition for the random walk $\Phi_n$.
This is formalized in Theorem~\ref{thm:random-walk-symmetric},
which we will prove next.
\begin{proof}[Proof of Theorem~\ref{thm:random-walk-symmetric}]
    We reuse the notation of the previous proof.
    Note that if $\mathcal{D}$ is centered,
    then $Z_k^{(i)}$ is a centered Gaussian random variable for all $k \ge 1$ and $i \in \{1,\ldots,n\}$.
    Thus, 
    \begin{equation*}
        \E\left[\tilde{Z}_k(t)\right] = \sum_{i=1}^n n\cdot\beta(nt-i+1) \cdot \E\left[Z_k^{(i)}\right] = 0,
    \end{equation*}
    for any $k \ge 1$, which implies that $\mathcal{D}_n$ is centered.
    To show symmetry, we must assume 
    that the bump function $\beta$ is symmetric around $\frac{1}{2}$, i.e.\ $\beta(\frac{1}{2}+t) = \beta(\frac{1}{2}-t)$ for all $t \in \Reals$.
    Then, we have
    \begin{align*}
        \tilde{Z}_k(1-t) &= \sum_{i=1}^{n} n\cdot\beta(n(1-t)-i+1) \cdot Z_k^{(i)} \\
        &= \sum_{i=1}^{n} n\cdot\beta\left(\frac{1}{2}+\frac{1}{2}-i-n (t-1)\right) \cdot Z_k^{(i)} 
        = \sum_{i=1}^{n} n\cdot\beta(i+n (t-1)) \cdot Z_k^{(i)}  \\
        &= \sum_{j=1}^{n} n\cdot\beta((n-j+1)+n (t-1)) \cdot Z_k^{(n-j+1)} \\
        &= \sum_{j=1}^{n} n\cdot\beta(n t-j+1) \cdot Z_k^{(n-j+1)},
    \end{align*}
    since all $Z_k^{(i)}$ are independent and identically distributed, 
    this implies that $\tilde{Z}_k(t)$ and $\tilde{Z}_k(1-t)$ have the same law. 
    Thus, we have that $\mathcal{D}_n$ is time-symmetric.
    It follows by Theorem~\ref{thm:inversion-invariance}, 
    that the law of $\Phi_n$ is invariant under inversion.
\end{proof}
We will conclude this section by proving Theorem~\ref{thm:random-walk-filling},
which states that any Hofer ball has a positive probability of being hit by the random walk 
if $\mathcal{D}$ is autonomously exhaustive.
\begin{proof}[Proof of Theorem~\ref{thm:random-walk-filling}]
    Let $\phi \in \Ham(M,\omega)$ and $\eps > 0$.
    Furthermore, let $n \geq \norm{\phi}_{\Aut}$ be fixed.
    Then there are $\psi_1,\dots,\psi_n \in \Aut(M,\omega)$ such that $\phi = \psi_n\circ \cdots \circ \psi_1$.
    Let $\tilde{\psi}_1, \dots, \tilde{\psi}_n$ be any collection of Hamiltonian diffeomorphisms
    with $d_{\Hof}(\psi_i,\tilde{\psi}_i) < \frac{\eps}{n+1}$ for all $i \in \{1,\ldots,n\}$.
    We can now make the following observation:
    \begin{align*}
        d_{\Hof}(\psi_n\cdots\psi_1,\tilde{\psi}_n\cdots\tilde{\psi}_1) &= 
        d_{\Hof}(\tilde{\psi}_n^{-1}\psi_n\cdots\psi_1,\tilde{\psi}_{n-1}\cdots\tilde{\psi}_1) \\
        &= \norm{\tilde{\psi}_n^{-1}\psi_n\psi_{n-1}\cdots\psi_1\tilde{\psi}^{-1}_{1}\cdots\tilde{\psi}_{n-1}^{-1}}_{\Hof}
        \\ &\leq \norm{\tilde{\psi}_n^{-1}\psi_n}_{\Hof} + \norm{\psi_{n-1}\cdots\psi_1\tilde{\psi}^{-1}_{1}\cdots\tilde{\psi}_{n-1}^{-1}}_{\Hof} \\
        &= d_{\Hof}(\psi_n,\tilde{\psi}_n) + d_{\Hof}(\psi_{n-1}\cdots\psi_1,\tilde{\psi}_{n-1}\cdots\tilde{\psi}_1).
    \end{align*}
    Thus, $d_{\Hof}(\psi_n\cdots\psi_1,\tilde{\psi}_n\cdots\tilde{\psi}_1) < \eps$.
    It follows that 
    \begin{align*}
        \Prob\left[ \Phi_n \in B_{\Hof}(\phi,\eps)\right] &\geq
        \prod_{k=1}^n \Prob\left[ \phi_k \in B_{\Hof}\left(\phi_k,\frac{\eps}{n+1}\right)\right], 
    \end{align*}
    where we use the independence of $\phi_1,\dots,\phi_n$.
    Since $\mathcal{D}$ is autonomously exhaustive, 
    we have that $\Prob\left[ \phi_k \in B_{\Hof}\left(\phi_k,\frac{\eps}{n+1}\right)\right] > 0$
    for any $k \in \{1,\ldots,n\}$.
    Thus, $\Prob\left[ \Phi_n \in B_{\Hof}(\phi,\eps)\right] > 0$.
\end{proof}
\rmk{Note that for any $x \in M$, the random walk $(\Phi_n(x))_{n \in \Nats}$ is a point process or random walk on the manifold $M$.
Thus, the random walk on $\Ham(M,\omega)$ induces a random walk on $M$, by evaluation at the point $x$.
It would be interesting to study the properties of this induced random walk on $M$, and how they relate to the properties of the random walk on $\Ham(M,\omega)$.} %
\section{Limiting behavior}
In this section, we will prove several results on the limiting behavior of the measures $\muGP^{\mathcal{D}}$ as the regularity $\reg$ of the \ldd\ $\mathcal{D}$
approaches $0$ or $\infty$.
We will first study the limit $\reg \to \infty$,
as this case is vastly simpler.
The following lemma directly implies Theorem~\ref{thm:limit-reg-infty}, by the construction of $\muGP^{\mathcal{D}}$.
\begin{lemma}
    Let $\mathcal{D}$ be a centered \ldd. 
    As $\reg \to \infty$, we have $\muGPH^{\mathcal{D}_\reg} \weakto \delta_{0}$.
\end{lemma}
\begin{proof}
    Recall that 
    for any Borel probability measures $\{\mu_n\}_{n\in\Nats}, \mu$ on a metric space $X$, we have that $\mu_n \weakto \mu$ if and only if
    $\liminf_{n\to\infty} \mu_n(A) \geq \mu(A)$ for all open sets $A \subset X$.
    This is part of the Portmanteau lemma or Portmanteau theorem, see e.g.~\cite[Thorem 11.1.1]{dudley-2002}.
    We will now apply this to the situation at hand.
    Let $U \subset C^\infty_0$ be an open set in the $C^0$-topology.
    If $0 \notin U$, then $\delta_{0}(U) = 0$.
    Thus, we automatically get 
    \begin{equation*}
        \liminf_{\reg \to \infty} \muGPH^{\mathcal{D}_\reg}(U) \geq 0 = \delta_{0}(U).
    \end{equation*}
    Thus, let us assume that $0 \in U$.
    Since $U$ is open, there is an $\eps > 0$ such that $B_{\infty}(0,\eps) \subset U$.
    Denote by $H^{(\reg)}$ the Gaussian process associated to $\mathcal{D}_\reg$
    via~\eqref{eq:GP_H}.
    Without loss of generality, we can assume that all the $H^{(\reg)}$ are defined on the same probability space.
    Furthermore, we write $w^{(\reg)}_n = \exp(-\frac{1}{2} \reg \lambda_n)$
    for the respective weights.
    In the proof of Lemma~\ref{lem:expected-osc-norm},
    we showed that $\E[\norm{H^{(\reg)}}_{\infty}] < C' \cdot {\sum_{n=1}^\infty w^{(\reg)}_n \lambda_n^{\dim M}}$
    for some constant $C' > 0$ independent of $\reg$.
    Thus, as $\reg \to \infty$, we have $\E[\norm{H^{(\reg)}}_{\infty}] \to 0$.
    Let us assume that $\reg \gg 0$ is large enough such that $\E[\norm{H^{(\reg)}}_{\infty}] < \frac{\eps}{2}$.
    Then, by the Borell-TIS inequality for the $\infty$-norm (Lemma~\ref{lem:Borell-TIS-infinity}),
    we have 
    \begin{align*}
        \muGPH^{\mathcal{D}_\reg}(U) &\geq \muGPH^{\mathcal{D}_\reg}(B_{\infty}(0,\eps))
        = \Prob\left[\norm{H^{(\reg)}}_{\infty} < \eps\right] 
        \geq \Prob\left[\norm{H^{(\reg)}}_{\infty} < \frac{\eps}{2} + \E[\norm{H^{(\reg)}}_{\infty}] \right] 
        \\ &=1 - \Prob\left[\norm{H^{(\reg)}}_{\infty} \geq \frac{\eps}{2} + \E[\norm{H^{(\reg)}}_{\infty}] \right] 
        \\ &\geq1 - \exp\left(-\frac{\eps^2}{8\sigma_\reg^2}\right),
    \end{align*}
    where $\sigma_\reg^2 = \sup \{\Var[H^{(\reg)}(t,x)] \vert t \in [0,1], x \in M\}$.
    Recall that
    \begin{equation*}
    \Var[H^{(\reg)}(t,x)] = \sum_{\substack{n \geq 1 \\ m \geq 1}} w^{(\reg)}_n w^{(\reg)}_m \cdot \Cov\left[Z_n(t),  Z_m(t)\right] \cdot e_n(x_1) e_m(x_2).
    \end{equation*}
    It follows that $\sigma_\reg^2 \to 0$ as $\reg \to \infty$, due to the exponential decay of the weights $w^{(\reg)}_n$ as $\reg \to \infty$.
    Thus, we have $\liminf_{\reg \to \infty} \muGPH^{\mathcal{D}_\reg}(U) \geq 1 = \delta_{0}(U)$
    by the above estimate.
    Since $U$ was arbitrary, we conclude (by the Portmanteau lemma) that $\muGPH^{\mathcal{D}_\reg} \weakto \delta_{0}$ as $\reg \to \infty$.
\end{proof}

Next, we will study the limit $\reg \to 0$.
As stated in the introduction, this limit is vastly more complicated.
In particular, the limit depends heavily on the choice of \ldd.
We will start by studying the easier case, namely that of \ldds\ of $C^0$-limiting type.
\begin{lemma}\label{lem:GPH-limit-C0}
    Let $\mathcal{D}$ be a \ldd\ of $C^0$-limiting type.
    Then the measures $\muGPH^{\mathcal{D}_\reg}$ converge weakly to a measure 
    on $C^0([0,1] \times M)$ as $\reg \to 0$.
\end{lemma}
\begin{proof}
    Let us again   
    denote by $H^{(\reg)}$ the Gaussian process associated to $\mathcal{D}_\reg$
    via~\eqref{eq:GP_H}.
    Then we have
    \begin{equation*}
        H^{(\reg)}(t,x) = \sum_{n=1}^\infty w^{(\reg)}_n Z_n(t) e_n(x).
    \end{equation*}
    We want to show that 
    \begin{equation}\label{eq:convergence-to-C0}
        \lim_{\reg \to 0} H^{(\reg)}(t,x) = \sum_{n=1}^\infty Z_n(t) e_n(x) 
    \end{equation}
    holds almost-surely.
    To show this, we notice that 
    \begin{align*}
       \E\left[\norm{t,x \mapsto Z_n(t)e_n(x)}_\infty \right]
       &\leq \E\left[\norm{Z_n}_\infty \cdot \norm{e_n}_\infty\right] = \E[\norm{Z_n}] \cdot \norm{e_n}_\infty \\
       &= \E[\norm{\alpha_n Z_1}_\infty] \cdot \norm{e_n}_\infty
       \leq \abs{\alpha_n} \cdot \E[\norm{Z_1}_\infty]\cdot \norm{e_n}_\infty.
    \end{align*}
    where $\alpha_n$ is set such that $Z_n$ and $\alpha_n Z_1$ have the same law.
    By the $C^0$-limiting type assumption, we have $\sum_{n=1}^\infty \abs{\alpha_n} \cdot \norm{e_n}_\infty < \infty$.
    Thus, we obtain that 
    \begin{align*}
        \E\left[\norm*{\sum_{n=1}^\infty Z_n(t) e_n(x)}_\infty\right] < \infty,
    \end{align*}
    and thus the series converges almost-surely (and uniformly).
    This implies, due to 
    the fact that $w_n^{(\reg)} \to 1$ as $\reg \to 0$ for all $n \geq 1$,
    that the convergence in~\eqref{eq:convergence-to-C0} holds almost-surely.
    With this step, we are almost done.
    It is a standard result, that almost-sure convergence of random variables implies convergence in distribution.
    Thus, we define the $C^0([0,1] \times M)$-valued random variable 
    \begin{align*}
        \hat{H}(t,x) = \sum_{n=1}^\infty Z_n(t) e_n(x),
    \end{align*}
    and let $\hat{\mu}$ be the law of $\hat{H}$.
    Then, if by abuse of notation we identify $\muGPH^{\mathcal{D}_\reg}$ with its push-forward to $C^0([0,1] \times M)$,
    we have $\muGPH^{\mathcal{D}_\reg} \weakto \hat{\mu}$ as $\reg \to 0$.
    This completes the proof.
\end{proof}
With this lemma in hand, the proof of Theorem~\ref{thm:limit-reg-0-C0} follows easily.
\begin{proof}[Proof of Thorem~\ref{thm:limit-reg-0-C0}]
    This is essentially an extension of the proof of Lemma~\ref{lem:GPH-limit-C0}.
    We reuse all notation.
    Then we can consider the 
    $\widehat{\Ham}{}^{\Hof}(M,\omega)$-valued random variables
    $\phi_{H^{(\reg)}}^1$.
    Note that these are of course actually valued in $\Ham(M,\omega)$,
    but we will need the completion for the limit.
    Recall that, by Lemma~\ref{lem:H_to_phi_H_Lipschitz},
    we have that
    \begin{equation*}
        d_{\Hof}(\phi_{H}^1,\phi_F^1) \leq 2 \norm{H-F}_\infty.
    \end{equation*}
    Thus, since we have shown that $H^{(\reg)} \to \hat{H}$ almost surely as $\reg \to 0$,
    we conclude that the random variables
    $\{\phi_{H^{(\reg)}}^1\}_{\reg > 0}$ form an almost-surely convergent sequence (in the $d_{\Hof}$-completion).
    Thus, we can define the $\widehat{\Ham}{}^{\Hof}(M,\omega)$-valued random variable
    \begin{equation*}
        \hat{\phi} = \lim_{\reg \to 0} \phi_{H^{(\reg)}}^1.
    \end{equation*}
    We again have almost-sure convergence, and thus convergence in distribution of the laws.
    Thus, if $\widehat{\muGP}^{\mathcal{D}_0}$ is the law of $\hat{\phi}$,
    then $\muGP^{\mathcal{D}_\reg} \weakto \widehat{\muGP}^{\mathcal{D}_0}$ as $\reg \to 0$.
\end{proof}
We will now study the $C^1$-limiting type case in more detail.
Note that while in the previous proof, 
we only obtained a measure on an abstract metric completion of $\Ham(M,\omega)$ in the limit, 
the limiting measure will now live on $\Hameo(M,\omega)$, which is a subspace of $\Homeo(M,\omega)$.
\begin{proof}[Proof of Theorem~\ref{thm:limit-reg-0-C1}]
    We will use the following comparison between the Hofer and $C^0$-norms for a Hamiltonian diffeomorphism $\phi$:
    \begin{equation*}
        d_{C^0}(\phi,\id_M) \leq C \cdot \sqrt{d_{\Hof}(\phi,\id_M)} \cdot \norm{D\phi}_{\infty},
    \end{equation*}
    where $C$ only depends on $(M,\omega)$ and the choice of Riemannian metric $g$.
    This result is due to Joksimović and Seyfaddini in~\cite{joksimovic-seyfaddini-2024}.
    Of course, we will choose $g$ as specified by the almost-complex structure $J$ 
    that is part of the \ldd\ $\mathcal{D}$.

    Let us again reuse all notation from the previous proofs.
    Now note that for any $\phi,\psi \in \Ham(M,\omega)$, we have
    \begin{equation*}
        d_{C^0}(\phi,\psi) = \max_{x \in M} d(\phi(x),\psi(x)) = \max_{y \in M}d(\phi(\phi^{-1}(y)),\psi(\phi^{-1}(y))) = d_{C^0}(\id_M, \phi\psi^{-1}).
    \end{equation*}
    Thus, we have that 
    \begin{equation*}
        d_{C^0}(\phi,\psi) \leq C \cdot \sqrt{d_{\Hof}(\phi,\psi)} \cdot \norm{D(\phi\psi^{-1})}_\infty.
    \end{equation*}
    We have already established that $\{\phi_{H^{(\reg)}}^1\}_{\reg > 0}$ 
    converges with respect to the Hofer metric as $\reg \to 0$.
    Since the $C^0$-metric is complete on $\Homeo(M,\omega)$,
    we simply need to show that $\norm{D(\phi_{H^{(\reg)}}^1 (\phi_{H^{(\reg')}}^1)^{-1})}_\infty$ 
    does not blow up as $\reg,\reg' \to 0$.

    It suffices to control $\norm{D(\phi_{H^{(\reg)}}^1 (\phi_{H^{(\reg')}}^1)^{-1})}_\infty$
    for a fixed $\reg > 0$ and $\reg' < \reg$ varying.
    We first 
    note that
    \begin{equation*}
        \norm{D(\phi_{H^{(\reg)}}^1 (\phi_{H^{(\reg')}}^1)^{-1})}_\infty \leq \norm{D\phi_{H^{(\reg)}}^1}_\infty \cdot \norm{D\phi_{H^{(\reg')}}^1}_\infty,
    \end{equation*}
    where we use the chain rule and that the operator norm is submultiplicative and invariant under inversion.

    We will now derive a general bound for $\norm{D\phi_F^1}_\infty$ for a Hamiltonian function $F \in C_0^\infty$.
    For any $ x\in M$, we have that $\phi^t_H(x)$ satisfies the ODE $\frac{d}{dt}\phi^t_H(x) = X_{F_t}(\phi^t_H(x))$
    with initial condition $\phi^0_H(x) = x$.
    Thus, if we denote by $D\phi^t_H(x)$ the differential of $\phi^t_H$, 
    we obtain that is satisfies the ODE $\frac{d}{dt}D\phi^t_H(x) = D X_{F_t}(\phi^t_H(x)) \circ D\phi^t_H(x)$
    with initial condition $D\phi^0_H(x) = \id_{T_x M}$.
    Now let $u(t) = \norm{D\phi^t_H(x)}_{op}$, where $\norm{\cdot}_{op}$ is the operator norm on $T_x M$ induced by the Riemannian metric $g(\cdot,\cdot) = \omega(\cdot,J\cdot)$.
    Then, we have
    $\frac{d}{dt}u(t) \leq \norm{D X_{F_t}(\phi^t_H(x))}_{op} \cdot u(t)$.
    Thus, by Grönwall's inequality, we have
    \begin{equation*}  
        u(t) \leq \exp\left(\int_0^t \norm{D X_{F_s}(\phi^s_H(x))}_{op} ds\right).
    \end{equation*}
    Note that $\norm{D\phi^t_H(x)}_{op} \leq u(t)$ for all $t \in [0,1]$.
    It follows, since $x \in M$ was arbitrary, that
    \begin{equation*}
        \norm{D\phi_H^t}_\infty \leq \exp\left( \int_0^t \norm{D X_{F_s}}_\infty ds \right).
    \end{equation*}
    Now, we should further note that $X_{F_t}(x) = J(x) \nabla F_t(x)$,
    where $\nabla$ is the Riemannian gradient with respect to the metric $g$.
    It follows that $D X_{F_s} = DJ \nabla F_s + J \nabla^2 F_s$,
    where $\nabla^2$ is the Riemannian Hessian with respect to the metric $g$.
    This implies that 
    \begin{equation*}
        \norm{D\phi_H^t}_\infty \leq \exp\left( \norm{DJ}_\infty \int_0^t \norm{\nabla F_s}_\infty ds + \norm{J}_\infty \int_0^t \norm{\nabla^2 F_s}_\infty ds \right).
    \end{equation*}
    While in general this is not a very useful bound,
    the fact that our Gaussian process Hamiltonian functions $H^{(\reg)}$ are built from 
    eigenfunctions of the Laplacian allows us to control these terms.

    In particular, note that 
    \begin{align*}
        \nabla H^{(\reg)}(t,x)&= \sum_{n=1}^\infty w^{(\reg)}_n Z_n(t) \nabla e_n (x)
    \end{align*}
    and 
    \begin{align*}
        \nabla^2 H^{(\reg)}(t,x)&= \sum_{n=1}^\infty w^{(\reg)}_n Z_n(t) \nabla^2 e_n (x).
    \end{align*}
    We will now appeal to two useful facts about eigenfunctions of the Laplacian:
    In~\cite{shi-xu-2010} the authors show that there is a constant $C > 0$ depending only on $(M,g)$ such that
    $\norm{\nabla e_n}_\infty \leq C \lambda_n \norm{e_n}_\infty$ for all $n \geq 1$.
    In~\cite{cheng-thalmaier-wang-2024} the authors extend this result to the Hessian, i.e.\ there is a constant $C' > 0$ depending only on $(M,g)$ such that $\norm{\nabla^2 e_n}_\infty \leq C' \lambda_n \norm{e_n}_\infty$ for all $n \geq 1$.
    Now, since $0 < w^{(\reg)}_n < 1$,
    we have that 
    \begin{align*}
        \norm{\nabla H^{(\reg)}(t,x)}_\infty &\leq \sum_{n=1}^\infty \abs{Z_n(t)} \cdot \norm{\nabla e_n}_\infty 
        \leq C \cdot \abs{Z_1(t)} \cdot \sum_{n=1}^\infty \abs{\alpha_n} \lambda_n \norm{e_n}_\infty < \infty,
    \end{align*}
    where we use the assumption that $\mathcal{D}$ is of $C^1$-limiting type at the very end.
    In the same way, we obtain
    \begin{align*}
        \norm{\nabla^2 H^{(\reg)}(t,x)}_\infty &\leq \sum_{n=1}^\infty \abs{Z_n(t)} \cdot \norm{\nabla^2 e_n}_\infty 
        \leq C' \cdot \abs{Z_1(t)} \cdot \sum_{n=1}^\infty \abs{\alpha_n} \lambda_n \norm{e_n}_\infty < \infty,
    \end{align*}
    where we again use the assumption that $\mathcal{D}$ is of $C^1$-limiting type at the very end.
    These inequalities imply, in particular, that $\norm{\nabla H^{(\reg)}(t,x)}_\infty < \infty$ and $\norm{\nabla^2 H^{(\reg)}(t,x)}_\infty < \infty$ almost-surely.
    This follows directly from the fact $Z_1$ is smooth almost surely and thus $\abs{Z_1(t)} < \infty$ almost-surely for all $t \in [0,1]$.
    Set $T \coloneqq \sum_{n=1}^\infty \abs{\alpha_n} \lambda_n \norm{e_n}_\infty < \infty$.
    Then, by combining the above estimates, we obtain that 
    \begin{align*}
        \norm{D \phi_{H^{(\reg)}}^1}_\infty &\leq \exp\left( \norm{DJ}_\infty \int_0^t  \norm{\nabla H^{(\reg)}(s,x)}_\infty ds + \norm{J}_\infty \int_0^t \norm{\nabla^2 H^{(\reg)}(s,x)}_\infty ds\right) \\
        &\leq \exp\left( \norm{DJ}_\infty \cdot C \cdot T \cdot \int_0^t \abs{Z_1(s)} ds + \norm{J}_\infty \cdot C' \cdot T \cdot \int_0^t \abs{Z_1(s)} ds\right) \\
        &<\infty
    \end{align*}
    holds almost-surely.
    In the last step we use the fact that $\int_0^t \abs{Z_1(s)} ds < \infty$ almost-surely,
    since the samples of $Z_1$ are smooth (and thus integrable) almost-surely.
    Since this bound does not depend on $\reg$,
    we conclude that $\norm{D(\phi_{H^{(\reg)}}^1 (\phi_{H^{(\reg')}}^1)^{-1})}_\infty$ 
    is bounded from above almost-surely.
    Thus, 
    the random variables $\{\phi_{H^{(\reg)}}^1\}_{\reg > 0}$ converge almost-surely with respect to $d_{C^0}$.
    In particular, they converge in law.
    Note that since the Hamiltonian functions $H^{(\reg)}$ also converge almost-surely, 
    this limit lies in $\Hameo(M,\omega)$.
    We can define $\overline{\muGP}^{\mathcal{D}_0}$ as the law of the $C^0$-limit.
    Then $\overline{\muGP}^{\mathcal{D}_0}$ is a measure on $\Hameo(M,\omega)$
    and $\muGP^{\mathcal{D}_\reg} \weakto \overline{\muGP}^{\mathcal{D}_0}$ as $\reg \to 0$.
\end{proof} %
\section{Simulations}\label{sec:simulations}
We conclude this paper with an informal note on computer simulations and what we can learn from them in this setting.
It is noteworthy that the properties of the measure $\muGPH^{\mathcal{D}}$ are amenable to study through computer simulation,
which is not often the case in symplectic geometry.
In the following, 
we will present the results of some numerical simulations 
of the behavior of Hamiltonian diffeomorphisms sampled from $\muGP^{\mathcal{D}}$.
The code for these simulations is available in the ancillary files of this preprint on arXiv.
\subsection{Setup}
For these simulations we always work on $\T^2$
endowed with the standard symplectic form and the standard complex structure.
We identify $\T^2$ with $[0,1]_{/\sim}^2$ with $0 \sim 1$ on both coordinates.
We choose the following eigenbasis 
of $L^2(\T^2)$:
Let $k,j \in \mathbb{N}_0$ such that $k+j \ge 1$.
Then $4\pi^2(k^2+j^2)$ is an eigenvalue of the Laplace-Beltrami operator on $\T^2$ with multiplicity $4$ (respectively $2$ if $j = 0$ or $k =0$). 
We choose the following orthonormal basis of eigenfunctions associated to this eigenvalue:
\begin{align*}
\{
&x,y \mapsto 2\cos(2\pi k x)\cdot\cos(2\pi j y),\\
&x,y \mapsto 2\cos(2\pi k x)\cdot\sin(2\pi j y),\\
&x,y \mapsto 2\sin(2\pi k x)\cdot\cos(2\pi j y),\\
&x,y \mapsto 2\sin(2\pi k x)\cdot\sin(2\pi j y)
\}.
\end{align*}
The ordering is in this case not important, since the Gaussian processes associated
to each eigenfunction will have the same law.
Given any $\reg > 0$,
we now consider the \ldd\ $\mathcal{D}_2$ from the proof of Theorem~\ref{thm:existence-of-ldd}.
To actually run computations
we need to truncate the series expansion in~\eqref{eq:GP_H}
at some point.
We choose to only include those eigenfunctions of the above form with $k,j \leq 25$.
Recall that the coefficient processes $Z_n$ are of the form 
\begin{equation*}
    Z_n(t) = X^{(n,1)}_0 + {\sqrt{2}}\sum_{k=1}^\infty \exp\left(-2\reg\pi^2k^2\right) \cdot \left( X_k^{(n,1)} \cos(2\pi k t) + X_k^{(n,2)} \sin(2\pi k t)\right),
\end{equation*}
where $X_0,X_1^{(1)},X_1^{(2)},X_2^{(1)},X_2^{(2)},\ldots$ are i.i.d.\ standard normal random variables.
Here we truncate the series at $k = 10$.
With this computational model,
drawing samples from the distribution $\muGP^{\mathcal{D}}$ or $\muGPH^{\mathcal{D}}$ 
corresponds to drawing a $52,500$-dimensional standard Gaussian vector.
Clearly, this is a very feasible computation.
Indeed, Figure~\ref{fig:simulation-1} was obtained using this method.
We will now disucss two phenomena that can be observed in the simulations.
\subsection{Diffusion}
The first phenomenon that we observe is a form of diffusion.
For this we start with some number of points $p_1, \ldots, p_m \in \T^2$
that are concentrated in a small region of the torus.
Then we apply $\phi_H^t$ for $H$ sampled from $\muGPH^{\mathcal{D}}$ 
to all of these points.
We observe in Figure~\ref{fig:simulation-2}
that these points quickly spread evenly throughout the torus.
This is exemplary of the behavior we observe in general in the simulations.
This can be interpreted as a form of diffusion, where the initial concentration of points is rapidly smoothed out over time.
We thus formulate the following conjecture:
\begin{conjecture}
    Let $p \in M$ be arbitrary and let $ev_p: \Ham(M,\omega) \to M$ be the evaluation map at $p$.
    Then for any $\eps > 0$,
    there exists a \ldd\ $\mathcal{D}$ such that the push-forward measure $(ev_p)_* \muGP^{\mathcal{D}}$ and the uniform measure on $M$
    are $\eps$-close in a suitable metric, e.g.\ the total variation or Wasserstein distance.
\end{conjecture}
As mentioned in the introduction, this phenomenon will be studied in detail in a follow-up paper.
\begin{figure}[h]
    \begin{center}
        \includegraphics[width=0.95\textwidth]{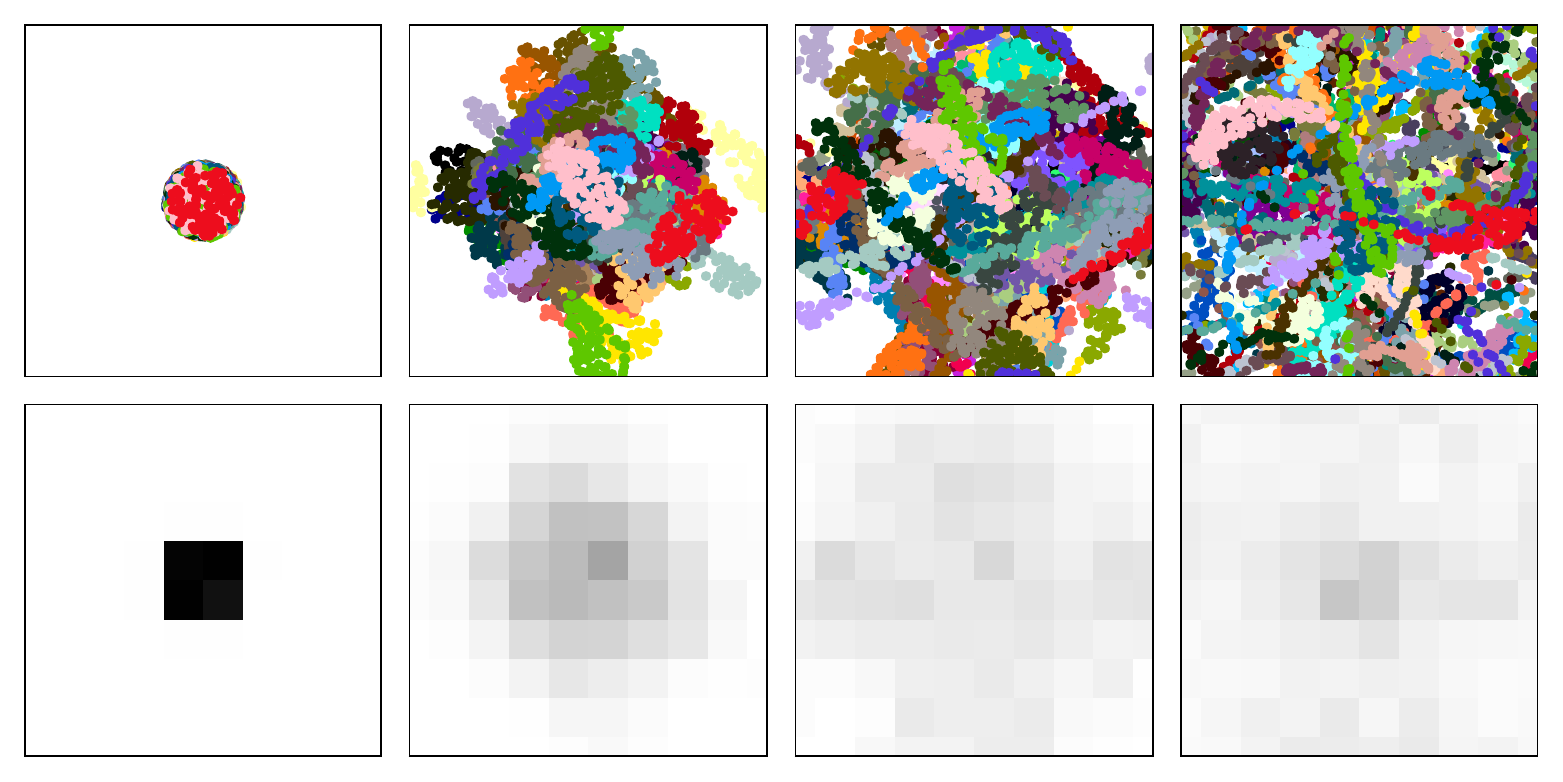}
        \caption{First row: Here we sample $100$ Hamiltonian functions from $\muGPH^{\mathcal{D}}$.
        Then their time-$t$ flows are applied to $100$ points in a radius $0.1$ ball around $(0.5,0.5)$.
        Their positions are shown at $t=0,0.05, 0.1, 0.25$ from left to right.
        The different colors represent different draws.
        Second row: The torus is divided into a $10 \times 10$ grid and the number of points (shown in the first row) in each grid cell is counted. Darker colors represent more points.
        }\label{fig:simulation-2}
    \end{center}
\end{figure}
\subsection{Expected numbers of Lagrangian intersections and Crofton-type formulas}
We conclude this section with a remark on Crofton-type formulas
and with some numerical estimates on the expected number of intersection points
of two Lagrangians under a random Hamiltonian diffeomorphism.
We first briefly introduce the setup.
Let $M$ be a closed symplectic manifold
and let $L,K \subset M$ be two closed Lagrangian submanifolds.
Recently, Dimitroglou Rizell and Evans
have studied the number of surplus intersection points between
Lagrangians, see~\cite{dimitroglou-rizell-evans-2024}.
In the context of 
this work, 
the authors study the problem of computing an 
integral of the form 
\begin{equation*}
    \frac{1}{\mu(G)}\int_G \#(L \cap gK) d\mu(g),
\end{equation*}
where $G$ is a compact Lie group acting on $M$ by 
Hamiltonian diffeomorphisms and $\mu$ is the Haar measure on $G$.
If $M = \C P^n$
and $G = PU(n+1)$,
then this integral can be computed explicitly (see~\cite{le-1993})
as follows:
\begin{equation*}
    \int_{PU(n+1)} \#(\Reals P^n \cap gK) d\mu(g) = c_n \cdot \vol(K),
\end{equation*}
where $c_n$ is a constant depending only on $n$.
This formula, in analogy to the classical Crofton formula 
in Euclidean integral geometry, is called the \textit{Crofton formula}.
In light of the measure constructed in this paper, 
one can ask whether a similar formula holds for the integral 
\begin{align*}
    \int_{\Ham(M,\omega)}
    \#(L \cap \phi( K)) d\muGP^{\mathcal{D}}(\phi)
    = \E\left[\#(L \cap \phi_H^1( K))\right],
\end{align*}
when $\mathcal{D}$
is a centered, time-symmetric, periodic, periodically exhaustive, and frequency unbiased \ldd\
and $H$ is the Gaussian process given by~\eqref{eq:GP_H}.
Due to Theorem~\ref{thm:inversion-invariance}
we have that 
$\E[\#(L \cap \phi_H^1( K))] = \E[\#(L \cap (\phi_H^1)^{-1}( K))] = \E[\#(\phi_H^1( L) \cap K)]$.
Thus, there is a certain symmetry between $L$ and $K$ in this expression
that is not encountered in the classical Crofton formula.

We investigate
this phenomenon 
by estimating 
$\E\left[\#(L \cap \phi_H^1( K))\right]$
numerically.
For this, we consider $K = S^1 \times 0.5 \subset \T^2$.
Due to the aforementioned symmetry, it suffices to
simulate the action of $\phi_H^1$ on $K$.
We can then compute the intersection number with different test Lagrangians $L$.
Figure~\ref{fig:simulation-3} shows a visualization of some deformations of $K$ under samples from $\muGP^{\mathcal{D}}$ for $\reg = 0.14$.
We then estimate $\E[\#(L \cap \phi_H^1( K))]$
for the following choices for $L$:
\begin{enumerate}
    \item The straight vertical lines $L_1 = 0.3 \times S^1, L_2 = 0.5 \times S^1, L_3 = 0.7 \times S^1$; and
    \item the sloped vertical lines $L_4 = \{(\alpha,2\alpha) \mid \alpha \in S^1\}, L_5 = \{(\alpha,3\alpha) \mid \alpha \in S^1\},L_6 = \{(\alpha,4\alpha) \mid \alpha \in S^1\}$; and
    \item the straight horizontal lines $L_7 = S^1 \times 0.3, L_8 = S^1 \times 0.5, L_9 = S^1 \times 0.7$; and
    \item the sloped horizontal lines $L_{10} = \{(2\alpha,\alpha) \mid \alpha \in S^1\}, L_{11} = \{(3\alpha,\alpha) \mid \alpha \in S^1\},L_{12} = \{(4\alpha,\alpha) \mid \alpha \in S^1\}$; and
    \item the centered circles $L_{13} = \partial B((0.5,0.5), 0.1)$ and $L_{14} = \partial B((0.5,0.5), 0.2)$.
\end{enumerate}
\begin{figure}[h]
    \begin{center}
        \includegraphics[width=0.95\textwidth]{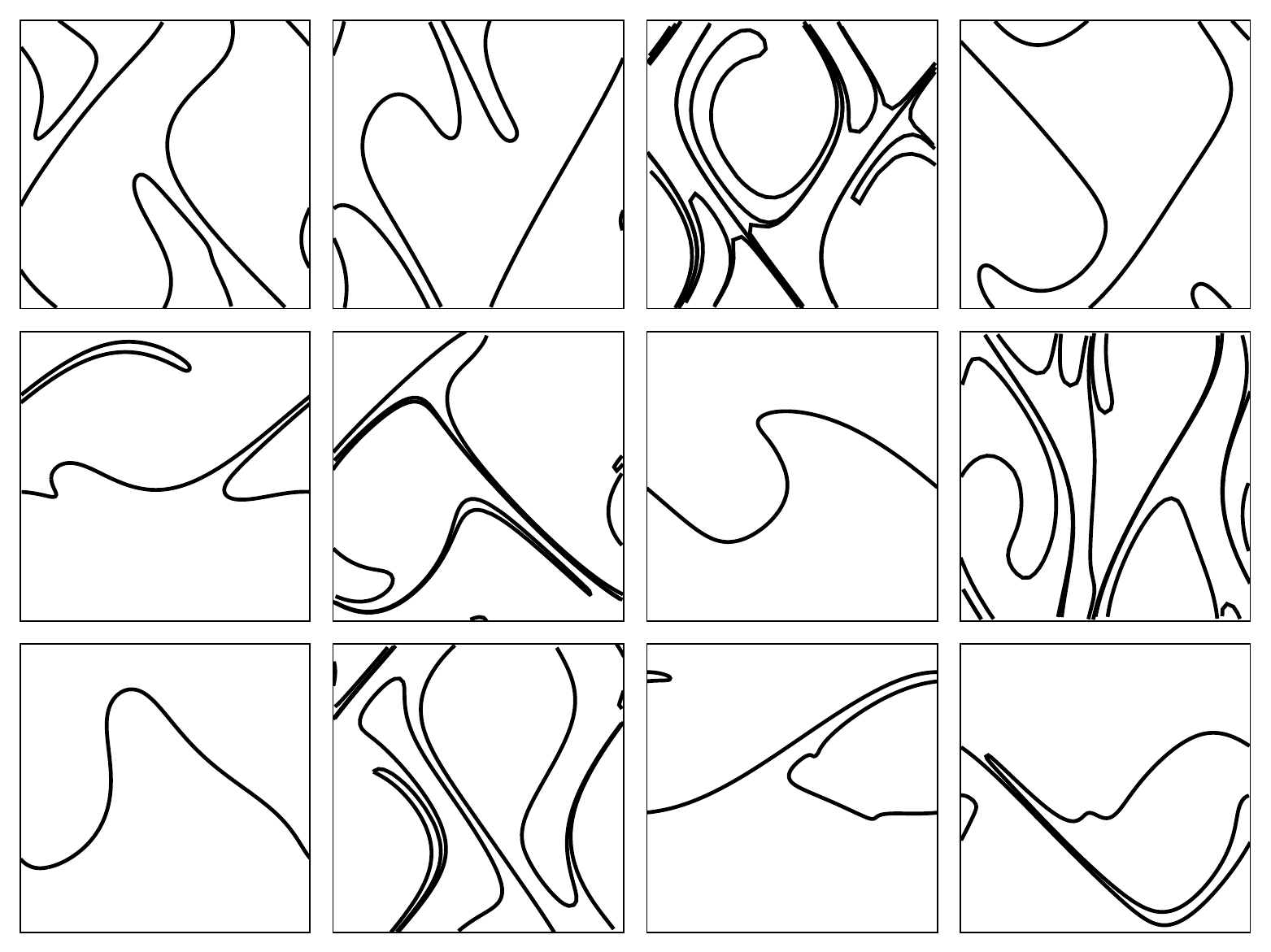}
        \caption{These curves represent $\phi(K)$ for twelve samples $\phi$ from $\muGP^{\mathcal{D}}$ with $\reg = 0.14$.}\label{fig:simulation-3}
    \end{center}
\end{figure}
The following table summarizes the results of these simulations.
It contains estimates of $\E[\#(L \cap \phi_H^1( K))]$
based on simulations for $\reg = 0.04, 0.06, 0.08, 0.1, 0.12,$ and $0.14$.
All these numbers are obtained from $600$ samples drawn from $\muGPH^{\mathcal{D}}$
for each of the values of $\reg$.
\begin{center} 
    \vspace{0.5cm}
\begin{tabular}{|c|c|c|c|c|c|c|c|c|} 
    \hline
    Lagrangian & Length & $\reg = 0.04$ & $\reg = 0.06$ & $\reg = 0.08$ & $\reg = 0.1$ & $\reg = 0.12$ & $\reg = 0.14$ \\ 
    \hline
    $L_1$    & $1$      & $50.24   $ & $26.18 $ & $16.25  $ & $10.55 $ & $6.39  $ & $ 3.45$   \\
    $L_2$    & $1$      & $50.17   $ & $26.58$  & $16.03  $ & $10.45 $ & $6.45  $ & $ 3.44$  \\
    $L_3$    & $1$      & $48.24   $ & $25.88$  & $16.7   $ & $10.54 $ & $6.54  $ & $ 3.4$   \\
    \hline                 
    $L_4$    & $2$      & $103.81  $ & $55.72$  & $35.79  $ & $23.14 $ & $14.38 $ & $ 7.45$  \\
    $L_5$    & $3$      & $145.18  $ & $78.81$  & $50.69  $ & $32.91 $ & $19.94 $ & $ 10.67$  \\
    $L_6$    & $4$      & $187.91  $ & $102.82$ & $66.01  $ & $42.52 $ & $26.17 $ & $ 14.01$  \\
    \hline                 
    $L_7$    & $1$      & $53.35   $ & $29.63$  & $19.71  $ & $12.67 $ & $8      $ & $4.12  $  \\
    $L_8$    & $1$      & $62.38   $ & $34.72$  & $23.3   $ & $15.13 $ & $9.16   $ & $4.88  $  \\
    $L_9$    & $1$      & $53.71   $ & $28.88$  & $19.61  $ & $12.96 $ & $7.84   $ & $4.1   $  \\ 
    \hline                 
    $L_{10}$ & $2$      & $ 104.09 $ & $ 57.57$ & $ 38.56 $ & $ 25.37$ & $ 15.9  $ & $ 8.3   $  \\ 
    $L_{11}$ & $3$      & $ 148.85 $ & $ 81.84$ & $ 56.13 $ & $ 37.16$ & $ 23.06 $ & $ 12.09 $  \\ 
    $L_{12}$ & $4$      & $ 193.41 $ & $ 107.12$& $ 73.33 $ & $ 48.74$ & $ 30.37 $ & $ 15.90 $  \\ 
    \hline                                                            
    $L_{13}$ & $0.63$ &  $ 42.09  $  & $22.22$  & $ 13.47$  & $ 8.52$  & $ 5$ &      $ 2.83 $ \\
    $L_{14}$ & $1.26$ &  $ 76.18  $  & $40.38$  & $ 25.28$  & $ 15.9$  & $ 9.54$   & $ 5.25 $ \\
    \hline
\end{tabular}
    \vspace{0.5cm}
\end{center}
These estimates suggest 
that a Crofton-type formula might hold
for the expected number of intersection points
of two Lagrangians under a random Hamiltonian diffeomorphism.
It is however also clear from this data that the formula does not exactly hold with the standard Riemannian metric on $\T^2$,
as the numbers above are qualitatively different for $L_7, L_8$, and $L_9$ even though these have the same length and the same homology class.
However, this is to be expected: since $L_8 = K$ and $K$ is non-displaceable,
we have $\Prob\left[\#(L_8 \cap \phi_H^1( K) = \emptyset)\right] = 0$.
However, by Theorem~\ref{thm:full-support}, 
we have that $\Prob\left[\#(L_7 \cap \phi_H^1( K) = \emptyset)\right] > 0$
and $\Prob\left[\#(L_9 \cap \phi_H^1( K) = \emptyset)\right] > 0$.
While obstructing the applicability of an exact replica of the Crofton formula,
this phenomenon might be incorporated into a Crofton-type formula by suitably adjusting the metric.
In light of the diffusion phenomenon described above,
one might expect that the correct volume is given by a rescaled version of the metric.
Thus, we make the following conjecture:
\begin{conjecture}\label{conj:crofton}
    Let $(M,\omega)$ be a closed symplectic manifold and $\mathcal{D}$ be a centered, time-symmetric, exhaustive or periodically exhaustive \ldd .
    Furthermore, let $L \subset M$ be a closed Lagrangian submanifold.
    Then there exists a constant $C > 0$ and a smooth function $\rho: M \to \Reals_{>0}$
    satisfying $\int_M \rho \omega^n = \vol(M)$
    depending only on $\mathcal{D}$ and $L$
    such that for any closed Lagrangian submanifold $K \subset M$,
    we have
    \begin{equation*}
        \int_{\Ham(M,\omega)} \#(L \cap \phi( K)) d\muGP^{\mathcal{D}}(\phi)
        = C \cdot \vol_{\rho \cdot g}(K),
    \end{equation*}
    where $g$ is the Riemannian metric associated to $\omega$ and the almost complex structure $J$ specified in $\mathcal{D}$.
    In particular, there exist constants $C', C'' > 0$ such that
    \begin{equation*}
        C' \cdot \vol_g(K) \leq \int_{\Ham(M,\omega)} \#(L \cap \phi( K)) d\muGP^{\mathcal{D}}(\phi) \leq C'' \cdot \vol_g(K)
    \end{equation*}
    for any closed Lagrangian submanifold $K \subset M$.
\end{conjecture}
One might further ask whether for any $\eps > 0$ there exists a suitable \ldd\ $\mathcal{D}$
such that $\rho$ is $\eps$-close in the $C^0$-distance (or another suitable distance) to being identically equal to $1$.
The data seems to suggest that this might happen as $\reg \to 0$,
but this question is less easy to answer numerically.
The main problem here is that as $\reg$ approaches $0$, running accurate simulations becomes more and more computationally expensive.
We will return to this question in future work. 
\bibliographystyle{alpha}
\bibliography{refs}

\end{document}